\documentclass{amsart}
\usepackage[utf8]{inputenc}
\usepackage{svg}
\usepackage[T1]{fontenc}
\usepackage[utf8]{inputenc}
\usepackage{subcaption}

\newcommand{\Fpbar}{\overline{\mathbb{F}}_p}

\usepackage{amsmath, amsfonts, amsthm, amssymb, amscd, multirow, float, mathtools,tikz-cd,comment,enumerate,bbm,stmaryrd,dsfont}

\usepackage{float}
\usepackage{bookmark}
\usepackage{url}

\DeclareMathOperator{\rank}{rank}
\usepackage{xcolor}
\usepackage{hyperref}
\hypersetup{colorlinks=true, linkcolor=blue}
\usepackage[normalem]{ulem}

\usetikzlibrary{matrix,positioning,arrows.meta}
\usepackage{mathtools}

\DeclareMathOperator{\Cls}{Cls}

\usepackage[nameinlink,capitalise]{cleveref}
\usepackage{mathrsfs}

\renewcommand{\O}{\mathcal{O}}

\DeclareMathOperator{\End}{End}
\DeclareMathOperator{\Nrd}{nrd}

\DeclareMathOperator{\discrd}{disc}

\DeclareMathOperator{\Span}{Span}
\DeclareMathOperator{\val}{val}

\newcommand{\1}{\mathds{1}}

\usepackage{commath, graphicx}

\newcommand{\mc}{\mathcal}
\newcommand{\bb}{\mathbb}

\usepackage{comment}
\theoremstyle{plain}
\newtheorem{theorem}{Theorem}[section] %

\usepackage[all]{xy} %
\xyoption{all} \xyoption{rotate}\xyoption{pdf}

\newtheorem{proposition}[theorem]{Proposition} %
\newtheorem{lemma}[theorem]{Lemma}
\newtheorem{corollary}[theorem]{Corollary}

\newtheorem{experiment}[theorem]{Experimental Observation}

\crefname{experiment}{Experimental Observation}{Experimental Observations}
\Crefname{experiment}{Experimental Observation}{Experimental Observations}

\theoremstyle{definition}
\newtheorem{definition}[theorem]{Definition}

\theoremstyle{remark}
\newtheorem{remark}[theorem]{Remark}

\AddToHook{env/proposition/begin}{\crefalias{theorem}{proposition}}
\AddToHook{env/lemma/begin}{\crefalias{theorem}{lemma}}
\AddToHook{env/corollary/begin}{\crefalias{theorem}{corollary}}
\AddToHook{env/conjecture/begin}{\crefalias{theorem}{conjecture}}
\AddToHook{env/experiment/begin}{\crefalias{theorem}{experiment}}
\AddToHook{env/definition/begin}{\crefalias{theorem}{definition}}
\AddToHook{env/remark/begin}{\crefalias{theorem}{remark}}
\AddToHook{env/question/begin}{\crefalias{theorem}{question}}
\AddToHook{env/example/begin}{\crefalias{theorem}{example}}

\newcommand{\mf}{\mathfrak}

\newcommand{\QQ}{\mathbb{Q}}
\newcommand{\NN}{\mathbb{N}}
\newcommand{\ZZ}{\mathbb{Z}}
\newcommand{\FF}{\mathbb{F}}
\newcommand{\CC}{\mathbb{C}}
\newcommand{\RR}{\mathbb{R}}
\newcommand{\sus}{\subset}
\newcommand{\repnum}[2]{a_{#1}(#2)}
\newcommand{\qprime}{r}
\newcommand{\T}{\bb{T}}
\newcommand{\Tan}{\bb{T}^{\mathrm{an}}}

\usepackage{booktabs}

\newcommand{\Tr}{\mathrm{trd}}

\newcommand{\nrd}{\mathrm{nrd}}
\newcommand{\trd}{\mathrm{trd}}
\DeclareMathOperator{\disc}{disc}
\DeclareMathOperator{\SL}{SL}
\DeclareMathOperator{\GL}{GL}
\DeclareMathOperator{\A}{\mathcal{A}}

\let\dotlessi\i
\renewcommand{\i}{\mathfrak{i}}
\renewcommand{\j}{\mathfrak{j}}
\renewcommand{\setminus}{\smallsetminus}

\title[Generating sets for maximal quaternion orders]{Generating sets for maximal orders in rational quaternion algebras} %
\date{October 7, 2026}

\author[Eisentr\"ager]{Kirsten Eisentr\"ager}
\address{Kirsten Eisentr\"ager, Department of Mathematics, The Pennsylvania State University, 109 McAllister Building, University Park, PA 16802, USA. {\tt{eisentra@math.psu.edu}}}

\author[Goren]{Eyal Z.~Goren}
\address{Eyal Z.~Goren, Department of Mathematics and Statistics, McGill University, 805 Sherbrooke St.\ W, Montréal, QC H3A0B9, Canada. \tt{eyal.goren@mcgill.ca}}

\author[Iezzi]{Annamaria Iezzi}
\address{Annamaria Iezzi, Univ. Grenoble Alpes, CNRS, Grenoble INP, LJK, 38000 Grenoble, France. \tt{annamaria.iezzi@univ-grenoble-alpes.fr}}

\author[Kir]{Harun Kir}
\address{Harun Kir, UMPA, ENS de Lyon, CNRS, UMR 5669, 46 all\'ee d'Italie, 69364 Lyon Cedex 07, France. \tt{harun.kir@ens-lyon.fr}}

\author[K{\dotlessi}r{\dotlessi}ml{\dotlessi}]{Eda K{\dotlessi}r{\dotlessi}ml{\dotlessi}}
\address{Eda Kirimli, University of Birmingham, University Road W, Edgbaston, Birmingham, B15 2TT, UK. \tt{e.kirimli@bham.ac.uk}}

\author[Love]{Jonathan R.~Love}
\address{Jonathan R.~Love, Department of Mathematics, University of Michigan, 530 Church Street, Ann Arbor, MI 48109, USA. \tt{jonlove@umich.edu}}

\author[Mahaney]{William E.~Mahaney}
\address{William E.~Mahaney, Department of Mathematics, Virginia Tech, 225 Stanger Street, Blacksburg, Virginia 24060, USA. \tt{mahaney1@vt.edu}}

\author[Park]{Jennifer Park}
\address{Jennifer Park, Department of Mathematics, The Ohio State University, 231 West 18th Avenue, Columbus, OH 43210, USA. \tt{park.2720@osu.edu}}

\author[Sabitova]{Maria Sabitova}
\address{Maria Sabitova, Department of Mathematics, Queens College, CUNY, 65-30 Kissena Blvd, Flushing, NY 11367, USA. \tt{maria.sabitova@qc.cuny.edu}}
 
\begin{document}

\begin{abstract}
    Given a maximal order $\mf{O}$ in a rational definite quaternion algebra, and a prime $\ell$ coprime to the discriminant of $\mf{O}$, this paper considers subsets of $\mf{O}$ consisting of elements with $\ell$-power norms that together generate $\mf{O}$ as a $\ZZ$-algebra. We prove two theorems about the existence of such sets: the first states that $\mf{O}$ is generated by elements of norm $\ell^k$ for any $k$ larger than an explicit bound, and the second states that there is a generating set for $\mf{O}$ consisting of at most three elements, each with norm a power of $\ell$. We discuss implications for the study of supersingular isogeny graphs. As steps towards these theorems, for quaternion orders $\O$ that are not necessarily maximal, we also prove structural results about the algebra of Brandt matrices for $\O$ and explicit bounds on the coefficients of the theta function of $\O$. Computational experiments are also discussed.
\end{abstract}

\maketitle
\markboth{\MakeUppercase{Eisentr\"ager, Goren, Iezzi, Kir, K{\dotlessi}r{\dotlessi}ml{\dotlessi}, Love, Mahaney, Park, Sabitova}}{\shorttitle}

\small{
\tableofcontents
}

\section{Introduction}
The principal results of this paper concern generating sets for certain orders in definite quaternion algebras, where essential use is made of the theta functions of these orders. %
As we shall explain, the project is closely related to questions of current interest in mathematical cryptography, but it equally has its roots in several other areas of mathematics: lattices and their generating sets, counting representations of integers by quadratic forms, quaternion algebras, and the theory of Brandt modules, which allows us to relate the ideal theory of quaternion orders to modular forms. This paper contributes to all of these areas. 

For ease of exposition, we will focus in the introduction on the case that $\mc{A}_p$ is a definite quaternion algebra over $\mathbb{Q}$ ramified only at $p$ and $\infty$, and $\mf{O}$ is a maximal order in $\mc{A}_p$; such orders arise as endomorphism rings of supersingular elliptic curves over an algebraic closure $\overline{\FF}_p$ of a finite field $\FF_p$.
Fix a prime $\ell\neq p$. In \cite[Theorem 1.1]{Goren_Love_Elts_Prescribed_Norm} it is proven that there exists a subset $C\subseteq \mf{O}$ such that the norm of each $x\in C$ is a power of $\ell$ and~$C$ generates $\mf{O}$ as a $\ZZ$-algebra:\footnote{In fact the cited result proves that $C$ generates $\mf{O}$ as a $\ZZ$-module, but in this paper we focus primarily on generating sets of $\ZZ$-algebras.} that is, every element of $\mf{O}$ can be written as a $\ZZ$-linear combination of products of elements in $C$, including the empty product $1\in\mf{O}$.
In this paper we prove effective bounds on these generating sets. Our main result is a bound on the norms of the generators.
\begin{theorem}\label{thm: explicit p and k}
    Let $p$ and $\ell$ be distinct primes, and $k$ a positive integer with $\ell^{k/2-1}>3pk$.
    For any maximal order $\mf{O}\sus\mc{A}_p$, the elements of norm $\ell^k$ in $\mf{O}$ generate $\mf{O}$ as a $\ZZ$-algebra. 
\end{theorem}

For example, these assumptions hold if $p>\ell^{17}$ and $k>3 \log_{\ell}(p)$.
A more general version of this result --- applicable to maximal orders in any definite quaternion algebra over $\QQ$ --- is proven in \cref{thm:generation bound}.

Along with bounding the norms of elements of the generating set $C$, we also prove a sharp bound on the number of elements in $C$. 
\begin{theorem}\label{prop:intro number of generators}
    For any maximal order $\mf{O}$ in $\mc{A}_p$, there exists a set of at most three elements, each with norm in $\{\ell^k:k\geq 1\}$, that generates $\mf{O}$ as a $\ZZ$-algebra.
\end{theorem}
This is a special case of \cref{prop:number of generators}, which generalizes this result to arbitrary definite quaternion algebras $\mc{A}$ over $\QQ$ and replaces the set of allowable norms $\{\ell^k:k\geq 1\}$ with any infinite set of positive integers coprime to the discriminant of $\mc{A}$.

\subsection{Motivation: cryptography and otherwise}\label{subsec:motivation}
A primary motivation for these results arises from the perspective of cryptographic protocols based on supersingular isogenies and the associated isogeny graphs. These directed graphs are constructed as follows: one fixes two distinct primes $\ell \neq p$. The vertices of the graphs are supersingular $j$-invariants in characteristic $p$, of which there are about $p/12$, all of which are in $\FF_{p^2}$. Outgoing edges from a vertex $j_0$ correspond to isogenies of degree $\ell$ from a preselected elliptic curve with $j$-invariant $j_0$, considering two isogenies as the same if they have the same kernel.
This is thus a directed graph, which is $(\ell+1)$-regular for outgoing edges. %
This defines the supersingular $\ell$-isogeny graph $G(p, \ell)$. %
These graphs and their various generalizations have been studied intensively in recent years, largely motivated by their cryptographic applications (\cite{Goren_Love_SSEC_Survey} contains references to the relevant literature).
However, supersingular graphs were studied prior to their use in cryptography, notably by Mestre \cite{Mestre} in the context of calculating spaces of modular forms and Pizer \cite{PizerBAMS} in the context of Ramanujan graphs.

\medskip

In \cite{CGL}, Charles, Goren, and Lauter introduced a cryptographic hash function (henceforth the ``CGL hash function'') based on supersingular isogeny graphs, which was later suggested for use in post-quantum cryptography. %
While the jury is still out on the general applicability of this construction, %
there is no doubt that a host of fruitful mathematical questions arose, and many interesting developments concerning elliptic curves, complex multiplication and quaternion algebras have been made in that context. This paper contributes further to this growing field.

The basic construction of the CGL hash function %
is as follows: the input message, a string $M$ of 0's and 1's of large length, is to be converted
into a shorter string, the hash value. One chooses a large prime $p$ and a small prime $\ell$; for instance $\ell = 2$ and (currently) $p\approx 2^{500}$ are good choices. The original string $M$ is converted in some deterministic fashion into directions for walking the supersingular isogeny graph $G(p, \ell)$ starting at some publicly known initial vertex $j_0$ that is part of the protocol. The hash value is the final vertex $j$ the walk reaches, which in turn can be encoded in some fashion as a string of $0$'s and $1$'s (of length about $1000$, for the parameters mentioned above). 

\medskip

The key underlying hardness assumption for the above hash function and other cryptographic constructions based on supersingular isogenies is the computational difficulty of finding isogenies between supersingular  elliptic curves, often referred to as the {\em Isogeny Problem.} 
This problem was shown to be equivalent, under polynomial time reductions, to the \textit{Supersingular Endomorphism Ring Problem} \cite{HLMW2026}, building on earlier work \cite{KLPT,Wesolowski}. Given a prime $p$ and a supersingular elliptic curve $E$ defined over
$\mathbb{F}_{p^2}$, the Supersingular Endomorphism Ring Problem asks to find four endomorphisms of $E$ in an efficient representation that generate the endomorphism ring of $E$ as a lattice. 
Each cycle in a supersingular $\ell$-isogeny graph through a vertex $j_0$ produces an endomorphism of the associated elliptic curve $E_0$ by composing the isogenies associated to each edge in the cycle. Modulo certain details concerning curves with extra automorphisms, one can build a dictionary between cycles in $G(p,\ell)$ and endomorphisms of $\ell$-power degree. 
 In particular, one can ask whether such cycles generate the endomorphism ring of $E_0$, and if so, how many of them are needed. 

The Supersingular Endomorphism Ring Problem is also equivalent to that of determining, given $E$ and $p$ as above, the maximal order $\mf{O}$ %
in the rational quaternion algebra $\mc{A}_p$ ramified only at $p$ and $\infty$ which is isomorphic to the endomorphism ring of $E$. %
These different problems related to isogenies and the relevant literature on them are discussed in~\cite{CDGKRW}.
The computation of endomorphism rings was first studied by Kohel in his thesis \cite{KohelThesis}, where he proposed an approach based on finding cycles in $G(p,\ell)$. Later work determined necessary and sufficient conditions for two such cycles to generate a finite-index subring of $\mf{O}$ \cite{BCNEMP}. There are algorithms that use such a finite-index subring to compute the entirety of $\mf{O}$ (see e.g.~\cite{EHLMP20}), but the output of these algorithms is typically not a description of $\mf{O}$ purely in terms of cycles of $G(p,\ell)$.
In \cite{Goren_Love_Elts_Prescribed_Norm} it was shown that such cycles do generate the entirety of $\mf{O}$, but the results were not effective: no bound on the number of cycles required or the lengths of such cycles was given. \Cref{thm: explicit p and k} and \cref{prop:intro number of generators} give explicit bounds.

More broadly, one is led to a natural question in the theory of lattices: given an integral lattice $\Lambda$ 
and a set $S$ of positive integers, is $\Lambda$ generated (as a $\ZZ$-module) by vectors whose norms lie in $S$? General results concerning this can be found in \cite{Goren_Love_Elts_Prescribed_Norm, Goren_Love_SSEC_Survey}. To approach this problem one may consider the sublattice of $\Lambda$ spanned by elements with norm in $S$, and ask whether it equals $\Lambda$. This in turn leads to a natural follow-up question. Given a pair of integral lattices $\Lambda_1 \supsetneq \Lambda_2$ 
of the same rank with associated theta functions $\theta_{\Lambda_1}, \theta_{\Lambda_2}$, their difference $\theta_{\Lambda_1}-\theta_{\Lambda_2}$ is a nonzero modular form of a certain level and weight determined effectively by properties of the lattices with non-negative Fourier coefficients. One can then ask, %
how large are the coefficients of the terms of this form with exponent in $S$? %
While this question may be posed for arbitrary lattices, our results in this direction rely on the multiplicative structure of an order in a quaternion algebra, a property general lattices lack.
In the case of a singleton $S=\{\ell^k\}$ for an appropriate value of $k$, we prove that the associated coefficient is positive when $\Lambda_1$ is a maximal order in a rational definite quaternion algebra and $\Lambda_2$ is any suborder (\cref{prop: intro max submax upper lower} below); this is the key ingredient in the proof of \cref{thm: explicit p and k}.

\subsection{Discussion of results} 

We now turn to a more in depth discussion of the results of this paper. Let $\mc{A}$ be any definite quaternion algebra over $\QQ$, and $\mf{O}$ a maximal order in $\mc{A}$. Our aim is to prove effective results on generating sets of $\mf{O}$ consisting of elements with norm a power of $\ell$.

\subsubsection{Bounding norms}

Our main result is the existence of a generating set $C$ for $\mf{O}$ with elements whose norm is an explicitly bounded power of a prime $\ell$ (\cref{thm:generation bound}). To obtain this, we prove explicit bounds on $\repnum{n}{\O}$, the number of elements of norm $n$ in an order
$\O\subseteq\mc{A}$. Let $\phi$ denote Euler's totient function. %
\begin{proposition}[cf.~\cref{prop: max submax upper lower}]\label{prop: intro max submax upper lower}
    Let $\O\sus\mc{A}$ be an order,  $\ell\nmid\disc(\mc{A})$ a prime, and $k\geq 1$ be an integer. If $\O$ is maximal then
    \[\repnum{\ell^k}{\O}\geq \frac{24}{\phi(\disc(\mc{A}))}\frac{\ell^{k+1}-1}{\ell-1}-12k\ell^{(k+2)/2},\]
    and if $\O$ is not maximal then 
    \[\repnum{\ell^k}{\O}\leq \frac{16}{\phi(\disc(\mc{A}))}\frac{\ell^{k+1}-1}{\ell-1}+12k\ell^{(k+2)/2}.\]
\end{proposition}
\cref{thm:generation bound} (and \cref{thm: explicit p and k}) follows fairly directly from these bounds. Namely, let $\ZZ[C]$ denote the $\ZZ$-algebra generated by the set $C$ consisting of all the elements of $\mf{O}$ of norm $\ell^k$. %
The assumptions on $k$ ensure that $\mf{O}$ contains strictly more elements of norm $\ell^k$ than any non-maximal order, so the order $\ZZ[C]$ must be maximal, that is, $\ZZ[C]=\mf{O}$.

We have thus reduced to a question about bounding $a_n(\O)$, the number of elements of norm $n$ in an order $\O\sus\mc{A}$. To motivate our approach, we first describe a classical method for counting the number of elements of each norm in an integral lattice, and then explain where the approach breaks down and the modifications we make. We begin by considering the \emph{theta function} of $\O$,
\[\theta_\O(q):=\sum_{n=0}^\infty a_n(\O)q^n.\]
If we evaluate this expression at $q=e^{2\pi i z}$ for $z$ in the upper half plane, the result is a modular form of weight $2$ and level $\Gamma_0(N)$ %
for some positive integer $N$ which depends only on the structure of $\O$ as a lattice. Classical results about modular forms tell us that there exists a decomposition
\[\theta_\O(q)=E(q)+\sum_{i=1}^m c_if_i(q^{d_i})\]
where $E$ is an \emph{Eisenstein series}, and for each $i=1,\ldots,m$, $f_i$ is a \emph{newform} of some level $M_i$, $M_i\mid N$, $d_i\mid (N/M_i)$, and $c_i$ is a constant. We can therefore compute $a_n(\O)$ by computing the coefficients of each modular form on the right. Eisenstein series have explicitly computable coefficients, and while the coefficients of newforms are more mysterious, their growth is precisely controlled by the \emph{Deligne bound}. This approach allows us to write $a_n(\O)$ as the sum of a main term, namely the coefficient of $q^n$ in $E(q)$, and an error term coming from the newforms. When we restrict to a set of positive integers $n$ satisfying certain mild constraints (for example, being coprime to $N$ 
is sufficient), we can then prove that the main term dominates the error term for sufficiently large $n$, giving us asymptotic formulas for $a_n(\O)$.

The downside to this approach is that a priori, we have no control over the coefficients $c_i$, which means that it is impossible to obtain effective bounds using this approach; this problem arises because we only used the fact that $\theta_\O$ lies \emph{somewhere} in the vector space of modular forms with given weight and level. Instead we consider a slightly modified approach, which takes advantage of the algebraic structure of quaternion orders in order to show that $\theta_\O$ lies fairly close to the origin in this space of modular forms.

To do this we define the \emph{Brandt module} $\mf{M}(\O)$ of $\O$, which is a vector space with basis vectors $v_1,\ldots,v_h$ corresponding to invertible right-ideal classes $[I_1],\ldots,[I_h]$ of $\O$; we typically order the basis so that $[I_1]$ is the principal right-ideal class. The Brandt module comes equipped with \emph{Brandt matrices} $B_n$ for all $n\geq 0$, which are linear maps on $\mf{M}(\O)$ that measure relative containment of ideal classes (\cref{subsec:brandt}). The structure of the Brandt module together with its Brandt matrices may be thought of as a noncommutative analogue of the ideal class group of an order in a number field. In the case that $\O$ is a maximal order in $\mc{A}_p$, the Brandt matrix $B_\ell$ for prime $\ell\neq p$ is an adjacency matrix for $G(p,\ell)$. For more information see for example \cite{GrossCMS}. 
An important result we need about the Brandt algebra is the following.
\begin{theorem}[\cref{thm:Brandt algebra commutative}]\label{thm:intro Brandt algebra commutative}
    Let $\O$ be an order in a quaternion algebra over $\QQ$. The Brandt matrices $B_n$ for all $n\geq 0$ pairwise commute.
\end{theorem}

\begin{remark}
    We could not find this result in the literature; most sources only prove this for maximal orders $\O$, or for operators $B_n$ with $n$ coprime to the discriminant of $\O$. We note that the behavior of Brandt matrices $B_\ell$ for primes $\ell$ dividing the discriminant can be quite subtle. Many existing approaches to proving results of this form do not apply in the cases we need: see \cref{ex:brandt commutative is tricky}. 
\end{remark}

We then define a symmetric bilinear pairing $\Theta$ on $\mf{M}(\O)$ valued in the space $\mc{M}_2(\Gamma_0(N))$ of modular forms of weight $2$ and level $\Gamma_0(N)$. %
This map generalizes the notion of theta function to the entire Brandt module: $2\Theta(v_i,v_j)$ for each $i,j$ is the theta function of some lattice in $\mc{A}$, with $2\Theta(v_1,v_1)=\theta_\O$ (\cref{lem:Theta pairing}). 
Under this map, simultaneous eigenvectors for all Brandt operators are sent to modular forms that are eigenvectors for all Hecke operators $T_n$ with $n$ coprime to $N$ (cf. \cref{lem:Brandt_Hecke_equal_spectra}). We can therefore prove the following decomposition, which should be thought of as a modification of the classical decomposition into Eisenstein series and newforms.
\begin{theorem}[\cref{thm:eigenform_coeffs}]\label{thm:intro eigenform_coeffs}
    Let $\O$ be an order in a definite quaternion algebra over~$\QQ$, and let $h$ be the class number of $\O$. For some positive integer $N$, %
    there exist modular forms $f_1,\ldots,f_h\in \mathcal{M}_2(\Gamma_0(N))$ such that
    \begin{itemize}
        \item each $f_i$ is an eigenvector for the Hecke operator $T_n$ for all $n$ coprime to $N$,
        \item the coefficient of $q$ in each $f_i$ is equal to $1$, and
        \item there exist $c_1,\ldots,c_h\geq 0$ with $c_1+\cdots+c_h=|\mathcal{O}^\times|$ and
        \[
        \theta_\O(q)=c_1f_1(q)+\cdots+c_hf_h(q).
        \]
    \end{itemize}
\end{theorem}
The new key input here is the non-negativity of the coefficients $c_i$, from which we can conclude that the cuspidal contribution to this sum is effectively bounded. 

Unfortunately, because the forms $f_i$ occurring in this decomposition are not necessarily Eisenstein series or newforms, we do not have as precise control over their $q$-expansions. 
To handle this, we restrict our attention to maximal orders as well as \emph{submaximal orders}, those orders $\O$ that are contained in a maximal order $\mf{O}$ with no other order strictly between the two. We classify the possible structures of submaximal orders (\cref{lem:max_suborders}) and then we consider the possible $q$-expansions that the forms $f_i$ from \cref{thm:intro eigenform_coeffs} may take in each case. This part of the proof is by far the most technical, but the end result is the recovery of a similar decomposition as the one found in the classical case: some forms are ``Eisenstein-like'' in that their coefficients are explicitly computable (\cref{sec:eisenstein part}), and others are ``newform-like'' in that their coefficients can be effectively bounded (\cref{sec:cusp part}). Putting these results together yields explicit lower bounds on $a_n(\mf{O})$ when $\mf{O}$ is maximal, and explicit upper bounds on $a_n(\O)$ when $\O$ is submaximal; because any non-maximal order is contained in a submaximal order, this upper bound then applies to arbitrary non-maximal orders, which gives the desired bounds as in \cref{prop: intro max submax upper lower}.

\subsubsection{Bounding number of generators}
\;
Finally, we turn to bounding the number of elements in a generating set, namely \cref{prop:number of generators}. For this proof we largely follow the methods developed in \cite{Goren_Love_SSEC_Survey}, namely the use of a local-global principle for lattice cosets \cite[Theorem 6.6.1]{Goren_Love_SSEC_Survey}. A direct application of this local-global principle can easily produce a $\ZZ$-algebra generating set $C$ for $\mf{O}$ with $|C|\leq 6$. To get the bound down to three (the best possible, see \cref{rmk:number gens optimal}) requires one key new idea, which is to produce two elements $\alpha,\beta\in\mf{O}$ which generate quadratic orders with coprime discriminants; after this, the third element can be selected so that locally, prime by prime, it fills out the rest of~$\mf{O}$. The main tool that allows this process to work is \cref{lem:extend basis}, which says that given a local order $\mf{O}_\ell$ and an element $x\in\mf{O}_\ell$ such that $1,x$ can be completed to a basis of $\mf{O}_\ell$, it is in fact possible to complete it to a basis of the form $1,x,y,xy$ where $y$ has (almost) any desired norm. We note also that the proof shows that $\mf{O}$ is generated as a $\ZZ$-module by at most seven elements; see the discussion following \cref{prop:number of generators}.

\subsection{Outline}
The proof of \cref{thm:generation bound} (of which \cref{thm: explicit p and k} is a special case) is the focus of  \cref{subsec:quaternion orders}--\cref{sec:final counts}. In \cref{subsec:quaternion orders} we introduce the tools needed from the theory of quadratic forms, lattices, and quaternion orders, building up to the definition of Brandt matrices. \cref{app:commutativity} proves commutativity of the Brandt algebra. \cref{subsec:submaximal} classifies submaximal orders, and then proves special properties about the Brandt algebras of specific types of submaximal orders.

In \cref{sec:modular forms} we turn to the theory of modular forms, including an introduction of the theta operator and its properties. \cref{sec:map to M2} proves \cref{thm:intro eigenform_coeffs} and studies the Fourier coefficients of modular forms coming from simultaneous eigenvectors of Brandt matrices of maximal or submaximal orders. The above work culminates in \cref{sec:final counts}, which proves bounds on the number of elements of norm $\ell^k$ in a quaternion order and uses this to prove \cref{thm:generation bound}.

We note that \cref{sec:brandt algebra relations} and \cref{sec:cusp part}  are devoted to careful analysis of special cases: these sections investigate Brandt algebras of specific submaximal orders and the $q$-expansions of associated theta functions. The reader aiming to understand the big picture can come back to these sections as needed. %

In \cref{sec: generating set size} we prove \cref{prop:number of generators}. In \cref{sec:experiment} we report on some computational experiments.

\subsection{Acknowledgments}

This project began at the Banff International Research Station (BIRS) Workshop 25w5500, \emph{Isogeny Graphs in Cryptography}. The authors are very grateful to the organizers of this workshop: Travis Morrison, Victoria de Quehen, Jason LeGrow, and Chloe Martindale. The project was led by Jonathan Love, who was supported by ERC Starting Grant 101076941 (`\textsc{GAGARIN}'). Eisentr\"ager was partially supported by National Science Foundation award CNS-2521084. Goren was partially supported by NSERC grant 223148. Iezzi was partially supported by the French National Research Agency (ANR) under the "France 2030” program (ANR-15-IDEX-02), within the LabEx PERSYVAL, and under the grant ANR-20-CE40-0013. Kir was supported by the Agence Nationale de la Recherche under grants ANR-22-PETQ-0008 (PQ-TLS). Kirimli was partially supported by EPSRC through grant
number EP/V011324/1. Mahaney was supported by the Julian Chin Fellowship from the National Security Institute at Virginia Tech. Park was partially supported by DMS-2152182. Sabitova was supported in part by PSC-CUNY Award No. TRADB-56-58, jointly funded by The Professional Staff Congress and The City University of New York.   %

\subsection{AI disclosure}
An initial complete draft of the paper was written without AI usage, with the exception of the use of VSCode autocomplete for writing code to run computational experiments. Before releasing the paper, we uploaded the manuscript to Claude and ChatGPT for a final round of proofreading. In addition to some spelling and typesetting corrections, this process identified a logical gap in the proof of \cref{prop:Eu eigenvector}; we corrected this gap ourselves.

\subsection{Notation}

Here we collect any notation that is used in more than one section, together with either our running definition or convention, or a reference to the place where it is defined.
\begin{itemize}
    \item $\ell$ is a prime number
    \item $\FF_{\ell}$ is a finite field with $\ell$ elements
    \item $\overline{\FF}_\ell$ is an algebraic closure  of $\FF_{\ell}$
    \item $\QQ_\ell$ is the field of $\ell$-adic numbers, with valuation ring $\ZZ_\ell$
    \item $\QQ_{\ell^2}$ is the (unique up to isomorphism) quadratic unramified extension of $\QQ_\ell$, with valuation ring $\ZZ_{\ell^2}$
    \item $\mc{A}$ is a definite quaternion algebra over $\QQ$
    \item $\Lambda$ is a lattice in $\mc{A}$; we write $\O$ for orders, $I,J$ for invertible right $\O$-ideals, and $\mf{O}$ for maximal orders.
    \item $p$ is a prime number dividing $\disc(\mc{A})$ (\cref{subsec:quaternion lattices}) 
    \item $r$ is a prime number dividing $\disc(\O)$ (\cref{subsec:quaternion lattices}) but not dividing $\disc(\mc{A})$
    \item $\mc{A}_\ell:=\mc{A}\otimes_\QQ \QQ_\ell$ and $\Lambda_\ell:=\Lambda\otimes_{\ZZ}\ZZ_\ell$ are completions at $\ell$. One mild abuse of notation is that we also use subscripts $I_1,I_2,\ldots$ for lists of invertible right $\O$-ideals, in which case we write $(I_i)_\ell$ for the completion. We avoid using both unindexed and indexed ideals in the same section, so the meaning of $I_2$, for instance, should be clear from context.
    \item $h$ is the class number of $\O$ (\cref{subsec:ideals})
    \item For $A\in M_n(R)$ (the ring of $n\times n$ matrices with entries in the commutative ring $R$), $A(i,j)$ is the entry of $A$ in row $i$ and column $j$
    \item $q=e^{2\pi iz}$ for $z$ in the upper half plane of $\CC$ (\cref{eq:q expansion})
    \item $\langle\cdot\,,\cdot\rangle$: \cref{eq:trace form on A} if the inputs are in $\mc{A}$, \cref{eq:brandt inner product} if the inputs are in $\mf{M}(\O)$
    \item $\lambda_n(u)$: the eigenvalue of an eigenvector $u\in \mf{M}(\O)$ under the action of $B_n$
    \item $u_1,\ldots,u_h\in \mf{M}(\O)$: a normalized Brandt eigenbasis (start of \cref{sec:map to M2})
    \item $S$-eigenvector: an eigenvector for every linear operator in $S$
    \item $N_{K/F}$ is the field norm of a finite field extension $K/F$ 
    \item $\operatorname{disc}(x)=\operatorname{trd}(x)^2-4\operatorname{nrd}(x)$, $x\in\mc{A}$
    \item \cref{subsec:motivation}: $G(p,\ell)$
    \item \cref{subsec:lattices}: $\Lambda^\sharp$, $\Span_R(C)$
    \item \cref{subsec:quaternion lattices}: $\overline{\alpha}, \; \nrd(\alpha), \; \trd(\alpha), \; \nrd(\Lambda), \; \langle\alpha,\beta\rangle_\Lambda$
    \item \cref{subsec:orders}: $R[C]$
    \item \cref{subsec:ideals}: $\O_L(I), \; \O_R(I), \; [I], \; \mathrm{Gen}(I),\; J\subseteq_n I$
    \item \cref{subsec:brandt}: $I_1=\O,I_2,\ldots,I_h, \; \mf{M}(\O), \; v_1,\ldots,v_h, \; B_n, \; \mf{B}(\O)$
    \item \cref{ss: red_loc}: $\mf{b}_n(i,j), \; \mf{N}_n(J,I)$
    \item \cref{ss:local}: $\mf{N}_{\ell^n}(x,\O_\ell), \;  N_{\ell^n}(x)$
    \item \cref{sec:genus subspace}: $\mf{M}_{\mathrm{G}}(\O), \; e, \; e'$
    \item \cref{sec:brandt algebra relations}: $W_{\mf{r}}$, $E$
    \item \cref{subsec:modular forms}: $\mc{M}_2(\Gamma_0(N)), \; \mc{S}_2(\Gamma_0(N)), \; \mc{S}_2^{\mathrm{new}}(\Gamma_0(N)), \; T_n, \; \T, \; \Tan$
    \item \cref{subsec:theta functions}: $\theta_\O, \; a_n(\O), \; \Theta$
    \item \cref{cor:max submax eigenform_coeffs}: $g$ 
    \item \cref{cor: submaximal eisenstein bound}: $g_n(\O)$
\end{itemize}

\section{Background: quaternions}\label{subsec:quaternion orders}

For more information on the topics in this section, the authors refer the reader to \cite{Voight}. Throughout we use $F$ to refer to a field of characteristic $0$, and $R\subseteq F$ a principal ideal domain with field of fractions $F$. We will specifically consider $F=\QQ$ and $R=\ZZ$, or $F=\QQ_\ell$ and $R=\ZZ_\ell$ for some prime $\ell$, or $F=R=\bb{R}$.

\subsection{Lattices}\label{subsec:lattices}

Let $V\simeq F^n$ be a finite-dimensional $F$-vector space. An $R$-submodule $\Lambda \subseteq V$ of rank~$n$ is called a \textbf{lattice over $R$} (or an ``$R$-lattice''). In this case we have $\Lambda \otimes_R F  = V$. In the case $R=\ZZ$ we simply call $\Lambda$ a ``lattice.'' Given a subset $C\subseteq \Lambda$, we write $\Span_R(C)$ for the $R$-lattice generated by $C$; that is, the minimal $R$-sublattice of $\Lambda$ that contains $C$.

Let $\langle\cdot\,,\cdot\rangle\colon \Lambda\times\Lambda\to F$ be a nondegenerate symmetric bilinear form on $\Lambda$ (which extends uniquely to a bilinear form on $V$ by tensoring with $F$).
 We say that $\Lambda$ is \textbf{integral} with respect to $\langle\cdot\,,\cdot\rangle$ if $\langle x,y\rangle\in R$ for all $x,y\in\Lambda$, and that $\Lambda$ is \textbf{even integral} with respect to $\langle\cdot\,,\cdot\rangle$ if $\langle x,x\rangle\in 2R$ for all $x\in\Lambda$. Note that any even integral lattice is integral because $2\langle x,y\rangle=\langle x+y,x+y\rangle-\langle x,x\rangle-\langle y,y\rangle$. If $\Lambda$ and $\Lambda'$ are $R$-lattices with respective bilinear forms $\langle \cdot\,,\cdot\rangle$ and $\langle \cdot\,,\cdot\rangle'$, we say a map $\phi:\Lambda\to\Lambda'$ is an \textbf{isometry} if it is a bijection with $\langle \phi(x),\phi(y)\rangle'=\langle x,y\rangle$ for all $x,y\in\Lambda$.

\begin{definition} \label{def:dual_lattice} 
Given an $R$-lattice $\Lambda \sus V$ equipped with a nondegenerate symmetric %
bilinear form $\langle\cdot\,,\cdot\rangle$, the \textbf{dual lattice} of $\Lambda$  (with respect to $\langle\cdot\,,\cdot\rangle$) is 
\[\Lambda^\sharp = \{ \alpha \in V \; : \; \langle \alpha,x\rangle\in R \;\text{for all }x\in\Lambda \}.\]
If $R=\ZZ$ or $R=\ZZ_\ell$ for a prime $\ell$, we define the \textbf{level} of $\Lambda$ (also called the ``Stufe'' by some authors) to be the minimal positive integer $N$ such that 
$\Lambda^\sharp$ is an even integral lattice with respect to the bilinear form $(x,y)\mapsto N\langle x,y\rangle$. 
\end{definition}

Let $V$ be a vector space over $\QQ$. For a lattice $\Lambda\sus V$ and a prime $\ell \in \ZZ$, we define the \textbf{completion} of $\Lambda$ at $\ell$ to be the $\ZZ_\ell$-lattice $\Lambda_\ell \coloneqq \Lambda \otimes_{\ZZ} \ZZ_\ell$. A bilinear form $\Lambda\times\Lambda\to\QQ$ induces a bilinear form $\Lambda_\ell\times\Lambda_\ell\to\QQ_\ell$ for each prime $\ell$. 
The \textbf{local-global dictionary} for lattices in $V$ states that given any fixed lattice $\Lambda^0\sus V$, there is a bijection between lattices $\Lambda \sus V$ and tuples of lattices $(\Lambda_{\ell})$ over primes $\ell$ that satisfy $\Lambda_\ell =\Lambda^0_\ell$ for all but finitely many $\ell$ \cite[Theorem 9.1.1, Lemma 9.5.3]{Voight}.\footnote{All lattices in $V$ have the same completion at the infinite place, $\Lambda\otimes \RR= V\otimes\RR$, so it suffices to consider the finite places.} %
A property of a lattice $\Lambda$ is \textbf{local} if %
it holds for $\Lambda$ if and only if it holds for $\Lambda_\ell$ for all primes $\ell$. 
For instance, being an integral or an even integral lattice is a local property. The dual and level of a lattice can be computed locally, in the sense that they are uniquely determined by the duals and levels of the completions at all primes.

We finally note one elementary but useful property about lattices that will be used to check whether a certain set generates the lattice.
\begin{lemma}\label{lem:lattice quotient l}
    Let $\Lambda$ be a lattice, $\Lambda'\subseteq \Lambda$ a sublattice, and $\ell$ a prime. If $\Lambda'$ surjects onto $\Lambda/\ell\Lambda$ under reduction mod $\ell\Lambda$ then $\Lambda'_\ell=\Lambda_\ell$.
\end{lemma}
\begin{proof}
    Follows by Nakayama's Lemma or by a matrix determinant computation.
\end{proof}

\subsection{Lattices in quaternion algebras}\label{subsec:quaternion lattices}

A \textbf{quaternion algebra} $\mc{A}$ over $F$ is a (non-commutative) central simple %
 algebra of rank~$4$ over $F$. All quaternion algebras over $F$ can be given an $F$-basis $\{1, \i, \j, \i \j \}$ such that $\i^2 = a, \j^2 = b, \i\j = -\j\i$ for some nonzero $a, b \in F$. With respect to such a basis we can define the \textbf{standard involution} on $\mc{A}$ by $\overline{t+x\i+y\j+z\i\j} = t-x\i-y\j-z\i\j$; we then define the 
\textbf{(reduced) trace} of a quaternion element $\alpha\in \mc{A}$ as $\trd(\alpha) = \alpha + \overline{\alpha}$ and the \textbf{(reduced) norm} as $\nrd(\alpha) = \alpha \overline{\alpha}$. Since no other notion of trace or norm appears in this paper, we will usually drop the word ``reduced'' for convenience.

Using the trace, we have a nondegenerate symmetric bilinear form on $\mathcal{A}$,
\begin{equation}\label{eq:trace form on A}
    \langle\alpha, \beta\rangle := \trd(\alpha \overline{\beta})=\nrd(\alpha+\beta)-\nrd(\alpha)-\nrd(\beta),
\end{equation}
which we call the \textbf{trace form} %
on $\mc{A}$. This bilinear form induces a quadratic form  $Q\colon \mc{A}\to F$ given by the reduced norm, $Q(\alpha) = \nrd(\alpha) = \frac{1}{2}\langle \alpha, \alpha \rangle$. 

Now suppose $\mc{A}$ is a quaternion algebra over $\QQ$. For any field extension $F/\QQ$, $\A\otimes_{\bb{Q}} F$ is a quaternion algebra over $F$. We say $\mc{A}$ is \textbf{definite} if $\mc{A} \otimes_{\QQ} \bb{R}$ is isomorphic to a division algebra over $\RR$ (the classical Hamilton quaternions). Otherwise $\A$ is \textbf{indefinite}, in which case $\A\otimes_{\bb{Q}} \bb{R} \cong M_2(\bb{R})$.
In a similar way, we say a prime $p \in \ZZ$ \textbf{ramifies} in $\mc{A}$ if the algebra $\mc{A}\otimes_{\QQ} \bb{Q}_p$ is a division algebra, %
otherwise $p$ \textbf{splits} in $\mc{A}$ and $\mc{A}\otimes_{\QQ} \QQ_p \cong M_2(\QQ_p)$ is a matrix algebra.  %
The \textbf{discriminant} $\disc(\A)$ of $\mc{A}$ is the product  of primes that ramify in $\mc{A}$.

If $\Lambda\subseteq\A$ is a lattice, the set $\{\nrd(x):x\in\Lambda\}\subseteq\QQ$ generates a finitely generated $\ZZ$-module, which is cyclic because $\ZZ$ is a principal ideal domain. We define the \textbf{norm} of $\Lambda$, $\nrd(\Lambda)$, to be the positive generator of this $\ZZ$-module. 
\begin{definition}\label{def:normalized trace form}
    Given a lattice $\Lambda$ in a quaternion algebra $\mc{A}$ over $\bb{Q}$, we define the \textbf{normalized trace form} on $\Lambda$ by
    \[\phantom{\alpha,\beta\in\Lambda}\qquad \langle \alpha,\beta\rangle_{\Lambda}:=\frac{\trd(\alpha\overline{\beta})}{\nrd(\Lambda)},\qquad \alpha,\beta\in\Lambda.\]
\end{definition}
For any $\alpha\in\Lambda$, $\trd(\alpha\bar{\alpha})=2\nrd(\alpha)$ is an even integer multiple of $\nrd(\Lambda)$, so the normalized trace form equips $\Lambda$ with the structure of an even integral lattice. Given a prime $\ell$, we may equip $\Lambda_\ell=\Lambda\otimes_\ZZ \ZZ_\ell$ with the bilinear form induced from the normalized trace form on $\Lambda$.

\begin{lemma}\label{lem:disc square}
    If $(b_1,b_2,b_3,b_4)$ is a basis for a lattice $\Lambda$ in a quaternion algebra $\mc{A}$ over $\QQ$, then 
    \begin{equation}
        \det (\langle b_i,b_j\rangle_{\Lambda})_{1\leq i,j\leq 4}\in\ZZ \label{quant:disc_squared} 
    \end{equation}
    is a perfect square that does not depend on the choice of basis.
\end{lemma}
\begin{proof}
    This value is an integer because $\Lambda$ is an integral lattice with respect to $\langle\cdot\,,\cdot\rangle_\Lambda$, so it suffices to prove that it is a square in $\QQ$. 
    Let $\Lambda_0\sus\mc{A}$ be the lattice with basis $(b_1',b_2',b_3',b_4')=(1,\i,\j,\i\j)$. There exists a positive integer $m$ such that $m\Lambda\subseteq \Lambda_0$, and 
    \begin{align*}
        \det (\langle b_i,b_j\rangle_{\Lambda})_{1\leq i,j\leq 4}&=\frac{1}{m^8\nrd(\Lambda)^4}\det (\trd(mb_i\overline{mb}_j))_{1\leq i,j\leq 4}\\
        &=\frac{[\Lambda_0:m\Lambda]^2}{m^8\nrd(\Lambda)^4}\det (\trd(b_i'\overline{b}_j'))_{1\leq i,j\leq 4}\\
        &=\frac{[\Lambda_0:m\Lambda]^2}{m^8\nrd(\Lambda)^4}(4ab)^2,
    \end{align*}
    with the second equality holding by the proof of \cite[Lemma 15.2.5]{Voight}. 
\end{proof}

The \textbf{(reduced) discriminant} of $\Lambda$, denoted $\discrd(\Lambda)\in\ZZ$, 
is the positive square root of the quantity \eqref{quant:disc_squared} in \cref{lem:disc square}.

\subsection{Quaternion orders}\label{subsec:orders}

Let $\A$ be a quaternion algebra over $F$. An \textbf{order} in $\A$ is an $R$-lattice $\O \sus \A$ (\cref{subsec:lattices}) that contains $1$ and is closed under multiplication. For any $x\in\O$ we have $\nrd(x),\trd(x)\in R$. Given a subset $C\subseteq \O$, we write $R[C]$ for the $R$-algebra generated by $C$; that is, the minimal suborder of $\O$ that contains $C$.

Let $\O$ be an order in a quaternion algebra $\A$ over $\QQ$. Since $\nrd(\O)=1$, the normalized trace form on $\O$ agrees with the trace form on $\mc{A}$. In particular, any order $\O$ is an even integral lattice with respect to the trace form.
We call the dual of $\O$ with respect to the trace form the \textbf{codifferent} $\O^\sharp$ of $\O$. Note that $\O^\sharp \supseteq \O$. %

We note here an equivalent definition of level, specific for orders in quaternion algebras over $\QQ$ (this is the definition given in  \cite{Brzezinski83} and \cite{Eichler55}, for example).
\begin{lemma} \label{lem:level_alternative}
    For an order $\O$ in a quaternion algebra $\mc{A}$ over $\QQ$, the level of $\O$ as a lattice $($with respect to the trace form$)$ is equal to $\nrd(\O^\sharp)^{-1}$.
\end{lemma}
\begin{proof}
    For any integer $n$, the bilinear form $(x,y)\mapsto n\Tr(x\overline{y})$ gives $\O^\sharp$ the structure of an even integral lattice if and only if $n\Nrd(x)\in\ZZ$ for all $x\in\O^\sharp$. This holds if and only if 
    \[n\Nrd(\O^\sharp)\ZZ=n \{ \nrd(x):x\in\O^\sharp \} \subseteq\ZZ,\] 
    which in turn holds if and only if $n\in\Nrd(\O^\sharp)^{-1}\ZZ$. 
    Since $\O^\sharp$ contains $1$, $\nrd(1)=1$ is an integer multiple of $\Nrd(\O^\sharp)$ by definition, so $\Nrd(\O^\sharp)$ must be the reciprocal of a positive integer. Thus $N=\Nrd(\O^\sharp)^{-1}$ is the smallest positive integer in $\Nrd(\O^\sharp)^{-1}\ZZ$.
\end{proof}

\begin{definition} \label{def:families_of_orders}
Let $\O$ be an order in a quaternion algebra $\A$ over $F$. We say $\O$ is \textbf{maximal} if it is not strictly contained in another order of $\A$. We say $\O$ is \textbf{Eichler} if $\O = \O_1 \cap \O_2$ is an intersection of two (not necessarily distinct) maximal orders $\O_1$ and $\O_2$.
\end{definition}

For a lattice $\Lambda$ in a quaternion algebra $\mc{A}$ over $\QQ$, the properties of being an order, a maximal order, or an Eichler order are all local properties. The discriminant and norm of $\Lambda$ can be computed locally. 
For an order $\O\subseteq \mc{A}$, the (reduced) discriminant of $\O$ is equal to  $[\mf{O}:\O]\disc(\mc{A})$, where $\mf{O}$ is any maximal order containing $\O$. %

\subsection{Ideals and ideal classes}\label{subsec:ideals}

In this paper we  work primarily with right orders and right ideals, but all definitions given in this section have analogous left-sided versions. 

Let $\mc{A}$ be a quaternion algebra over $F$. The \textbf{right order} of an $R$-lattice $I$ in $\A$ is the order $\O_{\text{R}}(I) = \{ \alpha \in \mc{A}\; | \; I \alpha  \subseteq I \}$. If $\O$ is an order, we say that a lattice $I$ is a \textbf{right $\O$-ideal} if $\O_{\text{R}}(I) \supseteq \O$. 

Now suppose $\O$ is an order in a quaternion algebra $\A$ over $\QQ$.
A right $\O$-ideal~$I$ is \textbf{invertible} if $I$ is a locally principal right $\O$-ideal, i.e. for all primes $\ell$ there is an $\alpha_{\ell} \in \mc{A}_{\ell}^\times$ such that $I_{\ell} = \alpha_{\ell} \O_\ell$. This is equivalent to the existence of a lattice $J$ such that $JI  = \O_{\text R}(I)=\O$ and $IJ = \O_{\text L}(I)$, where $\O_{\text L}(I)$ is the left order of~$I$.

\begin{definition}\label{def:subseteqn}
    Let $\O$ be an order in a quaternion algebra $\A$ over $\QQ$, and suppose that~$I$ is an invertible right $\O$-ideal. We say that $J$ is \emph{contained in $I$ with relative norm $n$}, or $J \subseteq_n I$, if $J$ is an invertible right $\O$-ideal with $J\subseteq I$ and $\nrd(J)=n\Nrd(I)$.
\end{definition}

Two right $\O$-ideals $I, J$ 
are in the same \textbf{right ideal class}, written $[I] =[J]$, if there exists  $\alpha \in \A^\times$ such that $I = \alpha J$. %
Ideals in the same right ideal class are isometric as lattices equipped with their associated normalized trace forms (\cref{def:normalized trace form}): since $\trd(ab)=\trd(ba)$ for $a,b\in\mc{A}$, for any $\alpha \in\mc{A}^\times$ and $x,y\in I$ we have:
\[\langle \alpha x,\alpha y\rangle_{\alpha I}=\frac{\Tr((\alpha x)(\bar{y}\bar{\alpha}))}{\nrd(\alpha I)}=\frac{\Tr(\bar{\alpha}\alpha x\bar{y})}{\nrd(\alpha)\nrd( I)}=\frac{\Tr(x\bar{y})}{\nrd(I)}=\langle x,y\rangle_{I}.\]
In particular $x\mapsto \alpha x$ is an isometry from $I$ to $\alpha I$.

We say that two invertible right $\O$-ideals $I,J$ are in the same \textbf{genus}, and write $J\in\mathrm{Gen}(I)$, if the completions $I_\ell$ and $J_\ell$, with bilinear forms induced by the normalized trace forms on $I$ and $J$ respectively, are isometric as $\ZZ_\ell$-lattices for every prime $\ell$. Ideals in the same class are automatically in the same genus.

The \textbf{right ideal class set} of $\O$ is the set $\Cls (\O)$ consisting of equivalence classes of invertible right $\O$-ideals. It is well known (see e.g. \cite[Theorem 17.1.1]{Voight}) that the right ideal class set of an order in a quaternion algebra over $\QQ$ (more generally over any number field) is finite, and we call its cardinality the \textbf{class number} of $\O$.

\subsection{Brandt modules and algebras}\label{subsec:brandt}

Let $\O$ be an order in a definite quaternion algebra over~$\QQ$.
Let $I_1,\ldots,I_h$ be a list of representatives of the right ideal classes of $\mathcal{O}$, where $h$ is the class number of $\O$ and each $I_j$ is an \textit{invertible} right $\O$-ideal. Without loss of generality we may take $I_1=\O$. We define the (real) \textbf{Brandt module} 
\begin{align*}%
    \mf{M}(\mathcal{O}):=\{f:\{I_1,\ldots,I_h\}\to\RR\}.
\end{align*}
Thus $\mf{M}(\O)\simeq \RR^h$, and we have a fixed standard basis $v_1,\ldots,v_h$ given by Kronecker delta functions, $v_i(I_j)=\delta_{ij}$. We usually think of $\mf{M}(\O)$ as column vectors of length $h$ with entries in $\bb{R}$. %

The Brandt module comes equipped with an inner product for which the standard basis is orthogonal.
For each $i$, let $\O_\text{L}(I_i)$ denote the left order of $I_i$, and let $w_i:=|\O_\text{L}(I_i)^\times/\{\pm1\}|$. Then we define
\begin{equation}\label{eq:brandt inner product}
    \langle v_i,v_j\rangle =\left\{\begin{array}{cc}
    w_i & \text{if }i=j, \\
    0 & \text{otherwise.}
\end{array}\right.
\end{equation}
The Brandt module also comes equipped with matrices $B_n$ for all $n\geq 0$ acting on the left on $\mf{M}(\O)$, 
defined entry-wise with respect to the standard basis $\{v_1, \ldots, v_h\}$ by
\begin{align}
    B_n(i,j)&:=\frac{1}{2w_i}\left|\left\{\alpha\in I_jI_i^{-1}:\nrd(\alpha)\frac{\nrd(I_i)}{\nrd(I_j)}=n\right\}\right| & \label{eq:Bnij def} \\
    &\phantom{:}=\frac{1}{2w_i}\left|\left\{\alpha\in I_jI_i^{-1}:\tfrac12\langle\alpha,\alpha\rangle_{I_jI_i^{-1}}=n\right\}\right|, \nonumber
\end{align}
where in the second line we use the normalized trace form on $I_jI_i^{-1}$ (\cref{def:normalized trace form}). For $n\geq 1$ we have the equivalent form
\begin{align} \label{eq:Bnij_def_2}
    B_n(i,j)&=\left|\{J\subseteq_n I_j: [J]=[I_i] \} \right|
\end{align}
by \cite[Lemma 41.2.7(a)]{Voight}, where $J\subseteq_n I_j$ means that $J$ is an invertible right $\O$-ideal with $\nrd(J)=n\Nrd(I_j)$ (\cref{def:subseteqn}).
Note that $B_1$ is the identity matrix, and $B_0$ is the matrix where every entry in row $i$ is $\frac{1}{2w_i}$. 
\begin{definition} \label{def:Brandt algebra}
The \textbf{Brandt algebra} $\mf{B}(\O)$ is the $\mathbb{Z}$-subalgebra of $\End( \mf{M}(\O) ) \cong M_{h \times h} ( \bb{R})$ generated by the Brandt matrices $\{ B_n \; : \; n \ge 0 \}$. 
\end{definition}

\begin{lemma}\label{lem:Bn self-adjoint}
    For all $n\geq 0$, $B_n$ is self-adjoint with respect to the inner product. 
\end{lemma}
\begin{proof}
    Let $D$ be the diagonal matrix with $D(i,i)=w_i$, so the inner product is given by $\langle v,w\rangle=v^tDw$. The self-adjointness of $B_n$ with respect to 
    $\langle v,w\rangle$ is equivalent to checking that $DB_n=B_n^tD$. Since $B_n^tD=(DB_n)^t$, it suffices to check that $DB_n$ is a symmetric matrix. We have
    \begin{align*}
        (DB_n)(i,j) &= w_iB_n(i,j)= \frac{1}{2}\left|\left\{\alpha\in I_jI_i^{-1}:\tfrac12\langle\alpha,\alpha\rangle_{I_jI_i^{-1}}=n\right\}\right|,\\
        (DB_n)(j,i) &= w_jB_n(j,i)= \frac{1}{2}\left|\left\{\alpha\in I_iI_j^{-1}:\tfrac12\langle\alpha,\alpha\rangle_{I_iI_j^{-1}}=n\right\}\right|.
    \end{align*}
    Since $I_i$ and $I_j$ are invertible right $\O$-ideals, we have $\overline{I_i}I_i=\nrd(I_i)\O$ and $\overline{I_j}I_j=\nrd(I_j)\O$ by \cite[16.6.14]{Voight}, so $\alpha\mapsto \overline{\alpha}\frac{\nrd(I_i)}{\nrd(I_j)}$ defines an isometry from $I_jI_i^{-1}$ to $I_iI_j^{-1}$, each equipped with their normalized trace forms. This proves that the two sets above have the same size.
\end{proof}

\section{Commutativity of the Brandt algebra}\label{app:commutativity}

Let $\O$ be an order in a definite rational quaternion algebra $\mc{A}$.
Recall from  \cref{def:Brandt algebra} that $\mf{B}(\O)$ is the $\ZZ$-algebra generated by Brandt matrices $B_n$ for all $n\geq 0$. The main result of this section is the following.

\begin{theorem}\label{thm:Brandt algebra commutative}
    For an order $\O$ in a rational definite quaternion algebra, $\mf{B}(\O)$ is commutative. 
\end{theorem}

We begin with a few notes on the subtleties underlying this result, and why it does not directly follow from certain well-known results in the literature on establishing commutativity of Hecke algebras. The proof begins in \cref{ss: red_loc}.

\subsection{Failure of commutativity in Hecke algebras}\label{ex:brandt commutative is tricky}
     We first recall that commutativity of $\mf{B}(\O)$ 
     is well known for maximal orders. More generally, it follows immediately from \cite[Chapter II Theorem 2]{Eichler73} that for an arbitrary order $\O$ contained in a maximal order $\mf{O}$, the Brandt matrices $B_n$ with $n$ coprime to $[\mf{O}:\O]$ generate a commutative subalgebra. However, for this paper we will need to use a stronger result, namely the full Brandt algebra is commutative for arbitrary \textit{submaximal} orders (\cref{def:submaximal}).

    For an arbitrary rational quaternion order $\O$, the Brandt algebra $\mf{B}(\O)$ may be naturally interpreted as a subalgebra of a \emph{quaternionic Hecke algebra} $\mf{H}(\O)$. One way to define this Hecke algebra is to consider the adelic counterparts $\widehat{\mc{A}}$ and $\widehat{\mc{O}}$ of the quaternion algebra~$\mc{A}$ and order $\mc{O}$, and define $\mf{H}(\O)$ to be the convolution algebra of functions on $\widehat{\mc{A}}^\times$ that are bi-invariant under $\widehat{\mc{O}}^\times$. This structure arises, for instance, in the Jacquet-Langlands correspondence, which compares $\mf{H}(\O)$ to the algebra of Hecke operators for $GL_2$. One may identify $\mf{B}(\O)$ with a subalgebra of $\mf{H}(\O)$ by associating each Brandt operator $B_n$ with the characteristic function of elements of norm $n$ in $\widehat{\mc{A}}^\times$.
    
    If $\O$ is maximal then $\mf{H}(\O)=\mf{B}(\O)$ is commutative, but for a non-maximal order, the full Hecke algebra is \emph{not} commutative in general. \Cref{thm:Brandt algebra commutative} shows nonetheless that $\mf{B}(\O)$ is a commutative subalgebra of $\mf{H}(\O)$ for an arbitrary $\O$.

    There are many well-established techniques for proving commutativity of subalgebras of Hecke algebras. One of these methods, often called ``Gelfand's trick,'' can be used to prove commutativity of any subalgebra generated by double cosets of $\widehat{\mc{O}}^\times$ that are preserved by an involutive anti-automorphism (such as the standard involution on $\widehat{\mc{A}}$). See \cite[Proposition 3.8]{Shimura_book} for a general version of this method. Unfortunately, while the set of elements of norm $n$ is stable under the standard involution, the individual $\widehat{\mc{O}}^\times$-double cosets comprising this set are not, in general. 
    
    The construction in the following proposition shows that there is no direct analogue of Gelfand's trick in this scenario. That is, even if a subalgebra of $\mf{H}(\O)$ is generated by operators that are each stable under the standard involution, the subalgebra may be non-commutative. Thus any proof of commutativity of $\mf{B}(\O)$ must rely on properties of Brandt matrices beyond just stability under the standard involution.

\begin{proposition}\label{prop:brandt commutative is tricky}
    Let $\epsilon\geq 2$, $\ell$ a split prime of $\mc{A}$, and $\O\subseteq \mc{A}$ an Eichler order of index $\ell^\epsilon$ in a maximal order. Then there exist two operators $A,B\in\mf{H}(\O)$ satisfying $\overline{A}=A$ and $\overline{B}=B$, but with $AB\neq BA$. %
\end{proposition}
\begin{proof}
    By standard arguments we may reduce to considering the local Hecke algebra at $\ell$. Without loss of generality we may take $\O_\ell$ to be the standard Eichler $\ZZ_\ell$-order of level $\ell^{\epsilon}$:
    \[\O_{\ell}=\left\{\begin{pmatrix}a&b \\ \ell^{\epsilon}c & d\end{pmatrix}: a,b,c,d\in\ZZ_\ell\right\},\]
    with $\O_\ell^\times$ consisting of those matrices in $\O_\ell$ with $a,d\in\ZZ_\ell^\times$.
    Let
    \[\alpha=\begin{pmatrix}
        \ell^2&1\\0&1
    \end{pmatrix}\qquad\text{and}\qquad \beta=\begin{pmatrix}
        \ell&0\\0&1
    \end{pmatrix}.\]
    We will check by direct computation that 
    \begin{equation}\label{eq:not in double coset product}
        \overline{\beta}\alpha=\begin{pmatrix}
        \ell^2&1\\0&\ell
    \end{pmatrix}\notin \left(\O_{\ell}^\times \alpha\O_{\ell}^\times \cup \O_{\ell}^\times \overline{\alpha}\O_{\ell}^\times\right)\cdot\left(\O_{\ell}^\times \beta\O_{\ell}^\times\cup\O_{\ell}^\times \overline{\beta}\O_{\ell}^\times\right).
    \end{equation}
    We first compute the double coset $\O_\ell^\times \overline{\beta}\alpha\O_\ell^\times$. For $b,c,f,g\in \ZZ_\ell$ and $a,d,e,h\in\ZZ_\ell^\times$, we perform the following computation, using the assumption $\epsilon\geq 2$:
    \[\begin{pmatrix}
        a&b\\\ell^{\epsilon}c&d
    \end{pmatrix}\begin{pmatrix}
        \ell^2&1\\0&\ell
    \end{pmatrix}\begin{pmatrix}
        e&f\\ \ell^{\epsilon}g&h
    \end{pmatrix}\equiv \begin{pmatrix}
        0 & ah+bh\ell \\
        0 & dh\ell
    \end{pmatrix} \pmod{\ell^2M_2(\ZZ_\ell)}.\]
    In particular, since $a,d,e,h\in\ZZ_\ell^\times$, the top-right entry has valuation $0$ and the bottom-right entry has valuation $1$. (A similar computation can be used to verify that $\overline{\alpha}\notin \O_\ell^\times\alpha\O_\ell^\times$ and $\overline{\beta}\notin \O_\ell^\times\beta\O_\ell^\times$.)
    On the other hand, setting $u=\begin{pmatrix}
        w&x\\ \ell^{\epsilon}y&z
    \end{pmatrix}$ with $x,y\in\ZZ_\ell$ and $w,z\in\ZZ_\ell^\times$, we have the following congruences modulo $\ell^2M_2(\ZZ_\ell)$:
    \begin{align*}
        \alpha u\beta&\equiv\begin{pmatrix}
        0&z\\0&z
    \end{pmatrix}, & \alpha u\overline{\beta}&\equiv\begin{pmatrix}
        0&z \ell\\0&z\ell
    \end{pmatrix},\\
    \overline{\alpha} u \beta&\equiv\begin{pmatrix}
        w\ell&x-z\\0&0
    \end{pmatrix},& \overline{\alpha}u \overline{\beta}&\equiv\begin{pmatrix}
        w&(x-z)\ell\\0&0
    \end{pmatrix}.
    \end{align*}
    No matter what $u\in \O_\ell^\times$ we take, none of these elements have a top-right entry of valuation~$0$ and a bottom-right entry of valuation $1$, so they do not lie in $\O_\ell^\times \overline{\beta}\alpha\O_\ell^\times$. 
    We can conclude that
    \[\overline{\beta}\alpha\notin \O_\ell^\times\alpha\O_\ell^\times\beta\O_\ell^\times\cup
    \O_\ell^\times\alpha\O_\ell^\times\overline{\beta}\O_\ell^\times\cup
    \O_\ell^\times\overline{\alpha}\O_\ell^\times\beta\O_\ell^\times\cup
    \O_\ell^\times\overline{\alpha}\O_\ell^\times\overline{\beta}\O_\ell^\times,\]
    establishing \cref{eq:not in double coset product}.

    Now take $A$ and $B$ to be the characteristic functions of the subsets $\O_{\ell}^\times \alpha\O_{\ell}^\times \cup \O_{\ell}^\times \overline{\alpha}\O_{\ell}^\times$ and $\O_{\ell}^\times \beta\O_{\ell}^\times\cup\O_{\ell}^\times \overline{\beta}\O_{\ell}^\times$ of $\mc{A}_\ell^\times$, respectively. Observe that these sets are stable under the standard involution, so $\overline{A}=A$ and $\overline{B}=B$. The convolution product $BA$ evaluates to a nonzero value at $\overline{\beta}\alpha$, but by \cref{eq:not in double coset product}, $AB$ evaluates to $0$ at $\overline{\beta}\alpha$. This proves that $AB\neq BA$.
\end{proof}

\subsection{Reducing to a local computation}\label{ss: red_loc}
We return to the proof of \cref{thm:Brandt algebra commutative}.
\begin{lemma}\label{lem: Eisenstein coefficient ideal count'}
    For $n\geq 1$ and $1\leq j\leq h$, we have 
    \[\sum_{k=1}^hB_n(k,j)=\left|\{J\subseteq_n I_j\}\right|,\]
    and this value is independent of $j$.
\end{lemma}
\begin{proof} 
    The equality follows from the definition of $B_n(k,j)$, since every invertible right $\O$-ideal belongs to a unique right ideal class $[I_k]$ for some $1\leq k\leq h$. %
    Since each invertible right $\O$-ideal $I_j$ is locally principal, the independence from $j$ follows from the local-global dictionary (cf.~\cref{subsec:lattices}). %
\end{proof}
\begin{lemma}\label{lem: B0}
    $B_0$ commutes with $B_n$ for every $n\geq 1$.
\end{lemma}
\begin{proof}
Recall that $B_0(i,j) = \frac{1}{2w_i}$, for any $i,j$.    Since $B_n$ is self-adjoint by  \cref{lem:Bn self-adjoint}, we have
\begin{align*}
    (B_nB_0)(i,j)&=\sum_{k=1}^h \frac{1}{2w_k}B_n(i,k)=\sum_{k=1}^h \frac{1}{2w_i}B_n(k,i),\\
    (B_0B_n)(i,j)&=\sum_{k=1}^h \frac{1}{2w_i}B_n(k,j).
\end{align*}
Thus
$(B_nB_0)(i,j)=(B_0B_n)(i,j)$ by \cref{lem: Eisenstein coefficient ideal count'}.
\end{proof}
Given \cref{lem: B0}, we can restrict our attention to the Brandt matrices $B_n$ with $n\geq 1$. For $n\geq 1$, define
\[
\mf{b}_n(i,j) = \bigl{ \{ }   J \subseteq_n I_j \; : \;  [J] = [I_i] \bigr{ \} }.
\]
Then $B_n(i,j)=|\mf{b}_n(i,j)|$ by \cref{eq:Bnij_def_2}. 

We also define
\[\mf{N}_n(J,I):=\left\{J':J\subseteq J'\subseteq_n I\right\},\]
as the set of invertible right $\O$-ideals $J'$ containing $J$ which are contained in $I$ with relative norm $n$ (\cref{def:subseteqn}).
\begin{proposition}\label{lem:count by fibre}
    For any $i,j\in\{1,\ldots,h\}$ and $m,n\geq 1$,
    \[(B_mB_n)(i,j)=\sum_{J\in \mf{b}_{mn}(i,j)}\left|\mf{N}_n(J,I_j)\right|.
    \]
\end{proposition}

\begin{proof}
The number $(B_mB_n)(i,j)=\sum_{k=1}^hB_m(i,k)B_n(k,j)$ equals the sum over all $k$ of the sizes of the product sets $\mf{b}_m(i,k)\times \mf{b}_n(k,j)$. First we fix $1\leq k\leq h$ as well as $J'\in \mf{b}_n(k,j)$. Pick $\alpha_{J'}\in\mathcal{A}^\times$ with $J'=\alpha_{J'} I_k$. The map $J''\mapsto \alpha_{J'} J''$ then induces a bijection
\begin{align*}
    \mf{b}_m(i,k)=&\left\{J''\subseteq_m I_k: [J'']=[I_i]\right\}\quad\to\quad\left\{J\subseteq_m J': [J]=[I_i]\right\}.
\end{align*}
Now we still keep $k$ fixed but allow $J'$ to vary: the map $(J'',J')\mapsto (\alpha_{J'} J'',J')$ defines a bijection from $\mf{b}_m(i,k)\times \mf{b}_n(k,j)$ to
\[
\left\{(J,J'):J\subseteq_m J'\subseteq_n I_j,\; [J]=[I_i],\;[J']=[I_k]\right\},
\]
where $J$ and $J'$ both vary over invertible right $\O$-ideals.
If we now take the union over all~$k$, we find that
\[
\left\{(J,J'):J\subseteq_m J'\subseteq_n I_j, \; [J]=[I_i]\right\}
\]
has size $(B_mB_n)(i,j)$. Writing this final set as a union over $J\in \mf{b}_{mn}(i,j)$ gives us the desired result.
\end{proof}

If $i,j$, and $J\in\mf{b}_{mn}(i,j)$ are fixed, the set $\mf{N}_n(J,I_j)$ can be determined locally. In particular we can use this to prove the following.
\begin{lemma}[{\cite[Proposition 41.3.1]{Voight}}]\label{lem:coprime commute}
    If $\gcd(m,n)=1$ then $B_mB_n=B_nB_m=B_{mn}$. %
\end{lemma}
\begin{proof}
    For each $J\in\mf{b}_{mn}(i,j)$, there is a unique $J\subseteq J'\subseteq_n I_j$, defined by setting $J'_\ell=J_\ell$ for $\ell\mid n$ and $J'_\ell=(I_j)_\ell$ for $\ell\nmid n$. So $B_nB_m=B_{mn}$ by \cref{lem:count by fibre}.
\end{proof}
Since we have also checked that $B_0$ commutes with every Brandt matrix, this allows us to make the following reduction.

\begin{lemma}\label{lem: BlO}
    $\mf{B}(\O)$ is commutative if and only if the $\ZZ$-subalgebra of $\mf{B}(\O)$ generated by $\{B_{\ell^k}:k\geq 0\}$ is commutative for every prime $\ell$.
\end{lemma}

\subsection{Local analysis}\label{ss:local}

From now on we fix a prime $\ell$ and we denote $\O_\ell = \O\otimes_{\ZZ} \bb{Z}_\ell$. For $a\in\ZZ_\ell$ we let $\val(a)$ denote the $\ell$-adic valuation, so that $a\in\ell^{\val(a)}\ZZ_\ell^\times$. We declare $\val(0)=\infty$.

For $x\in\O_\ell$ and $m\in\NN$, let
$\mf{N}_{\ell^m}(x)=\mf{N}_{\ell^m}(x,\O_\ell)$ denote the set of principal right ideals $I\sus\O_\ell$ of norm $\ell^m$ with $x\in I$. 
Let $N_{\ell^m}(x)=N_{\ell^m}(x,\O_\ell)=|\mf{N}_{\ell^m}(x,\O_\ell)|$ be the cardinality of $\mf{N}_{\ell^m}(x,\O_\ell)$. 

The idea of the proof of \cref{thm:Brandt algebra commutative} %
is to relate the sets $\mf{N}_{\ell^m}(x,\O_\ell)$ of local ideals to the sets $\mf{N}_{\ell^m}(J,I_j)$ of global ideals that appear in \cref{lem:count by fibre}, justifying the abuse of notation. We will prove the identity $N_{\ell^m}(x)=N_{\ell^m}(\bar{x})$ for all $x\in\O_\ell$ with $\nrd(x)\neq 0$ (\cref{prop:Nmx eq Nmbarx}), and use this to prove commutativity of Brandt matrices. 
We begin with some simple relations satisfied by $N_{\ell^m}(x)$.

\begin{lemma}\label{l:reductions} Let $x,x'\in\O_\ell$ and $m\in\NN$. %
\begin{enumerate}[$(1)$]
    \item If $x-x'\in\ell^m\O_\ell$, then $N_{\ell^m}(x)=N_{\ell^m}(x')$.
    \item If $n=\val(\nrd(x))$ and $m\leq n<\infty$, then $N_{\ell^m}(x)=N_{\ell^{n-m}}(\bar{x})$.
\end{enumerate}
\end{lemma}
\begin{proof}
For (1), any right $\O_\ell$-ideal $I$ of norm $\ell^m$ contains $\ell^m\O_\ell$, 
so $x\in I$ if and only if $x'\in I$. That is, $\mf{N}_{\ell^m}(x)=\mf{N}_{\ell^m}(x')$ and consequently $N_{\ell^m}(x)=N_{\ell^m}(x')$.

For (2), observe that 
    \[N_{\ell^m}(x)=\left|\left\{(y,z)\in\O_\ell\times\O_\ell:yz=x,\;\nrd(y)\in\ell^m\ZZ_\ell^\times\right\}/\!\sim\right|\]
    with the equivalence $(y,z)\sim (yu,u^{-1}z)$ for any $u\in\O_\ell^\times$. Indeed, to each $(y,z)$ we can associate the ideal $y\O_\ell\in\mf{N}_{\ell^m}(x,\O_\ell)$ 
    with two pairs determining the same ideal if and only if they are equivalent, and every ideal in $\mf{N}_{\ell^m}(x,\O_\ell)$ can be obtained in this way. From this we see that $(y,z)\mapsto (\bar{z},\bar{y})$ induces a bijection $\mf{N}_{\ell^m}(x)\to\mf{N}_{\ell^{n-m}}(\bar{x})$.
\end{proof}

Say $x\in\O_\ell$ is \textbf{$m$-norm-reduced} %
if $\nrd(x)\neq 0$ and
\begin{align}\label{eq:nrd div trd}
    \val(\nrd(x))\leq m+\val(\trd(x)).
\end{align}
Note that every $x$ with nonzero norm is $m$-norm-reduced for sufficiently large $m$, so this is only a nontrivial constraint when $m$ is fixed.
We will show in the proof of \cref{prop:Nmx eq Nmbarx} that whenever $x$ is $m$-norm-reduced, we can define an explicit bijection $\mf{N}_{\ell^m}(x)\to\mf{N}_{\ell^m}(\bar{x})$. In order to prove $N_{\ell^m}(x)=N_{\ell^m}(\bar{x})$ more generally, we first need a way to reduce to the $m$-norm-reduced case.

\begin{lemma}\label{lem:nrd div trd}
    Let $\O_\ell$ be a quaternion order over $\ZZ_\ell$. Let $x\in\O_\ell$ with $\nrd(x)\ne 0$. Let $m\in\NN$ and assume that $\val(\trd(x))\neq m$. Then there exists an $m$-norm-reduced $x'\in\O_\ell$ with  $x-x'\in\ell^m\O_\ell$.
\end{lemma}
\begin{proof}
    Let $t:=\val(\trd(x))$ and $n:=\val(\nrd(x))$. If $n\leq m+t$ then $x'=x$ is already $m$-norm-reduced, so suppose instead that $n>m+t$. Since we are assuming $m\neq t$, the valuation of
    \[\nrd(x+\ell^m)=\nrd(x)+\ell^m\trd(x)+\ell^{2m}\]
    is equal to $m+\min(m,t)$, because exactly one of the three terms attains this minimal valuation. In particular, this implies $\nrd(x+\ell^m)\ne 0$. 
    On the other hand,
    \[\trd(x+\ell^m)=\trd(x)+2\ell^m\]
    has valuation greater than or equal to $\min(m,t)$, so 
    \[\val(\nrd(x+\ell^m))=m+\min(m,t)\leq \val(\ell^m\trd(x+\ell^m)).\]
    Thus $x+\ell^m$ is $m$-norm-reduced.
\end{proof}

\begin{proposition}\label{prop:Nmx eq Nmbarx}
    For any quaternion order $\O_\ell$ over $\ZZ_\ell$, any $m\in\NN$ and any $x\in\O_\ell$ with $\nrd(x)\ne 0$, we have $N_{\ell^m}(x)=N_{\ell^m}(\bar{x})$.
\end{proposition}
\begin{proof}
    We first prove the result assuming $\val(\trd(x))\neq m$.
    By \cref{l:reductions} we can replace~$x$ with any $x'\equiv x\bmod{\ell^m\O_\ell}$ without changing the values of $N_{\ell^m}(x)$ and $N_{\ell^m}(\bar{x})$. So by \cref{lem:nrd div trd}, we may assume without loss of generality that $x$ is $m$-norm-reduced. 
    
    Under this assumption we will prove that $y\O_\ell\mapsto \bar{x}x^{-1}y\O_\ell$ is a well-defined map $\mf{N}_{\ell^m}(x)\to\mf{N}_{\ell^m}(\bar{x})$.
    Given $y\O_\ell\in \mf{N}_{\ell^m}(x)$, we have
    \[x\in y\O_\ell \quad\Leftrightarrow \quad \bar{x}\in \bar{x}x^{-1}y\O_\ell\]
    and $\nrd(\bar{x}x^{-1}y)=\nrd(y)$, so it suffices to prove $\bar{x}x^{-1}y\in\O_\ell$. Observe that $\nrd(y)\trd(x)$ has valuation $m+\val(\trd(x))$, and so its quotient by $\nrd(x)$ is in $\ZZ_\ell$ by \cref{eq:nrd div trd}. Since $y^{-1}x\in\O_\ell$, we have
    \[\trd(x)(x^{-1}y)=\frac{\trd(x)}{\nrd(y^{-1}x)}\overline{y^{-1}x}\in\O_\ell,\]
    and therefore
    \[\bar{x}x^{-1}y=\trd(x)(x^{-1}y)-y\in\O_\ell.\]
    This proves that $\bar{x}x^{-1}y\O_\ell\in \mf{N}_{\ell^m}(\bar{x})$ as desired.
    Since $\bar{x}$ is also $m$-norm-reduced, the same construction defines an inverse map $\mf{N}_{\ell^m}(\bar{x})\to\mf{N}_{\ell^m}(x)$, proving that the two sets have the same cardinality.

    Finally we consider the case $\val(\trd(x))=m$. Set $n:=\val(\nrd(x))$. If $n\leq 2m$ then~$x$ is $m$-norm-reduced, and we may continue the argument as above. Otherwise $n>2m$ and  so we have $\val(\trd(x))\neq n-m$. By \cref{lem:nrd div trd}, there is an $(n-m)$-norm-reduced element $x'\in\O_\ell$ with $x-x'\in \ell^{n-m}\O_\ell$, and so we may use the above argument to prove that $N_{\ell^{n-m}}(x)=N_{\ell^{n-m}}(\bar{x})$.  This implies $N_{\ell^m}(x)=N_{\ell^m}(\bar{x})$ by \cref{l:reductions}(2).
\end{proof}
 We are now ready to prove the commutativity of the Brandt algebra $\mf{B}(\O)$.
\begin{proof}[Proof of \emph{\cref{thm:Brandt algebra commutative}}]
    By \cref{lem: BlO}, it suffices to prove that $B_{\ell^k}B_{\ell^m}=B_{\ell^m}B_{\ell^k}$, for all primes $\ell$ and all $m,k\geq 0$. Setting $n:=m+k$, this is equivalent to prove that $B_{\ell^{n-m}}B_{\ell^m}=B_{\ell^m}B_{\ell^{n-m}}$, for all primes $\ell$ and for all   $0\leq m\leq n$.

    Fix $1\leq i,j\leq h$, and let $J\in \mf{b}_{\ell^n}(i,j)$. Write $(I_j)_\ell=\gamma\O_\ell$ for some $\gamma\in\mc{A}_\ell^\times$, so that $\gamma^{-1}J_\ell=x\O_\ell$ for some $x\in\O_\ell$ with $\val(\nrd(x))=n$.
    We can define a map
    \[
        \mf{N}_{\ell^m}(J,I_j)\to \mf{N}_{\ell^m}(x,\O_\ell)
        \]
    by $J'\mapsto \gamma^{-1}J'_\ell$. The map is well-defined because the set on the left varies over locally principal right $\O$-ideals, so $\gamma^{-1}J'_\ell$ is a principal right $\O_\ell$-ideal of the correct norm, with $J\sus J'$ implying $x\in \gamma^{-1}J'_\ell$. The constraint $\frac{\Nrd(J')}{\Nrd(I_j)}=\ell^m$ forces $J'_{\ell'}=(I_j)_{\ell'}$ for all primes $\ell'\neq \ell$, so by the local-global dictionary, this map is in fact a bijection. We therefore have
    \begin{align}\label{eq:global ideal count is local}
        |\mf{N}_{\ell^m}(J,I_j)|=N_{\ell^m}(x).
    \end{align}
    Replacing $m$ with $n-m$, the same argument together with \cref{l:reductions}(2) and \cref{prop:Nmx eq Nmbarx} shows that
    \[|\mf{N}_{\ell^{n-m}}(J,I_j)|=N_{\ell^{n-m}}(x)=N_{\ell^m}(\bar{x})=N_{\ell^m}(x).\]
    This proves that $|\mf{N}_{\ell^m}(J,I_j)|=|\mf{N}_{\ell^{n-m}}(J,I_j)|$ for all $J\in\mf{b}_{\ell^n}(i,j)$, so that by \cref{lem:count by fibre} we have $(B_{\ell^m}B_{\ell^{n-m}})(i,j)=(B_{\ell^{n-m}}B_{\ell^m})(i,j)$. 
\end{proof}

\section{Submaximal orders}\label{subsec:submaximal}

Let $\mc{A}$ be a quaternion algebra over $\QQ$, and let $\mf{O}\sus\mc{A}$ be a maximal order. Our strategy to prove \cref{thm:generation bound} 
is to show that $\mf{O}$ has more elements of a given norm than any of its proper suborders. It will be useful to consider the maximal elements among the set of all proper suborders, and so we introduce the following convenient terminology.

\begin{definition}\label{def:submaximal}
    An order $\O\subsetneq\mf{O}$ is a \emph{submaximal order of $\mf{O}$} (or, simply, $\O$ is \emph{submaximal in $\mf{O}$}) if it is maximal with respect to inclusion among the proper suborders of~$\mf{O}$.
\end{definition}

\subsection{Classification}

\begin{lemma}\label{lem:max_suborders} \label{def:submaximal_orders} 
    Let $\mf{O}$ be a maximal order in $\A$, and let $\O\subseteq \mf{O}$ be a submaximal order. Then one of the following holds:
    \begin{enumerate}[$(a)$]
        \item $[\mf{O}:\O]=\qprime$ for a split prime $\qprime$ of $\mc{A}$, and $\O$ is Eichler.
        \item $[\mf{O}:\O]=p$ for a ramified prime $p$ of $\mc{A}$, and $\O=\ZZ+\mf{P}$, with $\mf{P}$ the unique maximal two-sided ideal of norm $p$ in $\mf{O}$.
        \item $[\mf{O}:\O]=\qprime^2$ for a split prime $\qprime$ of $\mc{A}$, and $\O=R+{\qprime}\mf{O}$ for some imaginary quadratic order $R\subseteq\mf{O}$ with $R/\qprime R\simeq \FF_{\qprime^2}$.
        \item $[\mf{O}:\O]=p^2$ for a ramified prime $p$ of $\mc{A}$, and $\O=R+p\mf{O}$ for some imaginary quadratic order $R\subseteq\mf{O}$ with $R/p R\simeq \FF_{p^2}$.
    \end{enumerate}
\end{lemma}
\begin{remark}
    Cases $(c)$ and $(d)$ can be combined into a single case by ignoring the splitting behavior of the prime in $\mc{A}$, but the two cases behave differently enough in practice that it will be useful to separate them. The condition $R/\ell R\simeq \FF_{\ell^2}$ is equivalent to the condition that $\ell$ is inert in the integral closure of $R$ and does not divide the conductor of $R$; for odd~$\ell$, this is equivalent to the condition that the discriminant of $R$ is a quadratic non-residue modulo $\ell$. 
\end{remark}
\begin{proof}
    By the local-global principle for lattices (e.g., see \cite[Theorem 9.1.1]{Voight} together with \cite[Lemma 10.4.3]{Voight}), there exists a prime ${\ell}$ such that  $\O$ is maximal at every prime except ${\ell}$, so $\ell$ is the unique prime dividing $ [\mf{O}:\O]$. %
    Let $\pi:\mf{O}\to\mf{O}/{\ell}\mf{O}$ be reduction mod ${\ell}$. By \cref{lem:lattice quotient l}, $\O$ cannot surject onto $\mf{O}/{\ell}\mf{O}$ under $\pi$, so $\O$ lands in some proper subring $A\subseteq\mf{O}/{\ell}\mf{O}$. For any such $A$, $\pi^{-1}(A)$ will be a suborder of $\mf{O}$ containing $\O$, hence equal to it. Thus it suffices to classify maximal proper subrings $A$ of the $4$-dimensional $\FF_{\ell}$-algebra $Q:=\mf{O}/{\ell}\mf{O}$. Note that $A\neq \FF_{\ell}$, because we have proper subrings of the form $\FF_{\ell}[\alpha]$ for any $\alpha\in Q\setminus\FF_\ell$.

    First suppose that $A$ contains an element $\alpha$ with irreducible minimal polynomial over $\FF_{\ell}$, i.e. it contains a copy of $\FF_{{\ell}^2}$. Then $A$ is a vector space over $\FF_{{\ell}^2}$; since $A\neq Q$, $A$ must equal $\FF_{\ell}[\alpha]$ by dimension considerations. A preimage of $\alpha$ in $\mf{O}$ generates a quadratic order $R$ in which ${\ell}$ is inert, and $\pi^{-1}(A)=R+{\ell}\mf{O}$. This gives cases $(c)$ and $(d)$.

    Now we assume every element of $A\setminus\FF_\ell$ %
    has reducible minimal polynomial, and divide into cases depending on whether $\ell$ is split or ramified in $\mc{A}$. 
    \begin{itemize}
        \item If ${\qprime}:=\ell$ is a split prime of $\mc{A}$, then $Q\simeq M_2(\FF_{\qprime})$, and we can consider the natural action of $Q$ on $\FF_{\qprime}^2$. 
        Our working assumption implies that $A$ does not contain any nontrivial field extension of $\FF_\qprime$, so by \cite[Theorem 1]{Racine}, $A$ is the stabilizer of a $1$-dimensional subspace of $\FF_{\qprime}^2$. Picking a basis of $\FF_{\qprime}^2$ so that the first basis vector is in this stabilized subspace, we see that $A$ is isomorphic to the set of upper triangular matrices in $M_2(\FF_{\qprime})$. 
        Therefore $\pi^{-1}(A) \otimes \ZZ_{\qprime} \simeq \left( \begin{smallmatrix}             \ZZ_{\qprime} & \ZZ_{\qprime} \\ {\qprime}\ZZ_{\qprime} & \ZZ_{\qprime}
        \end{smallmatrix} \right)$, and so $\O$ is Eichler with index $\qprime$ in $\mf{O}$ (\cite[Proposition 23.4.3]{Voight}). %
        
        \item If $p:=\ell$ is a ramified prime of $\mc{A}$, then $Q\simeq \FF_{p^2}[\epsilon]$, with noncommutative multiplication 
        \[(a+b\epsilon)(c+d\epsilon)=ac+(b\bar{c}+ad)\epsilon.\]
        Here $\epsilon^2=0$ and $\bar{c}=c^p$ is the Frobenius involution on $\FF_{p^2}$. Note that $C:=\FF_p+\FF_{p^2}\epsilon$ is a $3$-dimensional subalgebra of $Q$. Any $\alpha\in Q\setminus C$ is of the form $a+b\epsilon$ with $a\in\FF_{p^2}\setminus\FF_p$, and hence it satisfies an irreducible quadratic minimal polynomial %
        over $\FF_p$ (as one can check that $a+b\epsilon$ and $a$ have the same minimal polynomial in this case). 
        Since $A$ contains no such elements, we have $A\subseteq C$ and hence $A=C$ by maximality. This corresponds under $\pi^{-1}$ to the order $\ZZ+\mf{P}$, where $\mf{P}=\pi^{-1}(\FF_{p^2}\epsilon)$ is the unique maximal two-sided ideal of $\mf{O}$ with norm~$p$. \qedhere
    \end{itemize}
\end{proof}

For the rest of this paper, following the notation in \cref{lem:max_suborders}, we use $p$ to refer to a prime at which $\mc{A}$ ramifies, and for a submaximal order $\O\subseteq\mf{O}$ we write~$\qprime$ for a prime dividing $[\mf{O}:\O]$ at which $\mc{A}$ splits. We continue to use $\ell$ for arbitrary primes, unless otherwise specified.

\begin{remark}
    We give a few indications of how submaximal orders fit within existing classifications of quaternion orders. Let $\O\subseteq \mf{O}$ be submaximal. Then $\O$ is Bass (\cite[Definition 24.5.1]{Voight}), which follows from the proof of \cref{lem:level equals disc} below. Since $\mc{A}$ is defined over $\QQ$, this implies by \cite{CSV} that $\O$ is basic (\cite[Definition 24.5.6]{Voight}). The order $\O$ is Eichler if and only if $[\mf{O}:\O]=r$ is a split prime in $\mc{A}$ (case $(a)$), in which case it is also hereditary (\cite[Definition 21.1.1]{Voight}).
    
    If $[\mf{O}:\O]$ is equal to $\qprime$, $p$, or $p^2$ (cases $(a)$, $(b)$, and $(d)$) then $\O$ is a ``special order'' in the terminology of Hijikata, Pizer, and Shemanske \cite[Definition 6.1]{HPS}. But since special orders are Eichler away from the ramified primes of $\mc{A}$ by definition, submaximal orders with $[\mf{O}:\O]=\qprime^2$ (case $(c)$) are not special. 
\end{remark}

One useful property shared by all submaximal orders is that their level as a lattice is equal to their reduced discriminant.

\begin{lemma}\label{lem:level equals disc}
    For any submaximal order $\O$ in a maximal order $\mf{O}$, the level of $\O$ as a lattice $($with respect to the trace form on $\mc{A})$ is equal to $\disc(\O)$.
\end{lemma}
\begin{proof}
    To prove this we will first show that $\O$ is Gorenstein.
    An order $\O$ in a quaternion algebra is called \textbf{Gorenstein} if its dual lattice $\O^\sharp$ %
    is in fact an invertible right $\O$-ideal. This is a local property. One can show that being Gorenstein is equivalent to ${\rm Hom}_{\bb{Z}}(\O, \bb{Z})$ being a projective $\O$-bimodule, which connects the definition to a more general theory of Gorenstein modules. Eichler orders, and hence maximal orders, are Gorenstein, but a general order is not. For all this, and much more, see \cite[\S 24.2]{Voight}.
    
    In order to show a submaximal order $\O$ is Gorenstein, let $\O'$ denote the Gorenstein saturation of $\O$, so that $\O=\ZZ+f\O'$ for some $f\geq 1$ \cite[Proposition 24.1.4]{Voight}. Then $[\O':\O]=f^3$. But the index of a submaximal order $\O$ in any superorder is either $1$, a prime, or the square of a prime, as one can see by calculating discriminants. Therefore, we must have $f=1$ and hence $\O'=\O$.

    Since $\O$ is Gorenstein, we have $\nrd(\O^\sharp)^{-1}=\disc(\O)$ \cite[Lemma 24.2.5]{Voight}, and so the result follows by \cref{lem:level_alternative}.
\end{proof}

\subsection{Genus vectors in the Brandt module}\label{sec:genus subspace}

Let $I,J$ be invertible right $\O$-ideals. Recall that we can equip $I$ and $J$ with their normalized trace forms $\langle \cdot\, ,\cdot\rangle_I$ and $\langle \cdot\, ,\cdot\rangle_J$, (\cref{def:normalized trace form}) %
and write $J\in\mathrm{Gen}(I)$ if the completions of these lattices are isometric at every prime.

Let $I_1,\ldots,I_h$ be representatives of the invertible right $\O$-ideal classes, and $v_1,\ldots,v_h$ the corresponding Brandt module standard basis vectors. For each $i$, we define the \textbf{genus vector} $\gamma_i\in \mf{M}(\O)$ by
\[\gamma_i:=\sum_{\substack{1\leq j\leq h \\ I_j\in\mathrm{Gen}(I_i)}} \frac{1}{2w_j}v_j,\]
and let $\mf{M}_{\mathrm{G}}(\O)\subseteq \mf{M}(\O)$ denote the subspace spanned by the genus vectors $\gamma_i$ for $1\leq i\leq h$. 
By construction, $\gamma_i=\gamma_j$ if $I_j\in\mathrm{Gen}(I_i)$, so the number of distinct genus vectors is typically smaller than $h$; in fact we will see below that for many orders (including all maximal orders) there is a unique genus vector. 
Let
\begin{align}\label{eqn:e def'n}
    e:=\sum_{1\leq j\leq h} \frac{1}{2w_j}v_j
\end{align}
be the sum of the (distinct) genus vectors.
\begin{lemma}\label{lem:e eigen}
    Let $\O\subseteq\mc{A}$ be a quaternion order. The vector $e$ defined by 
    \emph{\cref{eqn:e def'n}} is an eigenvector of $B_n$ for all $n\geq 0$. 
\end{lemma}
\begin{proof}
    The image of $B_0$ is the span of $e$, so $e$ is an eigenvector of $B_0$. For any $n\geq 1$ and $1\leq i\leq h$, the coefficient of $v_i$ in $B_n e$ is
    \begin{align*}
        \sum_{1\leq j\leq h} \frac{B_n(i,j)}{2w_j}=\sum_{1\leq j\leq h } \frac{B_n(j,i)}{2w_i}
    \end{align*}
    by self-adjointness of $B_n$ (see \cref{lem:Bn self-adjoint}). By \cref{lem: Eisenstein coefficient ideal count'} this equals $\frac{1}{2w_i}$ times a constant that does not depend on $i$, so $B_n e$ is a multiple of $e$. 
\end{proof}

We now focus on maximal and submaximal orders, and find that the genus subspace is either $1$- or $2$-dimensional.

\subsubsection{$1$-dimensional genus subspace}

\begin{lemma}\label{lem:1dim genus}
    Let $\O\subseteq\mc{A}$ be a maximal or submaximal order. Let $\mf{O}$ be a maximal order containing $\O$, and suppose that $[\mf{O}:\O]$ is \emph{not} equal to an odd ramified prime of $\mc{A}$. 
    Then $\mf{M}_{\mathrm{G}}(\O)$ is $1$-dimensional, spanned by $e$.
\end{lemma}
\begin{proof} 
    If $\O$ is submaximal then we use the classification result in \cref{lem:max_suborders}: we are in cases $(a)$, $(c)$ and $(d)$ of the lemma, and also in case $(b)$ if $p=2$.
    
    We first show that for every prime $\ell$, every element of $\ZZ_\ell^\times$ is in the image of $\nrd\colon \O_\ell\to \ZZ_\ell$. 
    If $\O_\ell$ is maximal, or if $[\mf{O}:\O]=\ell^2$, then $\O_\ell$ contains an embedding of $\ZZ_{\ell^2}$ (the ring of integers of the unramified quadratic extension of $\QQ_\ell$), and already $\nrd\colon \ZZ_{\ell^2}^\times\to\ZZ_\ell^\times$ is surjective.
    If instead $\ell=r=[\mf{O}:\O]$ is split in $\mc{A}$, then $\O_\ell$ contains an embedding of $\ZZ_\ell\times\ZZ_\ell$, and $\nrd:\ZZ_\ell^\times\times\ZZ_\ell^\times\to\ZZ_\ell^\times$ ($(x,y)\mapsto xy$) is surjective.
    Finally, if $\ell=p=[\mf{O}:\O]=2$ is ramified in $\mc{A}$, then $\mf{O}_2$ contains an embedding of $\ZZ_{4}$, so $\O_2$ contains an embedding of $2\ZZ_{4}$, and $\nrd:(1+2\ZZ_{4})\to (1+2\ZZ_2)=\ZZ_2^\times$ is surjective \cite[\S I.7 Proposition 3]{CasselsFrohlich}.

    Now, we prove that any two invertible right $\O$-ideals $I,I'$ with normalized trace forms are in the same genus. Fix a prime $\ell$. Since $I,I'$ are locally principal, there exists $\alpha\in\mc{A}_\ell^\times$ with $I'_\ell=\alpha I_\ell$. This implies $\nrd(I')$ and $\nrd(\alpha)\nrd(I)$ have the same $\ell$-adic valuation. Since we have proven that every element of $\ZZ_\ell^\times$ is in the image of the norm map, there exists $u\in\O_\ell$ with $\nrd(u)=\frac{\nrd(I')}{\nrd(\alpha)\nrd(I)}$.
    Then for any $x,y\in I_\ell$ we have
    \[\langle \alpha xu,\alpha y u\rangle_{I'}=\frac{\Tr((\alpha x u)(\bar{u}\bar{y}\bar{\alpha}))}{\nrd(I')}=\frac{\nrd(u)\nrd(\alpha)\Tr(x\bar{y})}{\nrd(I')}=\frac{\Tr(x\bar{y})}{\nrd(I)}=\langle x,y\rangle_I.\]
    So $x\mapsto \alpha x u$ is an isometry from $I_\ell$ to $I'_\ell$.
\end{proof}

\subsubsection{$2$-dimensional genus subspace}

Now suppose instead that $p=[\mf{O}:\O]$ is an odd ramified prime for $\mc{A}$. This case is covered in detail in \cite[Section 5]{Pizerp2M}. Pizer's setup assumes that the only ramified prime of $\mc{A}$ is $p$, %
but the summary below holds more generally for arbitrary quaternion algebras over $\QQ$. 
We also note that Pizer works with left $\O$-ideals while we work with right $\O$-ideals.

For each invertible right $\O$-ideal $I$, the normalized quadratic form $\frac12\langle x,x\rangle_I = \tfrac{\nrd(x)}{\nrd(I)}$ on $I$ either always outputs quadratic residues mod $p$ (and $I$ has \textbf{positive character}) or always outputs quadratic non-residues mod $p$ (and $I$ has \textbf{negative character}), %
where we ignore the outputs in $p \ZZ_p$ \cite[Proposition 5.1]{Pizerp2M}. Ideals have the same character if and only if they are in the same genus: this can be seen by a modification of the proof of \cref{lem:1dim genus} (see the proof of \cite[Lemma 5.20]{Pizerp2M} and \cite[Remark 5.21]{Pizerp2M} for more details). In particular this implies that the character is the same for all ideals in the same ideal class.

Now recall the representatives $I_1 = \O, \ldots, I_h$ of the right ideal classes of $\O$.
By \cite[Theorem 5.9]{Pizerp2M}, the class number $h$ of $\O$ is even, and we may label these representatives so that $I_1, \ldots, I_{h/2}$ have positive character, $I_{h/2+1}, 
\ldots, I_h$ have negative character, and for each $1\leq i\leq h/2$, $I_i$ and $I_{i+h/2}$ have isomorphic left orders (so that $w_i=w_{i+h/2}$).
Then we set
\begin{equation}\label{eqn:e' def'n}
e' = \sum_{i=1}^{h/2} \frac{1}{2w_i}v_i-\sum_{i=h/2+1}^{h} \frac{1}{2w_i}v_i.
\end{equation}
That is, $e$ is the sum of the two genus vectors, and $e'$ is the difference of the two genus vectors. The above observations allow us to conclude the following result.
\begin{lemma}\label{lem:2dim genus}
    Let $\O$ be submaximal in a maximal order $\mf{O}$. Assume that $[\mf{O}:\O]$ \emph{is} equal to an odd ramified prime of $\mc{A}$. Then $\mf{M}_{\mathrm{G}}(\O)$ is $2$-dimensional, spanned by $e$ and $e'$.
\end{lemma}
Since $w_i=w_{i+h/2}$ for all $1\leq i\leq h/2$, $e'$ is orthogonal to $e$. In the next lemma we show that $e'$, like $e$, is also an eigenvector for any Brandt matrix $B_n$, $n\geq 0$. The computation of the eigenvalues of $e$ and $e'$ is the primary focus of \cref{sec:eisenstein part}. 
\begin{lemma}\label{lem:e' eigen}
Let $\O$ be submaximal in a maximal order $\mf{O}$, and assume that $[\mf{O}:\O]$ is equal to an odd ramified prime of $\mc{A}$. Then $e'$ as given by \emph{\cref{eqn:e' def'n}} is an eigenvector of $B_n$ for all $n\geq 0$.
\end{lemma}
\begin{proof}
The image of $B_0$ is the span of $e$, but since $B_0$ is self-adjoint and $\langle e,e'\rangle=0$ we have $\langle e,B_0e'\rangle=0$. Therefore we must have $B_0e'=0$.

We now assume $n\geq 1$. 
Following the proof of \cite[Theorem 5.15]{Pizerp2M} for $s=0$, with respect to the standard basis $\{v_1,\ldots,v_h\}$, we can write
\[
B_n=
\begin{pmatrix}
C_n & D_n\\
D_n & C_n
\end{pmatrix}
\]
for some matrices $C_n$, $D_n$ of size $h/2\times h/2$. We first show that $C_n$ is self-adjoint and has constant column sums (cf.~\cite[Lemma 5.23]{Pizerp2M}, parts (i) and (ii)). For self-adjointness, letting $D'$ be the diagonal matrix with $w_1,\ldots,w_{h/2}$ along the main diagonal and $D=\begin{psmallmatrix}
D' & 0\\
0 & D'
\end{psmallmatrix}$,  we have $DB_n=B_n^tD$ by \cref{lem:Bn self-adjoint} and hence $D'C_n=C_n^tD'$. To show $\sum_{i=1}^{h/2}B_n(i,j)$ is a constant for all $1\leq j\leq h/2$, note that this quantity is simply the number of invertible right $\O$-ideals $J \subseteq_n I_j$ with positive character, and the conditions defining such $J$'s are local, so the proof follows as in \cref{lem: Eisenstein coefficient ideal count'}.

With respect to the basis $\{v_1,\ldots,v_h\}$, let 
$u=\left(\frac{1}{2w_1},\ldots, \frac{1}{2w_{h/2}}\right)^t$, so that 
$e=\begin{psmallmatrix}
    u \\ u
\end{psmallmatrix}$, $e'=\begin{psmallmatrix}
    u \\ -u
\end{psmallmatrix}$. Using the properties of $C_n$ shown above, the proof of \cref{lem:e eigen} shows that $u$ is an eigenvector for $C_n$. Now since $e$ is an eigenvector for $B_n$, we find that $(C_n+D_n)u$ is a scalar multiple of $u$, so $u$ is also an eigenvector of $D_n$. This implies $B_ne'=\begin{psmallmatrix}
    (C_n-D_n)u \\ (D_n-C_n)u
\end{psmallmatrix}$ is a scalar multiple of $e'$ as desired.
\end{proof}

\subsection{Brandt algebra relations}\label{sec:brandt algebra relations}

Let $\O$ be a maximal or a submaximal order in the quaternion algebra $\mc{A}$.
For the proof of \cref{thm:generation bound} we will need to consider the action of the Brandt matrices $B_{\ell^k}$ on $\mf{M}(\O)$ for various primes $\ell$ and arbitrary exponents $k\geq 1$. We divide into cases depending on the power of the prime that divides the index of $\O$ in a maximal order $\mf{O}$.

For our purposes we will not need to study ramified primes of $\mc{A}$ that divide the index, so by \cref{lem:max_suborders} we only need to consider three cases: $\ell\nmid [\mf{O}:\O]$ (\cref{lem:classical Brandt recurrence} and \cref{lem:classical ramified Brandt recurrence}), $\ell$ equals a split prime $r$ with $[\mf{O}:\O]=r$ (\cref{eq:recurrence}), or $\ell$ equals a split prime $r$ with $[\mf{O}:\O]=r^2$ (\cref{prop:Eu eigenvector}). One of the main results we provide for each of these orders is a recurrence relation satisfied by the Brandt matrices $B_{\ell^k}$ for $k\geq 0$: these are \cref{lem:classical Brandt recurrence}, \cref{eq:recurrence}, and \cref{cor:index r2 recurrence}, respectively. As this section is rather long and technical, we recommend readers skip ahead to \cref{sec:modular forms} on a first read-through and only refer back to this section as necessary.

\subsubsection{$\ell\nmid [\mf{O}:\O]$}
In this case $\O_\ell$ is maximal, and we have the following two classical results, depending on whether $\ell$ is split or ramified in $\mc{A}$. (We note that $\ell\mid\disc(\O) \Longleftrightarrow \ell\mid\disc(\mf{O})$.)

\begin{lemma}\label{lem:classical Brandt recurrence}
    Let $\O\subseteq \mc{A}$ be a quaternion order and $\ell\nmid\disc(\O)$ a prime. Then for all $k\geq 1$ we have
    \[B_{\ell^{k+1}}=B_\ell B_{\ell^k}-\ell B_{\ell^{k-1}}.\]
\end{lemma}
\begin{proof}
    See \cite[Chapter II Theorem 2]{Eichler73} or \cite[Proposition 41.3.6]{Voight}.
\end{proof}

\begin{lemma}\label{lem:classical ramified Brandt recurrence}
    Let $\O\subseteq \mc{A}$ be a quaternion order with $\O_p$ maximal, where $p$ is a ramified prime for $\mc{A}$. Then for all $k\geq 1$ we have $B_{p^k}=B_p^k$, and $B_p^2=1$.
\end{lemma}
\begin{proof}
    In this setting $\O_p$ is the valuation ring of the division algebra $\mc{A}_p$, so for each $k\geq 0$ there is a unique ideal of norm $p^k$ in $\O_p$, which is contained in all the ideals of norm $p^{k'}$ for $k'\leq k$. Using the local-global dictionary, this implies that if $I$ is an invertible right $\O$-ideal, then there is a unique ideal of relative norm $p^k$ in $I$ for every $k\geq 1$. 
    
    This implies that for all $J\in\mf{b}_{p^k}(i,j)$ we have $|\mf{N}_p(J,I_j)|=1$. By \cref{lem:count by fibre} this implies $B_{p^{k-1}}B_p=B_{p^k}$, which gives $B_p^k=B_{p^k}$ by induction. We compute $B_{p^2}=1$ by noting that the only $\O_p$-ideal of norm $p^2$ is $p\O_p$, so for any $1\leq j\leq h$ the unique ideal of relative norm $p^2$ in $I_j$ is $pI_j$, which is in the same ideal class as $I_j$. Thus $|\mf{b}_{p^2}(i,j)|$ is $1$ if $i=j$ and is $0$ otherwise.
\end{proof}

\subsubsection{$[\mf{O}:\O]=r$} \label{s: index_ell}

Let $\mf{O}$ be a maximal order containing $\O$ and suppose $[\mf{O}:\O]=r$ for a split prime $r$ of $\mc{A}$. The main result in this section is \cref{eq:recurrence}, a recurrence relation on the Brandt matrices $B_{r^k}$.

By the classification of submaximal orders \cref{lem:max_suborders}, $\O$ is an Eichler order of index~$r$ in $\mf{O}$.
Let $\mf{r}$ be the unique invertible two-sided ideal of $\O$ with norm $r$ \cite[23.4.19]{Voight}.
We can define an operator 
\begin{equation}\label{eq:Wr}
    W_{\mf{r}}:\mf{M}(\O)\to\mf{M}(\O)
\end{equation} 
on standard basis vectors induced by $[I_i]\mapsto [I_i\mf{r}]$.
\begin{lemma}\label{lem:W commutes}
    $W_{\mf{r}}$ is an involution that is self-adjoint with respect to the inner product on the Brandt module $\mf{M}(\O)$ and commutes with every Brandt matrix $B_n\in\mf{B}(\O)$.
\end{lemma}
\begin{proof}
    Since $\mf{r}^2=r\O$,  $W_{\mf{r}}$ is an involution. Thus, for $1\leq i,j\leq h$, we have  $W_{\mf{r}}v_i= v_j$ if and only if $v_i= W_{\mf{r}}v_j$, and so if these equalities do not hold we have
    \[\langle v_i,W_{\mf{r}}v_j\rangle=0=\langle W_{\mf{r}}v_i,v_j\rangle.\]
    If the equalities do hold, then $[I_j]=[I_i\mf{r}]$, which implies that $I_i$ and $I_j$ have isomorphic left orders and so
    \[\langle v_i,W_{\mf{r}}v_j\rangle=\langle v_i,v_i\rangle=w_i=w_j=\langle v_j,v_j\rangle=\langle W_{\mf{r}}v_i,v_j\rangle.\]
    This proves self-adjointness.
    
    Letting $\sigma$ be the permutation of $\{1,\ldots,h\}$ induced by $[I_{\sigma(i)}]:=[I_i\mf{r}]$, the following holds for any $n\geq 1$:
    \begin{align*}
        (B_nW_{\mf{r}})(i,j)&=B_n(i,\sigma(j))\\
        &=\{J\subseteq_n I_j\mf{r}: [J]=[I_i]\}\\
        &=\{J\subseteq_n I_j: [J]=[I_i\mf{r}]\}\\
        &=B_n(\sigma(i),j)=(W_{\mf{r}}B_n)(i,j),
    \end{align*}
    where in the third line we used the bijection $J\mapsto J\mf{r}^{-1}$.

    Finally, since $W_{\mf{r}}$ is self-adjoint and has constant column sum (equal to $1$), the argument from \cref{lem: B0} shows that it commutes with $B_0$.
\end{proof}

As a result, we may pick a basis for $\mf{M}(\O)$ consisting of simultaneous eigenvectors for all Brandt operators as well as for $W_{\mf{r}}$.
The main result we will prove in this section is a recurrence relation for the Brandt operators $B_{r^n}\in\mf{B}(\O)$.

\begin{proposition}
    \label{eq:recurrence}
    Let $\O\subseteq\mf{O}$ be submaximal of index $r$, where $r$ is a split prime of $\mc{A}$, %
    and let $W_{\mf{r}}$ be as in \emph{\cref{lem:W commutes}}. Then for all $n\geq 1$,
    \[B_{r^{n+1}}=B_{r^n}(B_r-rW_{\mf{r}})-rB_{r^{n-1}}.\]
\end{proposition}

\begin{proof}
In light of \cref{lem:count by fibre}, we may express $(B_{r^n}B_{r})(i,j)$ as a sum over $J\in \mf{b}_{r^{n+1}}(i,j)$ of the values $|\mf{N}_r(J,I_j)|$, and by \cref{eq:global ideal count is local}, we can write $|\mf{N}_r(J,I_j)|=N_{r}(x)$, where $(I_j)_r=\gamma\O_r$ and $\gamma^{-1}J_r=x\O_r$; since $J\subseteq_{r^{n+1}} I_j$ we have $x\in\O_r$ and $\val(\nrd(x))=n+1$. 
So we will partition $\mf{b}_{r^{n+1}}(i,j)$ into subsets depending on the value of $N_{r}(x)$, and use this to find an alternate expression for $(B_{r^n}B_{r})(i,j)$.
\begin{lemma}\label{lem:countideals}
    Let $x\in\O_r$ and suppose $\val(\nrd(x))\geq 2$. %
    Then
    \begin{align*}
        N_r(x)=\begin{cases}
            2r+1 & \text{for }x\in r\O_r,\\
            r+1 & \text{for }x\in\mf{r}_r\setminus r\O_r,\\
            1 & \text{for }x\in\O_r\setminus\mf{r}_r.
        \end{cases}
    \end{align*}
\end{lemma}

\begin{proof}
Since $\O$ is Eichler of index $r$, it is isomorphic to the subring of $M_2(\ZZ_r)$ consisting of matrices with bottom-left entry divisible by $r$ \cite[Proposition 23.4.3]{Voight}. For ease of notation we assume equality. In this notation we have 
\[\O_r=\begin{pmatrix}
    \ZZ_r & \ZZ_r \\ r\ZZ_r & \ZZ_r
\end{pmatrix}\qquad\text{and}\qquad  \mf{r}_r=\begin{pmatrix}
    r\ZZ_r & \ZZ_r \\ r\ZZ_r & r\ZZ_r
\end{pmatrix}.\]
We define the set $S$ to consist of the following $2r+1$ elements of $\O_r$:
    \begin{align*}
\alpha_i&:=\begin{pmatrix}
r & i\\
0&1
\end{pmatrix},
\quad
\beta_i:=\begin{pmatrix}
1&0\\
ir&r
\end{pmatrix}
\quad \text{for } i=0,\ldots,r-1,
\qquad
\varpi:=\begin{pmatrix}
0&1\\
r&0
\end{pmatrix}.
\end{align*}
    These elements generate each integral invertible right $\O_r$-ideal of norm $r$ exactly once, with $\varpi$ generating the two-sided ideal $\mf{r}_r$. 
    
    Now let $x=\begin{psmallmatrix}
        a&b\\ r c&d
    \end{psmallmatrix}\in \O_r$. To determine the number of integral invertible right $\O_r$-ideals $J$ of norm $r$ containing $x$, it suffices to determine, for each $\mu\in S$, whether or not $\mu^{-1}x\in\O_r$.
    To this end, we compute $\mu^{-1}x$ for each $\mu\in S$:
    \[\alpha_i^{-1}x=\begin{pmatrix}
        \frac{a}{r}-ci & \frac{b-di}{r} \\ cr & d
    \end{pmatrix},\qquad \beta_i^{-1}x=\begin{pmatrix}
        a & b \\ c-ai & \frac{d}{r}-bi
    \end{pmatrix},\qquad \varpi^{-1}x= \begin{pmatrix}
        c & \frac{d}{r} \\ a & b
    \end{pmatrix}.\]
    If $a$ and $d$ are both units, then $x\in\O_r^\times$, which does not occur because of the assumption $\val(\nrd(x))\geq 2$. If $d$ is a unit (so $r\mid a$), then there is exactly one $\mu\in S$ for which $\mu^{-1} x\in\O_r$, namely $\mu=\alpha_i$ with $i\equiv bd^{-1}\bmod r$. Likewise, if $a$ is a unit (so $r\mid d$), then we can only take $\mu=\beta_i$ with $i\equiv ca^{-1}\bmod r$. 
    
    Now consider the case that neither $a$ nor $d$ is a unit; this is exactly the case that $x\in \mf{r}_r$. In this setting, if $b$ and $c$ are both units, then $x\in \varpi\O_r^\times$, which does not occur by $\val(\nrd(x))\geq 2$. If $b$ is a unit (so $r\mid c$), then there are $r+1$ options: we can take $\mu=\varpi$ or $\mu=\beta_i$ for any $i$. If $c$ is a unit (so $r\mid b$), then again there are $r+1$ options: we can take $\mu=\varpi$ or $\mu=\alpha_i$ for any $i$. If $a,b,c,d$ are all divisible by $r$ (so $x\in r\O_r$), then there are $2r+1$ options, because $\mu^{-1} x\in\O_r$ for all $\mu\in S$.
  \end{proof}

Now that we have these counts for each case, we need to count how many $J\in\mf{b}_{r^{n+1}}(i,j)$ land in each of these cases. We have
\[x\in r\O_r\quad\Leftrightarrow \quad J\sus r I_j\quad\Leftrightarrow \quad\frac{1}{r}J\sus I_j\quad\Leftrightarrow \quad\frac{1}{r}J\in \mf{b}_{r^{n-1}}(i,j),\] %
so the number of $J$ with $N_r(x)=2r+1$ is equal to $B_{r^{n-1}}(i,j)$. 
Likewise,
\[x\in\mf{r}_r\quad\Leftrightarrow \quad  J\mf{r}^{-1}\sus I_j\quad\Leftrightarrow \quad J\mf{r}^{-1} \in \mf{b}_{r^n}(\sigma(i),j),\]
where $\sigma$ is the involution on $\{1,\ldots,h\}$ induced by $[I_{\sigma(i)}]=[I_i\mf{r}]$. That is, there is a bijection between ideals $J\in \mf{b}_{r^{n+1}}(i,j)$ such that $x\in\mf{r}_r$ and ideals in $I_j$ of relative norm $r^n$ and with ideal class $[I\mf{r}]$, and these are counted by $(W_{\mf{r}}B_{r^n})(i,j)$. So
\begin{align*}
    (B_{r^n}B_r)(i,j)&=\sum_{\substack{J\in\mf{b}_{r^{n+1}}(i,j)\\ N_r(x)=2r+1}}N_r(x)+\sum_{\substack{J\in\mf{b}_{r^{n+1}}(i,j)\\ N_r(x)=r+1}}N_r(x)+\sum_{\substack{J\in\mf{b}_{r^{n+1}}(i,j)\\ N_r(x)=1}}N_r(x) \\
    &=(2r+1)\cdot B_{r^{n-1}}(i,j)\\
    &\qquad +(r+1)\cdot (W_{\mf{r}}B_{r^n}(i,j)-B_{r^{n-1}}(i,j))\\
    &\qquad +1\cdot (B_{r^{n+1}}(i,j)-W_{\mf{r}}B_{r^n}(i,j)) \\
    &= r B_{r^{n-1}}(i,j)+r W_{\mf{r}}B_{r^n}(i,j)+B_{r^{n+1}}(i,j).
\end{align*}
This implies the statement of \cref{eq:recurrence}.
\end{proof}

\subsubsection{$[\mf{O}:\O]=r^2$}

Let $\mf{O}$ be a maximal order containing $\O$ and suppose that $[\mf{O}:\O]=r^2$ for $r$ a split prime of $\mc{A}$. In particular,  $\mf{O}_{\ell}=\O_{\ell}$ for all primes $\ell\neq r$.
The main result of this section is \cref{prop:Eu eigenvector}, which computes the eigenvalues of eigenvectors for Brandt matrices in $\mf{B}(\O)$ by relating them to eigenvectors for Brandt matrices in $\mf{B}(\mf{O})$. %

We begin by setting some notation. 
Fix representative invertible $\O$-ideals $I_1=\O$, $I_2,\ldots,I_h$, and let $v_1,\ldots,v_h$ be the corresponding standard basis for $\mf{M}(\O)$. Also fix representative invertible $\mf{O}$-ideals $I_1'=\mf{O},I_2',\ldots,I_{h'}'$, and let $w_1,\ldots,w_{h'}$ be the corresponding standard basis for $\mf{M}(\mf{O})$. %
We use $B_n$ and $\widetilde{B}_n$ to denote Brandt matrices in $\mf{B}(\O)$ and $\mf{B}(\mf{O})$, respectively. Following the notation of \cref{ss: red_loc}, for $1\leq i,j\leq h$ let 
\[\mf{b}_n(i,j)=\{J\subseteq_n I_j:[J]=[I_i]\}\]
as $J$ varies over right $\O$-ideals, and for $1\leq i,j\leq h'$ let 
\[\widetilde{\mf{b}}_n(i,j)=\{J\subseteq_n I'_j:[J]=[I'_i]\}\]
as $J$ varies over right $\mf{O}$-ideals. 

We define a map $\rho$ from invertible right $\O$-ideals to invertible right $\mf{O}$-ideals, given by ideal extension: 
\[\rho(I):=I\mf{O}.\]
Note that this is well-defined because invertible ideals are locally principal.
This map will allow us to explicitly relate counts of $\O$-ideals to counts of $\mf{O}$-ideals. 
See \cite[26.6.1]{Voight} for a discussion of the map $\rho$. %

For $i\in\{1,\ldots,h\}$, we define $\widetilde{i}\in\{1,\ldots,h'\}$ by $[I'_{\widetilde{i}}]=[\rho(I_i)]$. %
Using the extension map~$\rho$ we may define an operator on Brandt modules:%
\begin{align*}
    E:\mf{M}(\O)&\to \mf{M}(\mf{O})\\
    v_i&\mapsto w_{\widetilde{i}}.
\end{align*}
The map $E$ will allow us to relate Brandt eigenvectors for %
$\mf{B}(\O)$ to Brandt eigenvectors for $\mf{B}(\mf{O})$. 
We begin by considering Brandt matrices $B_n$ for $n$ coprime to $r$.
\begin{lemma}\label{lem:superorder coprime case}
    For all $n$ coprime to $r$, %
    \[EB_n=\widetilde{B}_nE.\]
\end{lemma}
\begin{proof}
    Fix $1\leq i'\leq h'$ and $1\leq j\leq h$. The quantity $(EB_n)(i',j)$ counts the set
    \begin{align*}
        \bigcup_{1\leq i\leq h\,:\,\widetilde{i}=i'} \mf{b}_n(i,j)&=\left\{J\subseteq_n I_j:[J\mf{O}]=[I'_{i'}]\right\}\\
        \intertext{and the quantity $(\widetilde{B}_nE)(i',j)$ counts the set}
        \widetilde{\mf{b}}_n(i',\widetilde{j})&=\left\{J\subseteq_n I'_{\widetilde{j}}:[J]=[I'_{i'}]\right\}.
    \end{align*}
    So it suffices to construct a bijection between these two sets. %
    Recall that $[I'_{\widetilde{j}}]=[I_j\mf{O}]$. As the size of $\widetilde{\mf{b}}_n(i',\widetilde{j})$ does not depend on the choice of representative $I'_{\widetilde{j}}$ in its equivalence class, we may assume without loss of generality that $I'_{\widetilde{j}}=I_j\mf{O}$.
    
    We now consider the map $\rho$ with restricted domain and codomain:
    \begin{align}\label{eq:rho restricted}
        \rho:\bigcup_{i\,:\,\widetilde{i}=i'}\mf{b}_n(i,j)&\to \widetilde{\mf{b}}_n(i',\widetilde{j}),\\
        J&\mapsto J\mf{O}.\nonumber
    \end{align}
    For any $J$ in the domain, its image is an invertible $\mf{O}$-ideal contained in $\rho(I_j)$ with relative norm $n$ and with $[\rho(J)]=[\rho(I_i)]=[I'_{i'}]$, 
    showing that the map is well-defined. On the other hand, any $I'$ in the codomain uniquely determines an ideal $J$ in the domain: since~$n$ is coprime to $r$ we must have $J_r=(I_j)_r$, and since $\mf{O}_{\ell}=\O_{\ell}$ for all $\ell\neq r$ we must have $J_\ell=I'_\ell$. We can conclude by the local-global dictionary that \cref{eq:rho restricted} is a bijection, proving that $EB_n=\widetilde{B}_nE$.
\end{proof}

We now consider Brandt matrices $B_{r^k}$ for $k\geq 1$. In this case the map $\rho$ of \cref{eq:rho restricted} is neither injective nor surjective. In order to quantify the failure of injectivity, we define a ``splitting'' operator 
\begin{align*}
    S:\mf{M}(\mf{O})&\to \mf{M}(\O)\\
    w_{j'}&\mapsto\sum_{i=1}^h c_{i,j'} v_i,
\qquad
c_{i,j'}
=
\left|\left\{
J \in \rho^{-1}(I'_{j'})
:
[J]=[I_i]
\right\}\right|.
\end{align*}
Note that $c_{i,j'}=0$ unless $\tilde i=j'$.
\begin{lemma}\label{lem:ES}
    The operator $E\circ S$ acts on $\mf{M}(\mf{O})$ via scalar multiplication by $r^2-r$.
\end{lemma}
\begin{proof}
    Let $I:=I'_{j'}$. 
The operator $(E\circ S)$ multiplies $w_{j'}$ 
by 
    $|\rho^{-1}(I)|$. 
     We can compute this quantity by the local-global correspondence: a given $J\in\rho^{-1}(I)$ is determined at each prime $\ell$ by an $\O_\ell$-ideal $J_\ell$ with $J_\ell\mf{O}_\ell=I_\ell$. For all $\ell\neq r$ we have $\O_\ell=\mf{O}_\ell$ and so $J_\ell=I_\ell$. 
     Now suppose we write $I_r=\alpha\mf{O}_r$ for some $\alpha\in\mathcal{A}_r^\times$. Then for every $u\in \mf{O}_r^\times$, $J_r:=\alpha u\O_r$ defines an $\O_r$-ideal with $J_r\mf{O}_r=I_r$. All such $\O_r$-ideals are obtained in this way, and we have $\alpha u\O_r=\alpha u'\O_r$ if and only if $u^{-1}u'\in \O_r^\times$. %
     Thus $|\rho^{-1}(I)|=[\mf{O}_r^\times:\O_r^\times]$. Using \cite[Lemma 26.6.7]{Voight} we find that this value equals $r^2-r$, because $r$ is split in $\mc{A}$ and $\O$ has Eichler symbol $-1$ at $r$. 
\end{proof}

In order to quantify the failure of surjectivity, we will show that for any $J\in\mf{b}_{r^k}(i,j)$, $J\mf{O}$ is contained not only in $I_j\mf{O}$, but in fact in $rI_j\mf{O}$.

\begin{lemma}\label{lem:contained in rad}
    If $J\subseteq I$ are invertible right $\O$-ideals 
    with $J_r\neq I_r$, then $J\subseteq rI\mf{O}$.
\end{lemma}
\begin{proof}
    It suffices to work locally at $r$. Let $I_r=\gamma\O_r$ for some $\gamma\in\mathcal{A}_r^\times$, so that $\gamma^{-1}J_r$ is a right $\O_r$-ideal strictly contained in $\O_r$. Since $\O_r/r\mf{O}_r\simeq \FF_{r^2}$ is a field, and $x\in \O_r$ is a unit if and only if $\nrd(x)\not\equiv 0\bmod r$, we can conclude that $\O_r$ is the disjoint union of $\O_r^\times$ and $r\mf{O}_r$.
    Since $\gamma^{-1}J_r$ does not contain a unit, it must be contained in $r\mf{O}_r$. Hence $J_r\subseteq \gamma r\mf{O}_r=rI_r\mf{O}_r$.
\end{proof}

\begin{remark}
    This can also be seen more explicitly. For convenience suppose $r$ is odd. Up to isomorphism we may write $\mf{O}_r=M_2(\ZZ_r)$, and $\O_r$ is the ring of matrices of the form $\begin{psmallmatrix}
        a & ub \\ b+r c & a+r d
    \end{psmallmatrix}$ for $a,b,c,d\in\ZZ_r$ and some fixed non-quadratic residue $u\in\ZZ_r^\times$. Following the proof above, we can then write $\gamma^{-1}J_r$ as $\alpha\O_r$ for some $\alpha\in\O_r\setminus\O_r^\times$. The norm of $\alpha$ then has the form 
    \[(a^2-ub^2)+r(ad-ubc)\in r\ZZ_r\]
    for some $a,b,c,d\in\ZZ_r$,
    and since $u$ is not a square in $\ZZ_r$, this forces $a,b\in r\ZZ_r$. Thus $\alpha\in r \mf{O}_r$.
\end{remark}

\begin{proposition}\label{prop:brandt lift}
    We have $B_{r}=0$, and for $n\geq 2$ we have
    \begin{align*}
        B_{r^n} &= S \widetilde{B}_{r^{n-2}}E.
    \end{align*}
\end{proposition}

\begin{proof}   
    For $n\geq 1$, $1\leq i,j\leq h$, and any $J\in \mf{b}_{r^n}(i,j)$, we have $J_r\neq (I_j)_r$ and therefore $J$ is contained in $rI_j\mf{O}$ by \cref{lem:contained in rad}. Then $\nrd(J)$ is a  multiple of $r^2\nrd(I_j)$, so $B_{r}(i,j)=0$.

    Now suppose $n\geq 2$. 
    For any $1\leq i,j\leq h$ we define a variant of 
    \cref{eq:rho restricted}. As in the proof of \cref{lem:superorder coprime case}, by changing representatives for the $\mf{O}$-ideal classes, we may assume without loss of generality that $\rho(I_j)= I'_{\tilde{j}}$ so we have a modification of the $\rho$ map
    \begin{align}
        \tfrac1r\rho:\bigcup_{k\,:\,\widetilde{k}=\widetilde{i}}\mf{b}_{r^n}(k,j)&\to \widetilde{\mf{b}}_{r^{n-2}}(\widetilde{i},\widetilde{j}),\label{eq:b map}\\
        J&\mapsto \tfrac{1}{r}\rho(J).\nonumber
    \end{align}
    This is well-defined by \cref{lem:contained in rad}. Now let $I'$ be an element of the codomain, so that $I'\subseteq I_j\mf{O}$. %
    For any $J\in r\rho^{-1}(I')$ we have $\frac1rJ\subseteq \frac1rJ\mf{O}=I'$, so that
    \[J\subseteq rI'\subseteq I_j(r\mf{O})\subseteq I_j\O=I_j.\]
    Since $[\rho(J)]=[I']=[I'_{\widetilde{i}}]$, we can conclude that $J$ is in the domain of this map. Therefore every element of $r\rho^{-1}(I')$ is in the domain of \cref{eq:b map}.
    
    By splitting $\mf{b}_{r^n}(i,j)$ along fibres of the map \cref{eq:b map}, we have
    \begin{align*}
        B_{r^n}(i,j)=\sum_{I'\in \widetilde{\mf{b}}_{r^{n-2}}(\widetilde{i},\widetilde{j})} \left|r\rho^{-1}(I')\cap\mf{b}_{r^n}(i,j)\right|.
    \end{align*}%
    Now we compute the size of each term in the sum. Since we have checked that every $J\in r\rho^{-1}(I')$ is in $\mf{b}_{r^n}(k,j)$ for some $k$, the only thing we have to check is the ideal class condition: that is, we must count the number of $J\in \rho^{-1}(I')$ with $[J]=[I_i]$. But this value depends only on the ideal class of $I'$. So replacing each $I'$ with $I'_{\widetilde{i}}$, we have
    \[B_{r^n}(i,j)=c_{i,\tilde{i}}\left(\widetilde{B}_{r^{n-2}}(\widetilde{i},\widetilde{j})\right),\]
    where $c_{i,\tilde{i}}$ is the number of $J\in\rho^{-1}(I'_{\widetilde{i}})$ with $[J]=[I_i]$. 
    Equivalently, $c_{i,\tilde{i}}$ is the component of $v_i$ in $Sw_{\widetilde{i}}$, so %
    \[B_{r^n}(i,j) = (S\widetilde{B}_{r^{n-2}})(i,\widetilde{j})=(S\widetilde{B}_{r^{n-2}}E)(i,j).\qedhere
    \] 
\end{proof}

As a consequence of this formula, together with \cref{lem:ES} and \cref{lem:classical Brandt recurrence}, it is straightforward to derive the following description of the algebra generated by the Brandt matrices $B_{r^k}\in\mf{B}(\O)$.
\begin{corollary}\label{cor:index r2 recurrence}
    The matrices $B_{\qprime^k}$ for all $k\geq 0$ are contained in the subalgebra of $\mf{B}(\O)$ generated by commuting matrices $B_{\qprime^2}$ and $B_{\qprime^3}$: we have $B_{\qprime}=0$ and
    \[B_{\qprime^{k+1}}=\frac{1}{\qprime^2-\qprime}B_{\qprime^{k}}B_{\qprime^{3}}-\qprime B_{\qprime^{k-1}}\]
    for all $k\geq 2$.
\end{corollary}

\begin{remark}
    There is a natural generalization of \cref{prop:brandt lift}. Suppose $\O$ is an order and $\ell$ is a prime such that $\O_\ell$ is Bass but not Eichler. Let $\widetilde{\O}$ be the order such that $\widetilde{\O}_\ell$ is the radical idealizer of $\O_\ell$, and $\widetilde{\O}_{\ell'}=\O_{\ell'}$ for $\ell'\neq \ell$. Let $\mf{r}$ be the right $\O$-ideal where $\mf{r}_\ell=\mathrm{rad}(\O_\ell)$ is the Jacobson radical of $\O_\ell$, and $\mf{r}_{\ell'}=\O_{\ell'}$ for $\ell'\neq \ell$. Then $\mf{r}$ is a two-sided invertible $\widetilde{\O}$-ideal, and using a similar argument we can show $B_{\ell^n}=0$ for $1\leq n<e$ and $B_{\ell^n}=SW_{\mf{r}}\widetilde{B}_{\ell^{n-e}}E$ for $n\geq e$, where~$e$ is the valuation of $\nrd(\mf{r})$ and $W_{\mf{r}}$ is the operator on $\mf{M}(\widetilde{\O})$ induced by 
    $[I]\mapsto [I\mf{r}]$.
\end{remark}

We end this section by consolidating the above results into a statement that relates eigenvalues of Brandt matrices for $\O$ to eigenvalues of Brandt matrices for $\mf{O}$.
Given any order $\O$ in $\mc{A}$, if $u\in\mf{M}(\O)$ is an eigenvector for every Brandt matrix in $\mf{B}(\O)$, we call $u$ a \textbf{Brandt eigenvector} and %
let $\lambda_n(u)$ denote the eigenvalue of $u$ under $B_n\in\mf{B}(\O)$. 

\begin{proposition}\label{prop:Eu eigenvector}
    Let $\O$ be a submaximal order contained in a maximal order $\mf{O}$ with $[\mf{O}:\O]=r^2$ for a split prime $r$ of $\mc{A}$. If $u\in\mf{M}(\O)$ is a Brandt eigenvector, then $\lambda_{r}(u)=0$ and one of the following holds.
    \begin{enumerate}[$(a)$]
        \item $\lambda_{r^k}(u)=0$ for all $k\geq 2$.
        \item There exists $w\in\mf{M}(\mf{O})$, an eigenvector for $\widetilde{B}_n\in\mf{B}(\mf{O})$ for all $n\geq 1$, satisfying $\lambda_{r^k}(u)=(r^2-r)\lambda_{r^{k-2}}(w)$ for all $k\geq 2$ and $\lambda_n(u)=\lambda_n(w)$ for $n$ coprime to $r$.
    \end{enumerate}
\end{proposition}
\begin{proof}
    If $Eu=0$ then $B_{r^k}u=0$ for all $k\geq 1$ by \cref{prop:brandt lift}, so $\lambda_{r^k}(u)=0$ for all $k\geq 1$. So from now on suppose $w:=Eu\neq 0$. Then from \cref{lem:superorder coprime case} we have $\widetilde{B}_nEu=\lambda_n(u)Eu$ for $n$ coprime to $r$, and by applying $E$ to both sides of \cref{prop:brandt lift} we have $B_{r}(u)=0$ and for $k\geq 2$ we have
    \[ES\widetilde{B}_{r^{k-2}}(Eu)=\lambda_{r^k}(u)(Eu).\]
    We have $ES=r^2-r$ by \cref{lem:ES}. This proves that $Eu$ is an eigenvector for $\widetilde{B}_{n}$ for any $n\geq 1$, with $\lambda_r(u)=0$ and $\lambda_{r^k}(u)=(r^2-r)\lambda_{r^{k-2}}(Eu)$ for $k\geq 2$. 
\end{proof}

Note that in fact the $w$ from \cref{prop:Eu eigenvector}(b) is an eigenvector for the full Brandt algebra $\mf{B}(\mf{O})$, a fact which can be derived from the proofs of \cref{lem:eisenstein coeffs}(c) (in the case $u$ is a multiple of the genus vector of $\O$) and \cref{lem:index r2 eigenvalues} (in the case $u$ is cuspidal).

\section{Background: modular forms}\label{sec:modular forms}

\subsection{Modular forms and Hecke operators}\label{subsec:modular forms} There are many excellent general references for modular forms; we mention \cite{DiamondShurman} and \cite{Ogg} that the reader wishing for more details regarding the fundamentals may consult.

The group $\SL_2(\ZZ)$ acts on the upper half plane $\bb{H} = \{ z \in \bb{C} \; : \; \Im(z) > 0 \}$ 
by fractional linear transformations as
$\bigl(\begin{smallmatrix} a & b \\ c & d \end{smallmatrix}\bigr) z := \frac{a z + b}{c z + d}$. Via $z\mapsto \bigl(\begin{smallmatrix} z\\1 \end{smallmatrix}\bigr)$, we can view this as derived from the action of $\SL_2(\bb{C})$ on $\bb{P}^1(\bb{C})$. This interpretation allows us to extend the action of $\SL_2(\bb{Z})$ to $\bb{H}^* = \mathbb{H} \sqcup \mathbb{P}^1(\QQ)$. One refers to $(1\;0)^t$ as $\infty$.

A \textbf{congruence subgroup} of $\SL_2(\ZZ)$ is a subgroup $\Gamma \le \SL_2(\ZZ)$ such that $\Gamma$ contains the kernel of the reduction-by-$N$ map $\SL_2(\ZZ) \to \SL_2(\ZZ/ N \ZZ)$ for some integer $N$. %
For $\Gamma \subseteq \SL_2(\bb{Z})$ a congruence subgroup, the \textbf{modular curve} $X(\Gamma)$ is the quotient 
$\Gamma \backslash \mathbb{H}^*$.  The \textbf{cusps} of a congruence subgroup $\Gamma$ are the $\Gamma$-orbits in $\bb{P}^1(\bb{Q})$. 

Fixing an integer $N \ge 1$, we will be particularly interested in the congruence subgroup of level $N$ which is the preimage of the Borel subgroup of $\SL_2(\ZZ/N\ZZ)$, denoted
\begin{align*}
    \Gamma_0(N) = \Biggl{ \{ } \begin{pmatrix} a & b \\ c & d \end{pmatrix} \equiv \begin{pmatrix} * & * \\ 0 & *  \end{pmatrix} \pmod{N} \Biggr{ \} } .
\end{align*}

A \textbf{modular form} of \textbf{weight} $k \in \bb{Z}_{\ge 0}$ for a congruence subgroup $\Gamma$ is a holomorphic function $f\colon \bb{H} \to \CC$, which is bounded in neighborhoods of the cusps of $\Gamma$ 
and satisfies the transformation law $f(\gamma z) = (c z + d)^k f(z)$ for all $\gamma = \begin{psmallmatrix}
a & b \\ c & d \end{psmallmatrix} \in \Gamma$. %
The set of modular forms of weight $k$ for $\Gamma_0(N)$ is a finite dimensional vector space over $\CC$ denoted $\mc{M}_k(\Gamma_0(N))$. Every modular form $f(z)$ for $\Gamma_0(N)$ has a Fourier expansion at $\infty$
\begin{equation}\label{eq:q expansion}
f(q) = \sum_{n=0}^\infty a_n q^n,  \; \, \quad \; q=q(z) = e^{2 i \pi z} \; \; , \; \; a_n \in \bb{C}. 
\end{equation}
This is often referred to as its \textbf{$q$-expansion}, and we will often write modular forms in terms of their $q$-expansions.
We will say $f\in\mc{M}_k(\Gamma_0(N))$ is \textbf{normalized} if $a_1=1$. Since there exist nonzero modular forms with $a_1=0$, it is not true in general that every element of $\mc{M}_k(\Gamma_0(N))$ is a scalar multiple of a normalized form.

A modular form $f \in \mc{M}_k(\Gamma_0(N))$ is a \textbf{cusp form} if the suitably defined $q$-series expansion of $f(\gamma z)$ has zero constant coefficient for all $\gamma \in \SL_2(\bb{Z})$ (equivalently, $f$ vanishes at the cusps of $\Gamma_0(N)$). We denote by $\mc{S}_k(\Gamma_0(N))$ the subspace of cusp forms of $\mc{M}_k(\Gamma_0(N))$. %

Fixing a positive integer $N$ and an integer $k \ge 0$, there is a special family of linear operators $\{ T_n \}_{n \ge 1}$ acting on $\mathcal{M}_k(\Gamma_0(N))$ known as the \textbf{Hecke operators}. For an integer~$n$ coprime to $N$, the $n$th Hecke operator acts on the $q$-expansion of a modular form $f(q) = \sum_{m=0}^\infty a_m q^m \in \mathcal{M}_k(\Gamma_0(N))$ as %
\begin{equation}\label{eq:Hecke action on q expansion}
    (T_nf)(q) = \sum_{m=0}^\infty \sum_{d \mid \gcd(m, n)} d^{k-1} a_{\frac{mn}{d^2}} q^m
\end{equation}
(cf.\ \cite[Proposition 5.3.1]{DiamondShurman}). 
\begin{definition}\label{def: anemic hecke algebra}
    Fix a positive integer $N$. We define the \textbf{anemic Hecke algebra} $\Tan_N$ to be the $\ZZ$-algebra generated by the Hecke operators $T_n$ on $\mathcal{M}_k(\Gamma_0(N))$ for all $n\ge 1$ with $\gcd(n,N)=1$. We say $f\in \mathcal{M}_k(\Gamma_0(N))$ is a \textbf{$\Tan_N$-eigenvector} if $f$ is an eigenvector for all $T\in\Tan_N$. 
\end{definition}
When the level $N$ is clear from context we may write $\Tan$ instead of $\Tan_N$. There is also a larger \textbf{Hecke algebra} $\T=\T_N$ of operators on $\mathcal{M}_k(\Gamma_0(N))$ generated as a $\ZZ$-algebra by suitably defined Hecke operators $T_n$ for all $n\ge 1$; see \cite[Section 5.3]{DiamondShurman} for a full definition. A $\T$-eigenvector, that is a modular form $f\in \mathcal{M}_k(\Gamma_0(N))$ that is an eigenvector for all $T\in\T$, is often called a \textbf{Hecke eigenform}. %
\begin{remark}\label{rmk:more general hecke}
    Hecke operators may also be defined on other spaces of modular forms, but we have restricted our presentation to the case relevant to this paper. For example, \cite{DiamondShurman} discusses Hecke operators for $\Gamma_1(N)$, and the results we need can be obtained by noting that $\mc{M}_2(\Gamma_0(N))$ is the set of modular forms for $\Gamma_1(N)$ with trivial character (see \cite[Exercise 4.3.3(a)]{DiamondShurman}). For instance, note that the diamond operators $\langle d\rangle$ for $(d,N)=1$ (see \cite[\S 5.2]{DiamondShurman}) act as the identity on $\mc{M}_2(\Gamma_0(N))$.
\end{remark}

The action of $\T$ on $\mathcal{M}_k(\Gamma_0(N))$ preserves the cuspidal subspace $\mc{S}_k(\Gamma_0(N))$. If $n_1, n_2$ are a pair of coprime integers, then the Hecke operators $T_{n_1}$ and $T_{n_2}$ satisfy $T_{n_1} T_{n_2} = T_{n_1 n_2}$. If $\ell$ is a prime coprime to $N$, then the Hecke operators $T_{\ell^n}$ satisfy the recursion relation $ T_{\ell^{n+1}} = T_{\ell}T_{\ell^n} - \ell^{k-1} T_{\ell^{n-1}}$, where $T_1$ is the identity operator. In particular, $\Tan$ consists of pairwise commuting linear operators (cf.\ \cite[\S5.3]{DiamondShurman}). 
For a congruence subgroup $\Gamma$, we can make $\mc{S}_k(\Gamma)$ into an inner product space with the \textbf{Petersson inner product} $\langle \cdot , \cdot \rangle: \mc{S}_k(\Gamma) \times \mc{S}_k(\Gamma) \to \bb{C}$, defined by 
$$
\langle f, g \rangle = \frac{1}{V_{\Gamma}} \int_{X(\Gamma)} f(z) \overline{g(z)} \left( \Im(z) \right)^k d \mu(z),
$$ 
where $d \mu(z) = \frac{dx dy}{y^2}$ is the hyperbolic measure on the upper half plane and $V_{\Gamma}$ denotes the volume of the modular
curve $X(\Gamma)$ with respect to this measure.

If $n$ is coprime to $N$, then the Hecke operator $T_n$, viewed as a linear operator on $\mc{S}_k(\Gamma_0(N))$, is self-adjoint with respect to the Petersson inner product (cf. \cite[Theorem 5.5.3]{DiamondShurman} and \cref{rmk:more general hecke}).
In particular, this implies that $\mc{S}_k(\Gamma_0(N))$ has a basis of $\Tan$-eigenvectors with real eigenvalues.

For positive integers $M \mid N$ there is an inclusion $\Gamma_0(N) \hookrightarrow \Gamma_0(M)$ giving an inclusion of modular forms $\mc{M}_k(\Gamma_0(M)) \hookrightarrow \mc{M}_k(\Gamma_0(N))$. Moreover, if $f \in \mc{M}_k(\Gamma_0(M))$ and $N = Mm$, then $f(z) \mapsto f(mz)$ is also an inclusion $\mc{M}_k(\Gamma_0(M)) \to \mc{M}_k(\Gamma_0(N))$; in terms of the $q$-expansion, we write this as $f(q)\mapsto f(q^m)$. For a fixed integer $N$, the \textbf{old subspace} $\mc{S}_k^{\text{old}}(\Gamma_0(N))$ of $\mc{S}_k(\Gamma_0(N))$ is the subspace generated by
\begin{equation*}%
    \bigcup_{\substack{M \mid N, \\ M\neq N}} \bigcup_{m \mid \frac{N}{M}} \{ f(q^m) \; : \; f \in \mc{S}_k(\Gamma_0(M))\},
\end{equation*}
and the \textbf{new subspace} $\mc{S}_k^{\text{new}}(\Gamma_0(N))$ of $\mc{S}_k(\Gamma_0(N))$ is, per definition, the orthogonal complement of $\mc{S}_k^{\text{old}}(\Gamma_0(N))$ in $\mc{S}_k(\Gamma_0(N))$ with respect to the Petersson inner product. Both $\mc{S}_k^{\text{new}}(\Gamma_0(N))$ and $\mc{S}_k^{\text{old}}(\Gamma_0(N))$ are stable under the action of the full Hecke algebra~$\T$ %
(cf.~\cite[Proposition 5.6.2]{DiamondShurman}). 

A modular form $f \in \mc{S}_k(\Gamma_0(N))$ is called a \textbf{newform} if it is a normalized $\Tan$-eigenvector 
in the new subspace $\mc{S}_k^{\text{new}}(\Gamma_0(N))$. %
In fact $\mc{S}_k^{\text{new}}(\Gamma_0(N))$ has a basis of newforms, which are all eigenvectors for the full Hecke algebra $\T$ with all eigenvalues in $\RR$ \cite[Theorem 5]{AtkinLehner}. %

If $f(q)=\sum a_nq^n$ is a newform in $\mc{S}_2(\Gamma_0(N))$, the \emph{Deligne bound} 
(credited in this case to Eichler and Shimura) gives a precise upper bound on the sizes of coefficients: $|a_\ell|\leq 2\sqrt{\ell}$ 
for all primes $\ell$ (including $\ell\mid N$). Since a newform $f$ of level $N$ has an Euler product
\[L(s, f) = \prod_{p \nmid N} \left(1 - a_p p^{-s} + p^{k-1-2s}\right)^{-1} \prod_{p \mid N} \left(1 - a_p p^{-s}\right)^{-1},
\]
it follows that this bound can be generalized to all coefficients.
\begin{lemma}[Deligne bound]\label{lem:deligne bound}
    If $f(q)=\sum a_nq^n$ is a newform of weight $2$ $($and any level$)$, then
    \[
    |a_n|\leq \sigma_0(n)\sqrt{n},
    \]
    where $\sigma_0(n)$ is the number of positive divisors of $n$.
\end{lemma}

At certain points in the paper we study cusp forms that are $\Tan$-eigenvectors, but which do not necessarily lie in the new subspace. We can use the following result to decompose these forms.

\begin{lemma}\label{lem:decompose almost eigenform}
    Let $g\in\mc{S}_2(\Gamma_0(N))$ be a $\Tan$-eigenvector. %
    Then for some $M\mid N$, %
    there exists a single newform $f\in\mc{S}_2^{\mathrm{new}}(\Gamma_0(M))$ and coefficients $c_d\in\CC$ such that
    \[g(q)=\sum_{d \mid \frac{N}{M}} c_df(q^d).\]
\end{lemma}
\begin{proof}
    By \cite[Theorem 5]{AtkinLehner}, $\mc{S}_2(\Gamma_0(N))$ has a basis consisting of cusp forms $f(q^d)$, where $M$ varies over divisors of $N$, $f$ varies over newforms in $\mc{S}_2^{\text{new}}(\Gamma_0(M))$, and $d$ varies over divisors of $\frac{N}{M}$. We use this to write
    \[g(q)=\sum_{i=1}^k c_i f_i(q^{d_i})\]
    for some $k\geq 1$, $c_i\in\CC\setminus\{0\}$, $M_i\mid N$, $f_i\in \mc{S}_2^{\text{new}}(\Gamma_0(M_i))$, and $d_i\mid\frac{N}{M_i}$. If for some $1\leq i,j\leq k$ we have $f_i\neq f_j$, then again by \cite[Theorem 5]{AtkinLehner}, there exist infinitely many primes $\ell$ (and in particular, at least one prime $\ell$ coprime to $N$) for which $f_i$ and $f_j$ have distinct eigenvalues under $T_\ell$. But $g$ is an eigenvector under $T_\ell$, so all $f_i$ must have the same eigenvalue under $T_\ell$, a contradiction. Hence all $f_i$ are equal to a single newform~$f$.
\end{proof}

\subsection{Theta functions}\label{subsec:theta functions}

Let $\Lambda$ be an even integral lattice with respect to a bilinear form $\langle \cdot\,,\cdot\rangle$ (see \cref{subsec:lattices}). %
The \textbf{theta function} 
$\theta_{\Lambda}(q)$ of $\Lambda$ is the formal $q$-series 
\[
\theta_{\Lambda}(q) = \sum_{x \in \Lambda} q^{\frac12\langle x,x\rangle} = \sum_{n=0}^\infty \repnum{n}{\Lambda} q^n, \; \text{ where } \; \repnum{n}{\Lambda} = | \{ x \in \Lambda \; : \; \tfrac12\langle x,x\rangle = n \}|.
\]
It is well known that the theta function of $\Lambda$ is a modular form of weight $\tfrac{\rank(\Lambda)}{2}$ %
and level equal to the level of the lattice (\cref{def:dual_lattice}), with a certain character; see e.g.\ \cite[Theorem 20]{Ogg}.

We will be specifically focused on the case where $\Lambda = \O$ is an order in a definite quaternion algebra $\mc{A}$ equipped with the standard bilinear form (trace form) on $\mc{A}$, so that $a_n(\O)$ is the number of elements $x\in\O$ with $\nrd(x)=n$. In this case, the theta function of $\O$ is a modular form of weight $2$ and level %
$N = \nrd(\O^\sharp)^{-1}$ by \cref{lem:level_alternative}.  %

Recall from \cref{subsec:brandt} that $\mf{M}(\O)$ denotes the Brandt module of $\O$. The following two results are well known throughout the literature (see for instance \cite[\S 3.3]{Kohel}, or \cite[Theorem 2.14]{PizerComputing} and \cite[Proposition 2.23]{PizerComputing}), %
though they are typically stated for a more restrictive class of orders. Since we are working in a more general setting, we include proofs for completeness, though these proofs largely follow those appearing in the literature.

Fix the standard basis $\{v_1, \ldots, v_h\}$ for $\mf{M}(\O)$ corresponding to the enumeration $\{ [I_1]=[\O], [I_2], \ldots, [I_h] \}$ of the right ideal classes of $\O$ as in \cref{subsec:brandt}.

\begin{lemma}\label{lem:Theta pairing}
    Let $\O$ be an order in a definite rational quaternion algebra, and suppose  $N:=\nrd(\O^\sharp)^{-1}$. There exists a symmetric bilinear pairing
        \begin{align*}
            \Theta:\mathfrak{M}(\mathcal{O})\times\mathfrak{M}(\mathcal{O})&\to \mc{M}_2(\Gamma_0(N))\\
            \Theta(v,w)&:= \sum_{n=0}^\infty \langle v,B_nw\rangle q^n.
        \end{align*}
        In particular,
        $2\Theta(v_i, v_j)$ is the theta function of the lattice $I_jI_i^{-1}$ with normalized trace form \emph{(\cref{def:normalized trace form})}, and $2\Theta(v_1, v_1) = \theta_{\O}$. %
\end{lemma}
\begin{proof}    
   We begin by considering $\Theta$ only as a map valued in the formal power series ring $\RR\llbracket q\rrbracket$. The inner product on $\mf{M}(\O)$ is symmetric and bilinear, and $B_n$ is self-adjoint (\cref{lem:Bn self-adjoint}), so $\Theta$ is symmetric and bilinear. To prove that the image lies in $\mc{M}_2(\Gamma_0(N))$, it suffices to prove that $\Theta(v_i,v_j)\in \mc{M}_2(\Gamma_0(N))$ for all $1\leq i,j\leq h$.

    For arbitrary $i,j$, we have $\langle v_i,B_n v_j\rangle=w_iB_n(i,j)$, so that by \cref{eq:Bnij def}, $2\Theta(v_i,v_j)$ is the theta function of the lattice $I_jI_i^{-1}$ with normalized trace form; this proves the final sentence in the statement of the lemma, since we took $I_1=\O$.    
    Observe that the level of an integral quadratic form can be computed locally (for instance, because the dual lattice can be defined locally, and integrality of a quadratic form is a local property). Locally at each prime $\ell$, we can write $(I_i)_\ell=\alpha\O_\ell$ and $(I_j)_\ell=\beta\O_\ell$ for some $\alpha,\beta\in\mathcal{A}_\ell^\times$, so that $(I_jI_i^{-1})_\ell=\beta\O_\ell\alpha^{-1}$. Now let $u:=\frac{\nrd(\alpha)}{\nrd(I_i)}\frac{\nrd(I_j)}{\nrd(\beta)}$, which is in $\ZZ_\ell^\times$. Then the map $x\mapsto \beta x\alpha^{-1}$ %
    is an isometry from $\O_\ell$ to the lattice $(I_jI_i^{-1})_\ell$ equipped with the bilinear form $(x,y)\mapsto u\langle x,y\rangle_{I_jI_i^{-1}}$. Multiplying the bilinear form of a $\ZZ_\ell$-lattice by a unit does not affect the level. So taking all $\ell$ together, this proves that $I_jI_i^{-1}$ with the normalized trace form has the same level as $\O$.
    
    Since the level of $\O$ is equal to $N=\nrd(\O^\sharp)^{-1}$ by \cref{lem:level_alternative}, we have $\Theta(v_i,v_j)\in \mc{M}_2(\Gamma_1(N))$ with character $d\mapsto \left(\frac{\Delta}{d}\right)$ by \cite[Theorems 20, 20+]{Ogg} (see also \cite[40.4.5]{Voight}), where $\Delta=\det(\langle b_i,b_j\rangle_{I_jI_i^{-1}})_{1\leq i,j\leq 4}$ for a basis $(b_1,b_2,b_3,b_4)$ of $I_jI_i^{-1}$. Since $\Delta$ is a perfect square by \cref{lem:disc square}, the character is trivial and $\Theta(v_i,v_j)\in \mc{M}_2(\Gamma_0(N))$. 
\end{proof}

We next prove a compatibility between Brandt operators $B_n$ and Hecke operators $T_n$ for $n$ coprime to $N$. %
This compatibility does not hold more generally: even if $u$ is an eigenvector for all Brandt matrices in $\mf{B}(\O)$, $\Theta(u,u)$ may fail to be an eigenvector for Hecke operators $T_\ell$ with $\ell\mid N$. See \cref{lem:index r eigenvalues} for a counterexample.

\begin{lemma}
\label{lem:Brandt_Hecke_equal_spectra}
For all $v,w\in \mf{M}(\O)$ and all $n\geq 1$ with $\gcd(n,N)=1$, %
we have \[T_n \Theta(v,w) = \Theta(v, B_n w).\] 
\end{lemma}

Before we begin the proof we set up some notation.
Since $\mf{B}(\O)$ is a commutative algebra (\cref{thm:Brandt algebra commutative}) of operators that are self-adjoint with respect to the inner product on $\mf{M}(\O)$, there exists a simultaneous eigenbasis $u_1,\ldots,u_h\in \mf{M}(\O)$ consisting of vectors that are pairwise orthogonal with respect to the inner product. We can normalize this basis so that $\langle u_i,u_i\rangle=1$ for all $i$; we will refer to this basis as a \textbf{normalized Brandt eigenbasis}, and each element of the normalized Brandt eigenbasis will be called a \textbf{Brandt eigenvector}.
In the case of a submaximal $\O$ in a maximal $\mf{O}$ with $[\mf{O}:\O]=r$, we will further impose the condition that each $u_i$ in a normalized Brandt eigenbasis is an eigenvector of the operator $W_{\mf{r}}$ defined by \cref{eq:Wr}; this can be ensured by \cref{lem:W commutes}.
Observe that for any $i\neq j$ we have
\begin{equation}\label{eq:Theta ui uj zero}
    \Theta(u_i,u_j)=\sum_{n=0}^\infty \langle u_i,B_nu_j\rangle q^n=0,
\end{equation}
because for all $n\geq 0$, $B_nu_j$ is a scalar multiple of $u_j$ and hence orthogonal to $u_i$.
\begin{proof}[Proof of \emph{\cref{lem:Brandt_Hecke_equal_spectra}}]
    Let $u_1,\ldots,u_h$ be a Brandt eigenbasis for $\mf{M}(\O)$. For all $n\geq 0$ define the linear functional $\lambda_n$ by $B_nu_i=\lambda_n(u_i)u_i$ for each $i$. Let $v=a_1u_1+\cdots +a_hu_h$ and $w=b_1u_1+\cdots+b_hu_h$. By bilinearity together with \cref{eq:Theta ui uj zero}, we have
    \[\Theta(v,w)=\sum_{i=1}^h a_ib_i\Theta(u_i,u_i),\qquad \Theta(v,B_n w)=\sum_{i=1}^h a_ib_i\lambda_n(u_i)\Theta(u_i,u_i).\]
    So it suffices to prove the following: for any Brandt eigenvector $u\in\mf{M}(\O)$, and any $n$ coprime to $N$, if we set 
    $$
    f_u:=\frac{\Theta(u,u)}{\langle u,u\rangle}=\sum_{n=0}^\infty \lambda_n(u)q^n,$$
    then $T_nf_u=\lambda_n(u)f_u$. Indeed, since $\langle u_i,u_i\rangle=1$ for each $i$, we have that $\Theta(u_i,u_i)=f_{u_i}$ and hence by linearity, 
    $$
    T_n\Theta(v,w)=
    \sum_{i=1}^h a_ib_iT_n f_{u_i}=\sum_{i=1}^h a_ib_i\lambda_n(u_i) f_{u_i}
     = \Theta(v, B_n w).  $$

    First consider the case $n=\ell$ for some prime $\ell\nmid N$. As seen above,  $f_u\in \mathcal{M}_2(\Gamma_0(N))$, so by %
    \cref{eq:Hecke action on q expansion}, $T_\ell$ acts on $f_u$ as
    \begin{equation}\label{tlfu}
    T_\ell f_u = \sum_{m = 0}^\infty \left(\lambda_{\ell m}(u)+\ell\lambda_{\tfrac{m}{\ell}}(u)\right) q^m.
    \end{equation}
    Here for convenience we set $\lambda_c(u)=0$ for any $c$ that is not a non-negative integer: that is, if $\ell\nmid m$ then the coefficient of $q^m$ in \cref{tlfu} above is $\lambda_{\ell m}(u)$. Now for each $m\geq 1$ write $m=\ell^km'$ with $m'$ coprime to $\ell$. Since $B_{\ell^k m'} = B_{\ell^k} B_{m'}$ by \cref{lem:coprime commute} %
    we have the equality $\lambda_{\ell^k m'}(u)=\lambda_{\ell^k}(u)\lambda_{m'}(u)$ of Brandt eigenvalues (and a similar decomposition for $\lambda_{\frac m\ell}(u)$ if $\ell\mid m$), so the coefficient of $q^m$ in \cref{tlfu} equals
    \[\lambda_{m'}(u)\left(\lambda_{\ell^{k+1}}(u)+\ell\lambda_{\ell^{k-1}}(u)\right).\]
    Now by an immediate simplification if $k=0$, and by \cref{lem:classical Brandt recurrence} if $k\geq 1$, %
    this in turn equals
    \[\lambda_{m'}(u)\lambda_{\ell^k}(u)\lambda_{\ell}(u)=\lambda_\ell(u)\lambda_{m}(u).\]
    Hence $T_\ell f_u=\lambda_\ell(u)f_u$.

    Now for the Hecke operator $T_{\ell^k}$ with $\ell \nmid N$ and $k \ge 1$,  we induct on $k$. Using the relation $T_\ell T_{\ell^k} = T_{\ell^{k+1}} + \ell T_{\ell^{k-1}}$ and the inductive assumption, we obtain 
    \begin{align*}
    T_{\ell^{k+1}} f_u & = (T_\ell T_{\ell^k}-\ell T_{\ell^{k-1}}) f_u \\
    & = \left( \lambda_\ell(u) \lambda_{\ell^k}(u) - \ell \lambda_{\ell^{k-1}}(u) \right) f_u.
    \end{align*}
    By the Brandt matrix relation $B_\ell B_{\ell^k} = B_{\ell^{k+1}} + \ell B_{\ell^{k-1}}$, we can rewrite this as $\lambda_{\ell^{k+1}}(u)f_u$, completing the induction. Finally, for the general case for $n=\prod_{i} \ell_i^{k_i}$ with $\gcd(n, N) = 1$, we use multiplicativity of Hecke operators and of Brandt matrices.
\end{proof}

\subsection{The cuspidal subspace}\label{subsec:cusp subspace}

Recall from \cref{sec:genus subspace} that the Brandt module of an order $\O\subseteq\mc{A}$ contains a \emph{genus subspace} $\mf{M}_{\mathrm{G}}(\O)$.
Let $\mf{M}_{\mathrm{G}}(\O)^\perp$ denote the subspace of $\mf{M}(\O)$ of vectors orthogonal to $\mf{M}_{\mathrm{G}}(\O)$ with respect to the Brandt module inner product. We call $\mf{M}_{\mathrm{G}}(\O)^\perp$ the \textbf{cuspidal subspace} of $\mf{M}(\O)$, and refer to %
elements of $\mf{M}_{\mathrm{G}}(\O)^\perp$ as being cuspidal, motivated by the following lemma. 
As with \cref{lem:Brandt_Hecke_equal_spectra,lem:Theta pairing}, this lemma also occurs in the literature for a more restricted class of orders (cf.~\cite[Theorem 2.21]{PizerComputing}, \cite[Proposition 5.27]{Pizerp2M}, and \cite[\S 3.3]{Kohel}), and we provide a proof of the more general case for completeness, following the proofs in the literature.
\begin{lemma}\label{lem:cusp subspace}
    The map $\Theta$ from \emph{\cref{lem:Theta pairing}} sends $\mf{M}(\O)\times \mf{M}_{\mathrm{G}}(\O)^\perp$ and $\mf{M}_{\mathrm{G}}(\O)^\perp\times \mf{M}(\O)$ into the space $\mc{S}_2(\Gamma_0(N))$ of cusp forms in $\mc{M}_2(\Gamma_0(N))$.
\end{lemma}
\begin{proof}
    By symmetry of $\Theta$ it suffices to prove $\mf{M}(\O)\times \mf{M}_{\mathrm{G}}(\O)^\perp$ is sent to cusp forms. A straightforward computation shows that 
    $\mf{M}_{\mathrm{G}}(\O)^\perp$ is spanned by vectors of the form $v_i-v_j$ for $1\leq i,j\leq h$ satisfying $I_j\in\mathrm{Gen}(I_i)$ (each of these vectors is orthogonal to every element of $\mf{M}_{\mathrm{G}}(\O)$ by construction, and a counting argument shows that together with $\mf{M}_{\mathrm{G}}(\O)$ they span the full Brandt module). So by bilinearity of $\Theta$, it suffices to prove that $\Theta(v_k,v_i)-\Theta(v_k,v_j)$ is a cusp form for all $1\leq i,j,k\leq h$ with $I_j\in\mathrm{Gen}(I_i)$.

    Since $I_i$ and $I_j$ (with their associated normalized trace forms) are in the same genus, $I_iI_k^{-1}$ and $I_jI_k^{-1}$ are also in the same genus. To see this, let $\ell$ be an arbitrary prime, and let $\phi:(I_i)_\ell\to (I_j)_\ell$ be an isometry. Write $(I_k)_\ell=\gamma\O_\ell$ for some $\gamma\in\mc{A}_\ell^\times$. Then for any $x,y\in (I_iI_k^{-1})_\ell$, we have $x\gamma,y\gamma\in (I_i)_\ell$, so
    \begin{align*}
        \frac{\trd((\phi(x\gamma)\gamma^{-1})\overline{(\phi(y\gamma)\gamma^{-1})})}{\nrd(I_jI_k^{-1})}&=\frac{\nrd(\gamma)^{-1}}{\nrd(I_k)^{-1}}\frac{\trd(\phi(x\gamma)\overline{\phi(y\gamma)})}{\nrd(I_j)}\\
        &=\frac{\nrd(\gamma)^{-1}}{\nrd(I_k)^{-1}}\frac{\trd((x\gamma)\overline{(y\gamma)})}{\nrd(I_i)}\\
        &=\frac{\trd(x\overline{y})}{\nrd(I_iI_k^{-1})}.
    \end{align*}
    Thus $x\mapsto \phi(x\gamma)\gamma^{-1}$ defines an isometry $(I_iI_k^{-1})_\ell\to (I_jI_k^{-1})_\ell$, with bilinear forms respectively induced by the normalized trace forms on $I_iI_k^{-1}$ and $I_jI_k^{-1}$.
    It is a classical result that if two lattices are in the same genus then the difference of their theta functions is a cusp form (see e.g.~\cite[p.~577]{Siegel}). So we can conclude that $\Theta(v_k,v_i)-\Theta(v_k,v_j)$, and hence every modular form in the image of $\mf{M}(\O)\times \mf{M}_{\mathrm{G}}(\O)^\perp$ under $\Theta$, is a cusp form.
\end{proof}

The space of genus vectors for a submaximal order $\O$ was determined in \cref{sec:genus subspace}: we found that $\mf{M}_{\mathrm{G}}(\O)$ is typically $1$-dimensional (spanned by the vector $e$), but in certain special cases it is $2$-dimensional (spanned by $e$ and $e'$). As a consequence of \cref{lem:cusp subspace}, we have that $\Theta(u,u)\in\mc{S}_2(\Gamma_0(N))$ for any Brandt eigenvector $u$ orthogonal to $e$ (and to $e'$ if defined).

\section{Decomposing theta functions of quaternion orders}\label{sec:map to M2}

For the remainder of the paper, $\mc{A}$ is a \emph{definite} quaternion algebra over $\QQ$. As in the discussion immediately following \cref{lem:Brandt_Hecke_equal_spectra}, we define a \textbf{normalized Brandt eigenbasis} of $\mf{M}(\O)$ to be a basis $u_1,\ldots,u_h$ which is orthonormal with respect to the standard inner product on $\mf{M}(\O)$, and such that each $u_i$ is an eigenvector for every operator in $\mf{B}(\O)$; such a basis exists by \cref{thm:Brandt algebra commutative}.
If $u$ is a Brandt eigenvector, we write $\lambda_n(u)$ to denote the eigenvalue of $u$ under the action of the Brandt matrix $B_n$.

\subsection{Theta function decomposition}

Since \cref{thm:Brandt algebra commutative} guarantees the existence of a normalized Brandt eigenbasis, the following result is a more precise version of \cref{thm:intro eigenform_coeffs}. We will use this result to decompose $\theta_\O$ into individual modular forms that are easier to analyze, from which we can obtain explicit bounds on the number of elements of each norm in a quaternion order.

\begin{theorem}\label{thm:eigenform_coeffs}
    Let $\O$ be an order in a definite quaternion algebra $\mathcal{A}$ over $\QQ$, $N:=\nrd(\O^\sharp)^{-1}$, and let $h$ be the class number of $\O$. %
    Let $u_1,\ldots,u_h$ be a normalized Brandt eigenbasis of $\mf{M}(\O)$, and set $f_i:=\Theta(u_i,u_i)$ for each $i=1,\ldots,h$. Then each $f_i$ is a normalized $\Tan$-eigenvector, and
    \[
    \theta_\O(q)=c_1f_1(q)+\cdots+c_hf_h(q)
    \]
    for some non-negative real numbers $c_1,\ldots,c_h$ with $c_1+\cdots+c_h=|\mathcal{O}^\times|$.
\end{theorem}

\begin{proof}
    By \cref{lem:Brandt_Hecke_equal_spectra}, we can conclude that $f_i:=\Theta(u_i,u_i)$ is a $\Tan$-eigenvector for all $i=1,\ldots,h$. The coefficient of $q$ in $f_i$ is $\langle u_i,u_i\rangle=1$ because $B_1$ is the identity matrix, so $f_i$ is normalized. We are left to show that we can write %
    \[\theta_\O=c_1f_1+\cdots+c_hf_h\]
    for non-negative real coefficients $c_1,\ldots,c_h$ satisfying $c_1+\cdots+c_h=|\O^\times|$.

    Let $v_1,\ldots,v_h$ be the standard basis of $\mathfrak{M}(\O)$, ordered so that $I_1=\O$. By \cref{lem:Theta pairing}, 
    we have $2\Theta(v_1,v_1)=\theta_\O$.
    Let $(c_{ij})_{i,j}\in \GL_h(\RR)$ be a change of basis matrix so that
    \[v_i=c_{i1}u_1+\cdots+c_{ih}u_h\]
    for all $i=1,\ldots,h$. By \cref{eq:Theta ui uj zero} we have $\Theta(u_i,u_j)=0$ for all $i\neq j$, so together with bilinearity and symmetry
    of $\Theta$ we have
    \begin{equation} \label{eq: theta_decom_in_proof}
        \theta_\O=2\Theta(v_1,v_1) =2\sum_{i=1}^hc_{1i}^2\Theta(u_i,u_i)=2\sum_{i=1}^hc_{1i}^2f_i,
    \end{equation}
    so $c_i:=2c_{1i}^2\geq 0$. We find $|\mathcal{O}^\times|=\sum_i c_i$ by comparing the coefficient of $q$ on each side.
\end{proof}

\begin{remark}
    Even though the conclusion of \cref{thm:eigenform_coeffs} only involves Hecke operators $T_n$ with~$n$ coprime to $N$, the full commutativity of $\mf{B}(\O)$ is still essential to the proof. If $u_1,\ldots,u_h$ forms an orthogonal basis of eigenvectors of $B_n$ for all $n$ coprime to $N$, but are not eigenvectors for all $B_n$, then \cref{eq:Theta ui uj zero} may not necessarily hold.
\end{remark}

We will apply \cref{thm:eigenform_coeffs} to maximal and submaximal orders $\O$. For this, it will be helpful to have a more refined statement, obtained by separating the Brandt eigenvectors into two types: those in the genus subspace $\mf{M}_{\mathrm{G}}(\O)$, and those orthogonal to the genus subspace.

\begin{corollary}\label{cor:max submax eigenform_coeffs}
    Let $\O$ be a maximal or submaximal order in a definite quaternion algebra $\mathcal{A}$ over $\QQ$, and $N:=\disc(\O)$. There exist $g\in\mc{M}_2(\Gamma_0(N))$ %
    and $\mf{B}(\O)$-eigenvectors $u_1,\ldots,u_m\in\mf{M}(\O)$ for some $m\geq 0$, %
    satisfying the following properties.
    \begin{enumerate}[$(a)$]
        \item The modular form $g$ satisfies
         \[
        g(q) =
        \begin{cases}
        d_1\frac{\Theta(e,e)}{\langle e,e\rangle}+d_2\frac{\Theta(e',e')}{\langle e',e'\rangle}
        &
        \text{if }
        [\mathfrak O:\mathcal O]=p\neq 2\text{ and }p\mid\disc(\mc{A}),
        \\[1em]
        d_1 \frac{\Theta(e,e)}{\langle e,e\rangle}
        &
        \text{otherwise,}
        \end{cases}
        \]
        for some $d_1,d_2\geq 0$, where $e,e'$ are as defined in \emph{\cref{eqn:e def'n}} and \emph{\cref{eqn:e' def'n}}. %
        \item For $1\leq i\leq m$, $f_i:=\Theta(u_i,u_i)$ is a normalized $\Tan$-eigenvector in $\mc{S}_2(\Gamma_0(N))$. %
        \item There exist $c_1,\ldots,c_m\geq 0$ with $c_1+\cdots+c_m\leq |\mathcal{O}^\times|$ such that
        \begin{align}\label{eq:decomp with genus part}
            \theta_\O(q)-g(q)=c_1f_1(q)+\cdots+c_mf_m(q).
        \end{align}
    \end{enumerate}
\end{corollary}

\begin{proof}
    We use \cref{lem:level equals disc} to obtain $N=\nrd(\O^\sharp)^{-1}=\disc(\O)$. 
    If $[\mf{O}:\O]$ is not equal to an odd ramified prime of $\mc{A}$, then we have from   \Cref{lem:1dim genus} that  $\mf{M}(\O)$ has only one (linearly independent) genus vector, namely $e$. In the other case, by \cref{lem:2dim genus}, it has exactly two (linearly independent) genus vectors, namely $e$ and~$e'$. 
    
    Since $e$ and $e'$ are $\mf{B}(\O)$-eigenvectors (see \Cref{lem:e eigen} and \Cref{lem:e' eigen}), we can extend these to a Brandt eigenbasis for $\mf{M}(\O)$. \Cref{lem:cusp subspace} shows that every Brandt eigenbasis vector outside the genus subspace is cuspidal. We normalize the Brandt eigenbasis and apply \cref{thm:eigenform_coeffs}; the desired result is obtained by re-labeling.
\end{proof}

\begin{remark}
    It is possible to have $m=0$ in \cref{cor:max submax eigenform_coeffs}, for instance whenever $\O$ is a quaternion order of class number $1$. In this case $(b)$ is vacuous and $(c)$ states that $\theta_\O=g$. 
    
    The form $g$ should be thought of as the ``Eisenstein component'' of $\theta_{\mathcal{O}}$. We avoid this terminology as we have not shown in all cases that it belongs to the Eisenstein subspace in $\mc{M}_2(\Gamma_0(N))$. In particular, while $\Theta(e,e)$ (and $\Theta(e',e')$ when it exists) is a $\Tan$-eigenvector, it is not necessarily an  eigenvector for the full Hecke algebra $\T$. Thus it may have a nonzero cuspidal contribution, provided that the coefficient of $q^n$ for this cusp form is zero for all $n$ coprime to $\disc(\O)$.
\end{remark}

We compute the coefficients of $g$ in \cref{sec:eisenstein part}, and bound the coefficients of each cusp form $f_i$ in \cref{sec:cusp part}.

\subsection{Contribution from genus vectors}\label{sec:eisenstein part}

As established in \cref{cor:max submax eigenform_coeffs}, the contribution of the genus subspace to the theta function $\theta_{\O}$ is captured entirely by the modular form $g$. For maximal orders and most submaximal orders, this contribution comes only from the standard genus vector $e$. The unique exception occurs for submaximal orders of index odd $p \mid \disc(\mc{A})$, which have a second genus vector $e'$. 

Our goal in this section is to explicitly compute and bound the  coefficients of $g$, treating the exceptional two-vector case separately in \cref{sec:p2}. Let $\phi(n)=n\prod_{\ell\mid n}\left(1-\frac1\ell\right)$ be Euler's totient function. %

\begin{proposition}\label{cor: submaximal eisenstein bound}
    Let $\O$ be maximal or submaximal, $\ell\nmid \disc(\mc{A})$ prime, and $k\geq 1$. Let $g(q)=\sum g_n(\O)q^n$ be as defined in \emph{\cref{cor:max submax eigenform_coeffs}}. If $\O$ is maximal, then
    \begin{equation} \label{eq: submaximal eisenstein bound 1}
        g_{\ell^k}:=g_{\ell^k}(\O)=\frac{24}{\phi(\disc(\mc{A}))}\frac{\ell^{k+1}-1}{\ell-1},
    \end{equation}
    and if $\O$ is submaximal, then
    \begin{equation} \label{eq: submaximal eisenstein bound 2}
        0\leq g_{\ell^k}\leq \frac{16}{\phi(\disc(\mc{A}))}\frac{\ell^{k+1}-1}{\ell-1}.
    \end{equation}
\end{proposition}

We prove this in three stages. We first compute $d_1$, which requires computing the eigenvalue of $e$ under $B_0$. We then compute the eigenvalue of $e$ under $B_{\ell^k}$. Finally we consider the contribution from $e'$ when it exists. %

Note that $e=B_0v_i$ for any standard basis vector $v_i$; it is a column eigenvector for $B_n$ for all $n\geq 0$. 

\begin{remark}
    This description assumes we treat the Brandt module as column vectors with $v\mapsto B_n v$.  We have $B_nv=v'$ if and only if $(v^tD)B_n=(v')^tD$, where recall $D$ is the diagonal matrix with $D(i,i)=w_i$ as above. %
    Thus the genus \emph{row} vector is the constant vector $e^tD=\frac12(1,1,\ldots,1)$, which up to scaling is the convention in \cite{Voight}.
\end{remark}

We now consider the coefficients of $\frac{\Theta(e,e)}{\langle e,e\rangle}=\sum_{n=0}^\infty \lambda_n(e)q^n$, where $\lambda_n(e)$ denotes the eigenvalue of $e$ under the action of $B_n$. Since $B_1$ is the identity matrix, we have $\lambda_1(e)=1$.
\subsubsection{Constant term}
A direct computation gives us $\lambda_0(e)=\sum \frac{1}{2w_i}$. We can evaluate this using the Eichler mass formula \cite[Main Theorem 25.3.19]{Voight}: 
\[\lambda_0(e)=\sum_{i=1}^h \frac{1}{2w_i}=\frac{\disc(\O)}{24}\prod_{\ell \mid \disc(\O)} \lambda(\O,\ell),\]
where for $\ell\mid\disc(\O)$ we have
\[\lambda(\O,\ell)=\begin{cases}
    1+\frac1\ell & \text{if $\O_\ell$ is residually split,} \\
    1-\frac1\ell & \text{if $\O_\ell$ is residually inert,} \\
    1-\frac1{\ell^2} & \text{if $\O_\ell$ is residually ramified}
\end{cases}\]
(see \cite[Definition 24.3.2]{Voight}). If $\O$ is maximal and $\ell\mid\disc(\O)$ then $\mc{A}_\ell$ is a division algebra and $\O_\ell$ is residually inert.
Now consider the classification of submaximal orders in \cref{lem:max_suborders}, with $\qprime$ being a split prime of $\mc{A}$ and $p$ a ramified prime of $\mc{A}$. If $\O$ is an Eichler order of index $r$ then $\O_r$ is residually split; if $\O$ has index $p$ then $\O_p$ is residually ramified; in every other case, $\ell\mid\disc(\O)$ implies $\O_\ell$ is residually inert. We therefore have the following for any maximal or submaximal order $\O\subseteq\mf{O}$:
\begin{align}\label{eq:constant term}
    \lambda_0(e)=\frac{\phi(\disc(\mc{A}))}{24}\cdot \begin{cases}
    1 & \text{if }[\mf{O}:\O]=1,\\
    \qprime+1 & \text{if }[\mf{O}:\O]=\qprime,\\
    p+1 & \text{if }[\mf{O}:\O]=p,\\
    \qprime(\qprime-1) & \text{if }[\mf{O}:\O]=\qprime^2,\\
    p^2 & \text{if }[\mf{O}:\O]=p^2.
\end{cases}
\end{align}

\subsubsection{Nonconstant terms}
Now suppose $n>1$ is coprime to $\disc(\mc{A})$. For each submaximal $\O$ we will determine $\lambda_n(e)$.
\begin{lemma}\label{lem: Eisenstein coefficient ideal count}
    For $n\geq 1$, 
    \begin{equation}\label{eq:lambda_n(e)}
    \lambda_n(e)=|\{J\subseteq_n I_j\}|,
    \end{equation}
    and the quantity does not depend on $j$. Hence, $\lambda_{mn}(e)=\lambda_m(e)\lambda_n(e)$ provided that $m,n$ are coprime, and the value of $\lambda_{\ell^k}(e)$ depends only on the isomorphism type of the completion $\O_\ell$.
\end{lemma}
\begin{proof}
Since we have $e^tB^t_n=\lambda_n(e)e^t$, we obtain $e^tB^t_nD=\lambda_n(e)(e^tD)$, where as before~$D$ is the diagonal matrix with $D(i, i) = w_i$, the matrix giving the inner product, and $e^tD=\frac12(1,1,\ldots,1)$ is the genus row vector. Since $B_n$ is self-adjoint by \cref{lem:Bn self-adjoint}, we have that $DB_n=B_n^tD$, and so it follows that $(e^tD)B_n=\lambda_n(e)(e^tD)$. This implies that $\lambda_n(e)$ is equal to the sum of entries in a column of $B_n$, and so    
    \cref{eq:lambda_n(e)} follows from \cref{lem: Eisenstein coefficient ideal count'}. The other assertions follow easily from \cref{eq:lambda_n(e)} together with the local-global dictionary. %
\end{proof}
By the multiplicativity of $\lambda_n(e)$ in \cref{lem: Eisenstein coefficient ideal count}, it suffices to determine $\lambda_n(e)$ for prime powers. This is carried out in the following lemma for powers of a prime $\ell$ in the cases of maximal and submaximal orders.
\begin{lemma}\label{lem:eisenstein coeffs}
    Let $\O$ be maximal or submaximal, $\ell$ and $r$ primes not dividing $\disc(\mc{A})$, and $k\geq 1$. %
    \begin{enumerate}[$(a)$]
        \item If $\ell\nmid\disc(\mc{A})$ and $\O_\ell$ is maximal, then $\lambda_{\ell^k}(e)=\frac{\ell^{k+1}-1}{\ell-1}$.
        \item If $[\mf{O}:\O]=\qprime$, then $\lambda_{\qprime^k}(e)=2\frac{\qprime^{k+1}-1}{\qprime-1}-1$.
        \item If $[\mf{O}:\O]=\qprime^2$, then $\lambda_{\qprime^k}(e)=\qprime^k-\qprime$.
    \end{enumerate}
\end{lemma}
\begin{proof}
Part $(a)$ is classical. Indeed, since $\lambda_1(e)=1$ and $\lambda_\ell(e) = \ell + 1$ for all primes $\ell \nmid \disc(\mc{A})$, %
using the relation $\lambda_{\ell^{k+1}}(e) = \lambda_{\ell^k}(e)\lambda_\ell(e) - \ell\lambda_{\ell^{k-1}}(e)$ (cf.\ \Cref{lem:classical Brandt recurrence}),    we can obtain by induction that $\lambda_{\ell^k}(e) = \frac{\ell^{k+1}-1}{\ell-1}$.

For $(b)$: note that this case is the one studied in \cref{s: index_ell}, and so \cref{lem:countideals} applies. Indeed, the proof of this lemma  shows that there are $2r+1$ right ideals of norm $r$ in $\O$. This implies that $\lambda_r(e)=2r+1$ using \cref{lem: Eisenstein coefficient ideal count}. We also observe that $W_{\mf{r}}e=e$ by definition of $W_{\mf{r}}$ (see \cref{s: index_ell} for its definition). We obtain the relation $\lambda_{r^{k+1}}(e) = (r+1)\lambda_{r^k}(e) - r\lambda_{r^{k-1}}(e)$ by the recurrence relation in \cref{eq:recurrence}, and so the assertion follows. 

For $(c)$: by the definition of $E$, since $e$ has strictly positive entries (in the standard basis), $Ee$ %
also has strictly positive entries. Hence the Fourier expansion of $\Theta(Ee,Ee)$ has nonzero constant term, and $Ee$ is not in the cuspidal subspace of $\mf{M}(\mf{O})$. Now let $\widetilde{e}\in\mf{M}(\mf{O})$ be the genus vector for $\mf{O}$, and take $n\geq 1$ such that $\lambda_n(\widetilde{e})$ is distinct from $\lambda_n(u)$ for all cuspidal Brandt eigenvectors $u\in\mf{M}(\mf{O})$; such $n$ exists, for example, by applying part $(a)$ and taking $k$ sufficiently large so that $\lambda_{\ell^k}(\widetilde{e})$ is larger than the Deligne bound. By \cref{prop:Eu eigenvector}, $Ee$ is an eigenvector for $\widetilde{B}_n$. Since $Ee$ is not in the cuspidal subspace, and $\widetilde{e}$ has distinct eigenvalues from all cuspidal eigenvectors, this forces $Ee$ to be a multiple of $\widetilde{e}$.
Hence $\lambda_{r^k}(Ee)=\frac{r^{k+1}-1}{r-1}$ by part $(a)$, and applying \cref{prop:Eu eigenvector} again we find that
    \[\lambda_{r^k}(e)=(r^2-r)\frac{r^{k-1}-1}{r-1}=r^k-r.\qedhere\]
\end{proof}

We have now determined $\lambda_n(e)$ for all $n$ coprime to $\disc(\mc{A})$. As discussed above, certain cases admit a second genus vector $e'$. We compute the corresponding eigenvalues $\lambda_n(e')$ in the following subsection.

\subsubsection{Two-dimensional genus space}\label{sec:p2}

Let $\O$ be submaximal of index $p>2$ ramified in $\mc{A}$. %
 This is the special case of two genus vectors $e$ and $e'$ as was mentioned above (see \cref{lem:2dim genus}). %
The proof of \cite[Theorem 5.34(b)]{Pizerp2M} shows that $\lambda_n(e')=\left(\frac{n}{p}\right)\lambda_n(e)$ for all $n\geq 0$, where $\left(\frac{n}{p}\right)$ is the Legendre symbol. In particular, $\lambda_0(e')=0$. We show that $\Theta(e,e)$ and $\Theta(e',e')$ contribute equally  to $\theta_\O$. More precisely, $d_1=d_2$ and $\langle e, e \rangle = \langle e', e' \rangle$ %
in $g(q)=d_1\frac{\Theta(e,e)}{\langle e,e\rangle}+d_2\frac{\Theta(e',e')}{\langle e',e'\rangle}$ %
(see \cref{cor:max submax eigenform_coeffs}). %

\begin{lemma} \label{lem:c_1=c_2_two_dimensional_genus}
    The coefficients of the normalized non-cusp forms in $\theta_\O$ are equal: $d_1 = d_2$.
\end{lemma}

\begin{proof}
Let $v_1$ be the first vector in the standard basis of $\mathfrak{M}(\O)$, ordered so that $I_1=\O$. 
    For a Brandt eigenbasis $u_1=e, u_2=e', u_3, \ldots, u_h$, we write
    \[ v_1 = a_1 e + a_2 e' + \sum_{i=3}^h a_i u_i, \]
    and we have $a_1 = \frac{\langle v_1, e \rangle}{\langle e, e \rangle}$ and $a_2 = \frac{\langle v_1, e' \rangle}{\langle e', e' \rangle}$ by orthogonality. As in \cref{eq: theta_decom_in_proof}, we write using the decomposition of the theta function,
    \[ \theta_\O = 2\Theta(v_1,v_1) = 2a_1^2 \Theta(e,e) + 2a_2^2 \Theta( e', e') + \cdots + 2a_h^2\Theta(u_h,u_h), \]
    and so 
    \[ d_1 = 2a_1^2 \langle e, e \rangle = \frac{2\langle v_1, e \rangle^2}{\langle e, e \rangle} \quad \text{and} \quad d_2 = 2a_2^2 \langle e', e' \rangle = \frac{2\langle v_1, e' \rangle^2}{\langle e', e' \rangle}. \]
    By direct computation of the inner products, the negative signs in $e'$ are squared, implying that $\langle e, e \rangle = \langle e', e' \rangle$. Furthermore, since $\langle v_1, e \rangle = v_1^t D e$, where $D$ is the diagonal matrix with $D(i,i) = w_i$, we obtain that 
    \[ \langle v_1, e \rangle = w_1\left(\frac{1}{2w_1}\right) = \frac{1}{2}. \]
    Similarly, $\langle v_1, e' \rangle = w_1\left(\frac{1}{2w_1}\right) = \frac{1}{2}$. 
    We now see that $\langle e, e \rangle = \langle e', e' \rangle$ and $\langle v_1, e \rangle = \langle v_1, e' \rangle$. Therefore, $d_1 = d_2$, which proves the assertion. %
\end{proof}

\subsubsection{Proof of \emph{\cref{cor: submaximal eisenstein bound}}}

We now compute the non-cuspidal part's contribution for maximal orders, and bound the contribution for all submaximal orders. Recall first that the $q$-expansions of the cusp forms have zero constant term. Recall also that there is only one case, namely $[\mf{O}:\O]=p\neq 2$, for which the genus space is $2$-dimensional.

    First suppose $\O$ is maximal. Comparing constant terms in \cref{eq:decomp with genus part}, we have $1=d_1\lambda_0(e)$, so that $d_1=\frac{24}{\phi(\disc(\mc{A}))}$ by \cref{eq:constant term}. Then $g_{\ell^k}=d_1\lambda_{\ell^k}(e)$ is determined by \cref{lem:eisenstein coeffs}$(a)$, and \cref{eq: submaximal eisenstein bound 1} holds as desired.

    Suppose now that $\O$ is submaximal, so that it is one of the four types of orders listed in \cref{lem:max_suborders}. We first consider the case that $\ell$ divides $\disc(\O)$. Since $\ell\nmid\disc(\mc{A})$, this implies $\ell=r$ is a split prime of $\mc{A}$ and either $[\mf{O}:\O]=\qprime$ or $\qprime^2$. By \cref{lem:eisenstein coeffs} we therefore have $\lambda_{r^k}(e)=2\frac{r^{k+1}-1}{r-1}-1$ or $\lambda_{r^k}(e)=r^k-r$, respectively. %
    Observing that 
    \[\frac{r^{k}-r}{r(r-1)}\leq \frac{2\frac{r^{k+1}-1}{r-1}-1}{r+1}\leq \frac{2}{r+1}\cdot\frac{r^{k+1}-1}{r-1}\]
    for all primes $r$ and all $k\geq 0$, we can conclude in either case (using the corresponding value of $\lambda_0(e)$ from \cref{eq:constant term}) that
    \begin{align*}
        g_{r^k}=\frac{\lambda_{r^k}(e)}{\lambda_0(e)}\leq\frac{24}{\phi(\disc(\mc{A}))}\cdot\frac{2}{r+1}\cdot\frac{r^{k+1}-1}{r-1},
    \end{align*}
    giving the desired bound in \cref{eq: submaximal eisenstein bound 2} since $r\geq 2$.

    Now suppose $\ell\nmid\disc(\O)$, so $\O_\ell$ is maximal. We first consider the case that $[\mf{O}:\O]$ is not an odd ramified prime of $\mc{A}$, so $\mf{M}_{\mathrm{G}}(\O)$ is one-dimensional.
    From \cref{eq:constant term} we see that $\lambda_0(e)\geq \frac{\phi(\disc(\mc{A}))}{24}\cdot 2$ for any submaximal order $\O$. It follows from \cref{lem:eisenstein coeffs}(a) that
    \begin{align*}
        g_{\ell^k}=\frac{\lambda_{\ell^k}(e)}{\lambda_0(e)}\leq \frac{12}{\phi(\disc(\mc{A}))}\cdot\frac{\ell^{k+1}-1}{\ell-1},
    \end{align*}
    which is strictly smaller than the desired bound in \cref{eq: submaximal eisenstein bound 2}.

    Finally suppose $[\mf{O}:\O]=p\neq 2$ is ramified in $\mc{A}$, so that $\mf{M}_{\mathrm{G}}(\O)$ is two-dimensional. By \cref{lem:c_1=c_2_two_dimensional_genus}, we have that $g(q)=d_1\left(\frac{\Theta(e,e)}{\langle e,e\rangle}+\frac{\Theta(e',e')}{\langle e',e'\rangle}\right)$, %
    and so comparing constant terms in \cref{eq:decomp with genus part}, we again have $1=d_1\lambda_0(e)$. Recall that $\lambda_n(e')=\left(\frac{n}{p}\right)\lambda_n(e)$. Thus 
 \begin{align*}
        g_{\ell^k} \;=\; \frac{\lambda_{\ell^k}(e) + \lambda_{\ell^k}(e')}{\lambda_0(e)} \;=\; \frac{\lambda_{\ell^k}(e)}{\lambda_0(e)}\left(1+\left(\frac{\ell^k}{p}\right)\right)\leq \frac{2\lambda_{\ell^k}(e)}{\lambda_0(e)}.
    \end{align*}
    Using \cref{eq:constant term} to compute $\lambda_0(e)=\frac{\phi(\disc(\mc{A}))}{24}(p+1)$ and \cref{lem:eisenstein coeffs}$(a)$ to compute $\lambda_{\ell^k}(e)$, we conclude
    \[g_{\ell^k} \leq \frac{2}{p+1}\cdot\frac{24}{\phi(\disc(\mc{A}))}\cdot\frac{\ell^{k+1}-1}{\ell-1}.\]
    This implies \cref{eq: submaximal eisenstein bound 2},  because $p\geq 3$. Since we have seen that the desired bounds in \cref{cor: submaximal eisenstein bound} hold in all cases, the proof is complete.

\begin{remark}
    If we assume $\ell\geq 3$, then the same proof shows that 
    \[g_{\ell^k}\leq \frac{12}{\phi(\disc(\mc{A}))}\cdot\frac{\ell^{k+1}-1}{\ell-1}\] 
    for any submaximal order $\O$.
    If $2\nmid\disc(\mc{A})$ then this stronger bound is sharp: for any prime $\ell\geq 3$ and $k\geq 0$, equality is obtained for submaximal orders of index $4$.
\end{remark}

\subsection{Cuspidal contribution}\label{sec:cusp part}

Now suppose $u\in \mathfrak{M}(\O)$ is a simultaneous eigenvector of $\mf{B}(\O)$ (and also of $W_{\mf{r}}$ if $\O$ has index $r$), normalized so that $\langle u,u\rangle=1$. Assume that $\Theta(u,u)\in \mc{S}_2(\Gamma_0(N))$ is a cusp form, where the level is $N=\disc(\O)$ (\cref{lem:level equals disc}). 

Following the proof of \cref{lem:Brandt_Hecke_equal_spectra}, $\Theta(u,u)$ is a $\Tan$-eigenvector. While it does \emph{not} generally belong to the new subspace $\mc{S}_2^{\text{new}}(\Gamma_0(N))$, the Deligne bound (\cref{lem:deligne bound}) still applies to its $n$-th Fourier coefficient $\lambda_n(u)$ whenever $n$ is coprime to $N$. 

The complication arises when our prime $\ell$ divides the level. Indeed, since $N=[\mf{O}:\O]\disc(\mc{A})$ and the prime $\ell$ may equal the index $r$, there are situations where $\ell$ divides~$N$, preventing the direct use of the Deligne bound. To bypass this, we establish slightly weaker bounds on $\lambda_{r^k}(u)$ by relating $\Theta(u,u)$ to suitable newforms, and applying the Deligne bound to the coefficients of those newforms instead.

\begin{proposition}\label{lem:cusp_coeff_bounds}
    Let $\O$ be maximal or submaximal, and let $u\in \mathfrak{M}(\O)$ be a cuspidal Brandt eigenvector. Let $\ell\nmid\disc(\mc{A})$ be prime. If $\O$ has index $r$ in a maximal order, assume further that $u$ is an eigenvector of $W_{\mf{r}}$.
    Then for all $k\geq 1$, %
    \[|\lambda_{\ell^k}(u)|\leq 2k\ell^{(k+2)/2}.\]
\end{proposition}

As in \cref{sec:brandt algebra relations}, we break into cases depending on the power of $\ell$ dividing $\disc(\O)$. That is, either $\O_\ell$ is maximal, $\O$ is submaximal of index $r$ and $\ell=r$, or $\O$ is submaximal of index $r^2$ and $\ell=r$. In each case we relate the coefficients of $\Theta(u,u)$ to the coefficients of a newform $f$ of appropriate level.

\begin{lemma}\label{lem:maximal eigenvalues}
    Let $\O$ be a maximal or submaximal order, $\ell\nmid\disc(\O)$ a prime, and let $u\in\mf{M}(\O)$ be a cuspidal Brandt eigenvector. Then there exists a newform 
    \[f(q)=\sum b_nq^n\in\mc{S}_2(\Gamma_0(M))\]
    for some $M\mid \disc(\O)$ such that $\lambda_{\ell^k}(u)=b_{\ell^k}$ for all $k\geq 0$.
\end{lemma}
\begin{proof}
    The condition $\ell\nmid\disc(\O)$ implies $\mc{A}$ is split at $\ell$ and $\O_\ell$ is maximal.
    By \cref{lem:decompose almost eigenform}, 
    there is a newform $f$ of some level $M\mid N$, $N:=\disc(\O)$, with 
    \[\Theta(u,u)(q)=\sum_{d\mid \frac{N}{M}} c_d f(q^d)\]
    for some coefficients $c_d\in\CC$,
    with $c_1=1$ by comparing the coefficient of $q$ on each side.
    Since $\ell\nmid\disc(\O)$ by assumption, every nontrivial divisor of $\frac{N}{M}$ has a prime factor distinct from $\ell$. So only the $d=1$ term contributes to the coefficient of $q^{\ell^k}$.
\end{proof}
\begin{remark}
    By a result of Ponomarev, we can strengthen \cref{lem:maximal eigenvalues} if $\O$ is %
    maximal: for any cuspidal Brandt eigenvector $u\in \mf{M}(\O)$, $\Theta(u,u)$ is a newform in $\mc{S}_2(\Gamma_0(\disc(\mc{A})))$ \cite[Theorem 4]{Ponomarev}. %
\end{remark}

\begin{lemma}\label{lem:index r eigenvalues}
    Let $\O$ be submaximal of index $r$, and $u\in\mf{M}(\O)$ a cuspidal Brandt eigenvector which is also an eigenvector for $W_{\mf{r}}$. Then either $|\lambda_{r^k}(u)|=1$ for all $k\geq 0$, or there exists a newform 
        \[f(q)=\sum b_nq^n\in\mc{S}_2(\Gamma_0(\disc(\mc{A})))\]
        and $\delta\in\{-1,1\}$ such that 
        \[\Theta(u,u)(q)=f(q)+r\delta f(q^r).\]
\end{lemma}
\begin{proof}
    Let $N=\disc(\O)$. We write $a_n:=\lambda_n(u)$ for convenience. By \cref{lem:Brandt_Hecke_equal_spectra} and \cref{lem:decompose almost eigenform}, there is a newform $f\in\mc{S}_2(\Gamma_0(M))$ for some $M\mid N$ such that \[\Theta(u,u)(q)=\sum_{d\mid \frac{N}{M}}c_df(q^d)\]
    for some coefficients $c_d\in\CC$.
Write $f(q)=\sum_{n=1}^\infty b_n q^n$ with $b_1=1$. Comparing coefficients on both sides, 
    \[
    a_n=\sum_{d\mid \gcd{ \left( n,\frac NM  \right)}} c_d b_{n/d}.
\]
Observe that $N$ is squarefree. We first prove that $\disc(\mc{A})$ must divide $M$. If not, suppose~$p$ divides $\disc(\mc{A})$ but not $M$. Then since the level of $f$ is coprime to $p$, its coefficients obey the standard Hecke relations $b_{p^{k+1}}=b_{p^k}b_p-pb_{p^{k-1}}$. On the other hand we have $a_p=\pm 1$ since 
$B_p^2=1$ (\cref{lem:classical ramified Brandt recurrence}). 
Solving the equations $a_p=b_p+c_p$ and 
\[1=a_{p^2}=b_{p^2}+c_pb_p=b_p^2-p+c_pb_p,\]
we find $|b_p|=p+1$, which contradicts the Deligne bound. Hence we have either $N=M$ or $N=rM$.

First consider the case $N=M$, so $\Theta(u,u)$ is a newform of level $\disc(\O)$. Then $a_{r^k}=b_{r^k}$ for all $k\geq 0$. Since $r$ divides the level of the newform $f$, the coefficients satisfy $a_{r^k}=a_r^k$ for all $k\geq 0$ \cite[Theorem 3(ii)$'$]{AtkinLehner}. %
Applying the recurrence relation of \cref{eq:recurrence} to $a_{r^k}$ we find
\[a_r^k=a_r^{k-1}(a_r-r\delta)-ra_r^{k-2},\]
where $\delta$ is the eigenvalue of $u$ under the involution $W_{\mf{r}}$.
Taking $k=2$ we find $\delta a_r=-1$, so that $|a_r|=1$. Hence $|a_{r^k}|=1$ for all $k$.

Now consider the case $N=rM$, so $\Theta(u,u)(q)=f(q)+cf(q^\qprime)$ %
for some $f\in\mathcal{S}_2(\Gamma_0(M))$ and $c\in\CC$. This results in the equation
    \[a_{\qprime^k}=b_{\qprime^k}+cb_{\qprime^{k-1}}\]
    for all $k\geq 0$ (using the convention that $b_x=0$ if $x$ is not a positive integer). Now since $\qprime$ is coprime to the level of $f$, the standard Hecke operator recurrence relation applies to the coefficients $b_{\qprime^k}$, while the $a_{\qprime^k}$ coefficients are instead subject to the recurrence relation of  \cref{eq:recurrence}. 
    
    We took our eigenbasis so that $u$ is an eigenvector for $W_{\mf{r}}$ as well, so as a consequence, letting $\delta=\pm 1$ be the eigenvalue of $u$ under $W_{\mf{r}}$, we have for all $k\geq 2$,
    \begin{align*}
        b_{\qprime^k}+cb_{\qprime^{k-1}}&=(b_{\qprime^{k-1}}b_{\qprime}-\qprime b_{\qprime^{k-2}})+c(b_{\qprime^{k-2}}b_{\qprime}-\qprime b_{\qprime^{k-3}}),\\
        a_{\qprime^k}&=a_{\qprime^{k-1}}(a_\qprime-\qprime\delta)-\qprime a_{\qprime^{k-2}}\\
        &=(b_{\qprime^{k-1}}+cb_{\qprime^{k-2}})(b_\qprime+c-\qprime\delta)-\qprime(b_{\qprime^{k-2}}+cb_{\qprime^{k-3}}),
    \end{align*}
    and the equality of these expressions simplifies to
    \begin{align*}
        (c-\qprime\delta)(b_{\qprime^{k-1}}+cb_{\qprime^{k-2}})=0.
    \end{align*}
    If $b_{\qprime^{k-1}}+cb_{\qprime^{k-2}}=0$ for all $k\geq 2$ then $b_{\qprime}=-c$ and $b_{\qprime^2}=c^2$, contradicting $b_{\qprime^2}=b_{\qprime}^2-\qprime$. Thus $c=\qprime\delta$, so we can write %
    \[
        \Theta(u,u)(q)=f(q)+ \qprime\delta f(q^{\qprime}).\qedhere\]
\end{proof}

\begin{lemma}\label{lem:index r2 eigenvalues}
    Let $\O$ be submaximal of index $r^2$, and let $u\in\mf{M}(\O)$ be a cuspidal Brandt eigenvector. Then either $\lambda_{r^k}(u)=0$ for all $k\geq 1$, or there exists a newform 
    \[f(q)=\sum b_nq^n\in\mc{S}_2(\Gamma_0(\disc(\mc{A})))\]
    such that $\lambda_{r}(u)=0$ and $\lambda_{r^k}(u)=(r^2-r)b_{r^{k-2}}$ for all $k\geq 2$.
\end{lemma}
\begin{proof}
    If $\lambda_{r^k}(u)=0$ for all $k\geq 1$ then we are done, so suppose this is not the case. Then by \cref{prop:Eu eigenvector}, we have $\lambda_r(u)=0$, and there exists $w\in\mf{M}(\mf{O})$, an eigenvector for $\widetilde{B}_n\in\mf{B}(\mf{O})$ for all $n\geq 1$, satisfying $\lambda_{r^k}(u)=(r^2-r)\lambda_{r^{k-2}}(w)$ for $k\geq 2$. Now since $u$ is cuspidal we have $B_0u=0$, so the sum of entries of $u$ is $0$. By definition of $E$, the sum of entries of $Eu$ equals the sum of entries of $u$, so we can conclude $\widetilde{B}_0(Eu)=0$. Thus $Eu$ is in fact a cuspidal eigenvector for the full Brant algebra $\mf{B}(\mf{O})$.
    Therefore there is a newform $f(q)=\sum b_nq^n$ such that $\lambda_{r^k}(w)=b_{r^k}$ by \cref{lem:maximal eigenvalues}. Hence $\lambda_{r^k}(u)=(r^2-r)b_{r^{k-2}}$ for all $k\geq 2$ as desired.
\end{proof}

\begin{proof}[Proof of \emph{\cref{lem:cusp_coeff_bounds}}]
    If $\O_\ell$ is maximal then there is a newform $f(q)=\sum b_nq^n$ such that $\lambda_{\ell^k}(u)=b_{\ell^k}$ by \cref{lem:maximal eigenvalues}, so $|\lambda_{\ell^k}(u)|\leq (k+1)\ell^{k/2}$ by the Deligne bound.

    If $\O$ is submaximal of index $r$ and $\ell=r$, then either $|\lambda_{r^k}(u)|=1$ for all $k$, or there is a newform $f(q)=\sum b_nq^n$ such that $\lambda_{r^k}(u)=b_{r^k}\pm r b_{r^{k-1}}$ %
    by \cref{lem:index r eigenvalues}, so 
    \begin{align*}
        |\lambda_{\qprime^k}(u)|&\leq |b_{\qprime^k}|+\qprime |b_{\qprime^{k-1}}|\\
        &\leq (k+1)\qprime^{k/2}+\qprime k\qprime^{(k-1)/2}\\
        &= (k+1+k\sqrt{\qprime})\qprime^{k/2}.
    \end{align*}

    If $\O$ is submaximal of index $r^2$ and $\ell=r$, then either $\lambda_{r^k}(u)=0$ for all $k$, or there is a newform $f(q)=\sum b_nq^n$ such that $\lambda_{r}(u)=0$ and $\lambda_{r^k}(u)=(r^2-r)b_{r^{k-2}}$ by 
    \cref{lem:index r2 eigenvalues}. In the latter case, with the Deligne bound $|b_{\qprime^{k-2}}|\leq (k-1)\qprime^{(k-2)/2}$, we therefore have 
    \[|\lambda_{\qprime^k}(u)|\leq (\qprime^2-\qprime)(k-1)\qprime^{(k-2)/2}\leq k\qprime^{(k+2)/2}\]
    for $k\geq 2$ (and the bound for $k=1$ is immediate).
\end{proof}

\section{Counting elements of prescribed norm}\label{sec:final counts}

Let $\mc{A}$ be a definite quaternion algebra over $\QQ$. For an order $\O\sus\mc{A}$, 
we let $\repnum{n}{\O}$ denote the number of elements of norm $n$ in $\O$, so that $\theta_{\O}(q)=\sum_{n=0}^\infty \repnum{n}{\O} q^n$.

\begin{proposition}\label{prop: max submax upper lower}
    Let $\O\sus\mc{A}$ be an order, $\ell\nmid\disc(\mc{A})$ a prime, and $k\geq 1$ be an integer. If $\O$ is maximal then
    \[\repnum{\ell^k}{\O}\geq \frac{24}{\phi(\disc(\mc{A}))}\frac{\ell^{k+1}-1}{\ell-1}-12k\ell^{(k+2)/2},\]
    and if $\O$ is not maximal then 
    \[\repnum{\ell^k}{\O}\leq \frac{16}{\phi(\disc(\mc{A}))}\frac{\ell^{k+1}-1}{\ell-1}+12k\ell^{(k+2)/2}.\]
\end{proposition}
\begin{proof}
    Once we prove the upper bound on $a_{\ell^k}(\O)$ for submaximal orders, the same upper bound for all non-maximal orders will follow, because any non-maximal order $\O$ is contained in a submaximal order $\O'$ and hence $a_{\ell^k}(\O)\leq a_{\ell^k}(\O')$.    
    We therefore assume, without loss of generality, that $\O$ is either maximal or submaximal. 

    We first consider the case that $|\O^\times|>6$. Up to isomorphism, there are precisely three such quaternion orders \cite[Theorem 11.5.14]{Voight}: these are the quaternion orders with discriminants $2$ (the Hurwitz order), $3$, and $4$ (the Lipschitz order). The lattice structures of these orders are well studied (they are $D_4$, the orthogonal direct sum $A_2 \oplus A_2$, and $\ZZ^4$, respectively), and in particular there are closed formulas for their theta functions: for any quaternion order of discriminant $N$ with $(N-1)\mid 12$, one has
    \[|\O^\times|=\frac{24}{N-1}\qquad \text{and} \qquad \theta_\O(q)=1+|\O^\times|\sum_{m=1}^\infty \left(\sum_{N\nmid d\mid m}d\right)q^m\]
    (for instance, using \cite[Chapter 26]{Voight}), %
    so in particular we have 
    \[a_{\ell^k}(\O)=\begin{cases}
        \displaystyle \frac{24}{\phi(\disc(\mc{A}))}\frac{\ell^{k+1}-1}{\ell-1}, & \disc\O=2\text{ or }3\text{ (maximal)}\\
        \displaystyle \frac{8}{\phi(\disc(\mc{A}))}\frac{\ell^{k+1}-1}{\ell-1}, & \disc\O=4\text{ (submaximal)}.
    \end{cases}\] 
    In all these cases the desired inequalities hold.
    
    From now on we assume $|\O^\times|\leq 6$. Recall that in the case of a submaximal $\O$ of index~$r$ in a maximal order, we assume that each $u_i$ is an eigenvector of the operator $W_{\mf{r}}$ defined by \cref{eq:Wr} (this is possible by \cref{lem:W commutes}). First we use \cref{eq:decomp with genus part} to write
    \[\theta_{\O}(q)-g(q)=c_1f_1(q)+\cdots+c_mf_m(q)\]
    with $g(q)$ the non-cuspidal part, $c_1,\ldots,c_m$ non-negative real numbers, $c_1+\cdots+c_m\leq |\O^\times|\leq 6$, and $f_1,\ldots,f_m$ normalized cusp forms that are $\Tan$-eigenvectors. Write $f_i(q)=\sum a_{i,n}q^n$ for each~$i$, and $g(q)=\sum g_n(\O)q^n$. By \cref{lem:cusp_coeff_bounds}, we have $|a_{i,\ell^k}|\leq 2k\ell^{(k+2)/2}$ for all $i$ and all $k\geq 1$. So comparing the coefficient of $q^{\ell^k}$ on both sides, we obtain
    \begin{align*}
        \left|\repnum{\ell^k}{\O}-g_{\ell^k}(\O)\right|&\leq c_1|a_{1,\ell^k}|+\cdots+c_m|a_{m,\ell^k}|\\
        &\leq (c_1+\cdots+c_m)2k\ell^{(k+2)/2}\\
        &\leq 12k\ell^{(k+2)/2},
    \end{align*}
    and therefore
    \[g_{\ell^k}(\O)-12k\ell^{(k+2)/2}\leq \repnum{\ell^k}{\O}\leq g_{\ell^k}(\O)+12k\ell^{(k+2)/2}.\]
    We now apply \cref{cor: submaximal eisenstein bound} to obtain the exact value of $g_{\ell^k}(\O)$ if $\O$ is maximal, and an upper bound on $g_{\ell^k}(\O)$ if $\O$ is submaximal.
\end{proof}

\begin{theorem}\label{thm:generation bound}
    Let $\mf{O}$ be a maximal order in a definite quaternion algebra $\mc{A}$ over $\QQ$, and let $\ell\nmid\disc(\mc{A})$ be a prime. For any $k\geq 1$ satisfying
    \begin{align}\label{eq:generation bound}
        \ell^{(k-2)/2}>3\phi(\disc(\mc{A}))k,
    \end{align}
    the elements of norm $\ell^k$ in $\mf{O}$ generate $\mf{O}$ as a $\ZZ$-algebra.
\end{theorem}
\begin{proof}
    Let $\O$ be any proper suborder of $\mf{O}$, and apply \cref{prop: max submax upper lower}:
    \begin{align*}
        \repnum{\ell^k}{\mf{O}}-\repnum{\ell^k}{\O}&\geq \frac{8}{\phi(\disc(\mc{A}))}\frac{\ell^{k+1}-1}{\ell-1}-24k\ell^{(k+2)/2}\\
        &\geq \frac{8}{\phi(\disc(\mc{A}))}\ell^k-24k\ell^{(k+2)/2},
    \end{align*}
    which is strictly positive by \cref{eq:generation bound}. This proves that the elements of norm $\ell^k$ in $\mf{O}$ do not lie in any proper suborder of $\mf{O}$, so the $\ZZ$-algebra they generate must be the entirety of $\mf{O}$. %
\end{proof}

If we take $\mc{A}=\mc{A}_p$ to be ramified at exactly $p$ and $\infty$, then we have $\phi(\disc(\mc{A}))=p-1$, so \cref{thm: explicit p and k} follows.

\section{Bounding the number of generators}\label{sec: generating set size}

Suppose we wish to find a set of generators for a maximal order $\mf{O}$ with norms constrained to lie in some set $S$ of positive integers. As a running example we may take $S=\{\ell^k:k\geq 0\}$ for some prime $\ell\nmid\disc(\mc{A})$. \Cref{thm:generation bound} tells us that we can replace the set of allowable norms $S$ with the singleton %
$\{\ell^k\}$ for any sufficiently large value of $k$. On the other hand, the corresponding set of generators for $\mf{O}$ may a priori be quite large.

\begin{proposition}\label{prop:weak number bound}
Let $\mf{O}$ be a maximal order in a definite quaternion algebra $\mc{A}$ over $\QQ$.
    Suppose that $\mf{O}$ is generated as a $\ZZ$-algebra by a finite set $C_0$, and let $n$ be the maximum norm attained by %
    any element of $C_0$. Then $\mf{O}=\ZZ[C]$ for a subset $C\subseteq C_0$ with \[|C|\leq 2\log_2(n)-\log_2(\disc(\mc{A}))+4.\]
\end{proposition}
\begin{proof}
    Since $\mf{O}=\ZZ[C_0]$, there must exist non-rational $\alpha,\beta\in C_0$ that do not lie in the same subfield of $\mc{A}$.
    Let $\Lambda_0=\ZZ[\alpha,\beta]$. Since $\alpha,\beta\in\mf{O}$, one can check that $\Lambda_0=\Span_\ZZ(1,\alpha,\beta,\alpha\beta)$, a full rank lattice.
    Given $\{b_1,b_2,b_3,b_4\}=\{1,\alpha,\beta,\alpha\beta\}$, a $\ZZ$-basis of $\Lambda_0$, 
    recall that the reduced discriminant of $\Lambda_0$ is the square root
of the determinant of the Gram matrix $\langle b_i,b_j\rangle=\trd(b_i\bar{b_j})$. Since the trace form is positive
definite, Hadamard's inequality gives
\[
\det\bigl(\langle b_i,b_j\rangle\bigr)
\le
\prod_{i=1}^4 \langle b_i,b_i\rangle,
\]
where
\[
\prod_{i=1}^4\langle b_i,b_i\rangle
=
\prod_{i=1}^4\trd(b_i\bar b_i)
=
\prod_{i=1}^42\nrd(b_i)=
2^4\nrd^2(\alpha)\nrd^2(\beta)\le (2n)^4.
\]
It follows that
\[
\disc(\Lambda_0)=\sqrt{\det\bigl(\langle b_i,b_j\rangle\bigr)}\le 
4n^2.
\]
Therefore,
\[
[\mf O:\Lambda_0]
=
\frac{\disc(\Lambda_0)}{\disc(\mf O)}
=
\frac{\disc(\Lambda_0)}{\disc(\mc A)}
\le
\frac{4n^2}{\disc(\mc A)}.
\]
    Now for $i\geq 0$, if $\Lambda_i\neq\mf{O}$, it is possible to pick $\gamma_{i+1}\in C_0\setminus \Lambda_i$ and let $\Lambda_{i+1}:=\ZZ[\alpha,\beta,\gamma_1,\ldots,\gamma_{i+1}]$. Then $\Lambda_{i}$ is strictly contained in $\Lambda_{i+1}$. %
    Once we reach a value of~$m$ with $\Lambda_m=\mf{O}$, we take $C=\{\alpha,\beta,\gamma_1,\ldots,\gamma_m\}$. Since 
    \[ [\mf{O}:\Lambda_{i+1}]\leq \frac12[\mf{O}:\Lambda_i]\qquad\text{for $i=0,\ldots,m-1$ and}\qquad [\mf{O}:\Lambda_m]=1,\]
    we must have \[m\leq \log_2([\mf{O}:\Lambda_0])\leq \log_2(4n^2/\disc(\mc{A})).\]
\end{proof}
As an example, let $p,\ell$ be primes with $p>\frac13\ell^{30}+1$. If  $k=\lfloor\frac 52\log_\ell(3p-3)\rfloor+1$, so we are in the conditions of \cref{thm: explicit p and k}, then $\mf{O}$ is generated by the set of elements of norm $\ell^k$. So we may conclude that there is a set of at most
\[5\log_\ell(3p-3)\log_2(\ell)+2\log_2(\ell)-\log_2(p)+4\leq \log_2(p^4\ell^2)+12\]
elements of norm $\ell^k$ that generates $\mf{O}$. While the number of elements of this generating set is polynomial in the bit length of $p$, it should be possible to find an even smaller generating set.
The following result tells us that we can indeed find a generating set with a very small number of elements, provided that we remove the constraint that the norms must be small.

\begin{theorem}\label{prop:number of generators}
    Let $\mf{O}$ be a maximal order in a definite quaternion algebra $\mc{A}$ over $\QQ$. Let~$S$ be any infinite set consisting of positive integers coprime to $\disc(\mf{O})$. Then $\mf{O}=\ZZ[C]$ for a subset $C\subseteq\mf{O}$ with $|C|\leq 3$ and $\nrd(x)\in S$ for all $x\in C$.
\end{theorem}

More precisely, we will show that there exist elements $\alpha,\beta,\gamma\in\mf{O}$, each with norm in~$S$, such that $\mf{O}=\Span_\ZZ(1,\alpha,\beta,\gamma,\alpha\beta,\alpha\gamma,\beta\gamma)$.

\begin{remark}\label{rmk:number gens optimal}
    Here are a few remarks on the optimality of \cref{prop:number of generators}. First note that this result does not hold for arbitrary quaternion orders, even if we exclude the bound on $|C|$. For instance, if $\O$ is an Eichler order of index $2$ in a maximal order, then $\O$ is not generated by any number of elements of odd norm \cite[2.3.4]{Goren_Love_Elts_Prescribed_Norm}.

    We also note that \Cref{prop:number of generators} does not hold under the constraint $|C|\leq 2$, even if we remove all constraints on the norms of elements of $C$. Indeed, if $\mf{O}$ is generated as a $\ZZ$-algebra by two elements $a_1,a_2\in\mf{O}$, then by \cite[Definition 15.4.4]{Voight} we have
        \[\disc(\mc{A})=\disc(\mf{O})=|\trd((a_1a_2-a_2a_1)\overline{a_1a_2})|=|\nrd(a_1a_2-a_2a_1)|,\]
    so that $\mf{O}$ contains an element $a_1a_2-a_2a_1$ of trace $0$ and norm $\disc(\mc{A})$. %
    Not every maximal order contains such an element: for instance, if $\disc(\mc{A})=p$, then $\mf{O}$ has an element of trace $0$ and norm $p$ if and only if it is isomorphic to the endomorphism ring of an elliptic curve over $\FF_p$. One can use \cite[Proposition 3.1]{he2025grosslatticessupersingularelliptic} to show a partial converse: if $\mf{O}$ is isomorphic to the endomorphism ring of an elliptic curve over $\FF_p$, then it is generated as a $\ZZ$-algebra by two elements.
\end{remark}

The key ingredient in the proof of \cref{prop:number of generators} is the notion of theta functions of lattice cosets: given a lattice $\Lambda\subseteq \mc{A}$, an element $t\in\Lambda$, and a positive integer $m$, we may define the theta function of $t+m\Lambda$. If $t\notin m\Lambda$ then this is an \emph{inhomogeneous theta function}. We use the following local-global principle for representability of integers by inhomogeneous theta functions, which is a variant of \cite[Theorem 6.6.1]{Goren_Love_SSEC_Survey}. 
\begin{theorem}\label{thm:local global inhomogeneous}
    Let $\mc{A}$ be a definite quaternion algebra over $\QQ$ and $\O\subseteq \mc{A}$ an order. Let $d\geq 2$ be a squarefree positive integer, and $t\in\O$ with $t\notin\ell\O$ for all primes $\ell\mid d$. 
    For any sufficiently large positive integer $\mu$ coprime to $\disc(\O)$, if there exists $x_\ell\in t + (d\O\otimes \ZZ_\ell)$ with $\nrd(x_\ell)=\mu$ for all primes $\ell$, then there exists $x\in t + d\O$ with $\nrd(x)=\mu$.
\end{theorem}
In fact \cref{thm:local global inhomogeneous} still holds even if we drop the squarefree constraint on $d$, but the necessary modifications to the proof take some extra work which is not needed for our applications here.
\begin{proof}
    If $d$ is prime then this is a special case of \cite[Theorem 6.6.1]{Goren_Love_SSEC_Survey}, with $L=\QQ$, $n=4$, $\Lambda=\O$, $\mf{p}=d\ZZ$, and $M=t+d\O$. We indicate here how the proof changes for arbitrary squarefree $d$. A rough sketch of the proof is to decompose the theta function of~$M$ into a cuspidal part and a non-cuspidal part, and, as in \cref{cor: submaximal eisenstein bound}, to show that the non-cuspidal part grows sufficiently rapidly as $\mu$ increases %
    (faster than the cuspidal part). The coefficient of $q^\mu$ in the non-cuspidal part of the theta function of $M$ is labeled $c_\mu(\theta_{\mf{g}(M)})$ in \cite{Goren_Love_SSEC_Survey}. The only part of the proof of \cite[Theorem 6.6.1]{Goren_Love_SSEC_Survey} that does not go through verbatim is the claim that the ratio of $c_\mu(\theta_{\mf{g}(M)})$ to $c_\mu(\theta_{\mf{g}(\Lambda)})$ is bounded below by a positive constant that does not depend on $\mu$; for this, we must modify the proof of \cite[Lemma 6.5.6]{Goren_Love_SSEC_Survey}.  %

    Applying (6.21) and (6.22) of loc.~cit.~to both $M$ and $\Lambda$ we obtain
    \[c_\mu(\theta_{\mf{g}(M)})\prod_\ell c_{\mu,\ell}(E_\Lambda;0)=\frac{m(\mf{g}(\Lambda))}{m(\mf{g}(M))}c_\mu(\theta_{\mf{g}(\Lambda)})\prod_\ell c_{\mu,\ell}(E_M;0),\]
    where $c_{\mu,\ell}(E_M;0)$ and $c_{\mu,\ell}(E_\Lambda;0)$ are certain local densities associated to Eisenstein series, defined in \cite[p. 71]{Shimura_repnum}, and $m(\mf{g}(\Lambda))$ and $m(\mf{g}(M))$ are masses of the genera of $\Lambda$ and~$M$ respectively. %
    Since $M_\ell=\Lambda_\ell$ for all primes $\ell\nmid d$, we may restrict the product on each side to $\ell\mid d$. The proof of \cite[Lemma 6.5.6]{Goren_Love_SSEC_Survey} shows that $c_{\mu,\ell}(E_\Lambda;0)\neq 0$. So it suffices to prove, for each prime $\ell\mid d$, that $c_{\mu,\ell}(E_M;0)/c_{\mu,\ell}(E_\Lambda;0)$ is bounded below by a positive constant that does not depend on $\mu$. 
    
    Now let $M'=t+\ell\Lambda$.  The proof of \cite[Lemma 6.5.6]{Goren_Love_SSEC_Survey} shows that the ratio $c_{\mu,\ell}(E_{M'};0)/c_{\mu,\ell}(E_\Lambda;0)$ is bounded below by a positive constant that does not depend on $\mu$. Since $d$ is squarefree we have $M_\ell=M'_\ell$ and therefore $c_{\mu,\ell}(E_{M};0)=c_{\mu,\ell}(E_{M'};0)$, proving the desired result.
\end{proof}

We will apply this local-global principle in the following way.
\begin{corollary}\label{cor:local global}
    Let $\mf{O}$ be a maximal order in a definite quaternion algebra over $\QQ$. Let~$S$ be an infinite set of positive integers coprime to $\disc(\mf{O})$, and $m\geq 2$ a positive integer. For all $\ell\mid m$ and $s\in S$, let $x_{\ell,s}\in \mf{O}_\ell\setminus\ell\mf{O}_\ell$ be an element satisfying $\nrd(x_{\ell,s})=s$. Then there exists $x\in \mf{O}$ and $s\in S$ such that $\nrd(x)=s$ and $x\equiv x_{\ell,s}\bmod{\ell\mf{O}_\ell}$ for all $\ell\mid m$.
\end{corollary} 
\begin{proof}
    Let $m'$ be the squarefree product of primes dividing $m$. For each $s\in S$, there exists $t_s\in \mf{O}$ such that $t_s\equiv x_{\ell,s}\bmod{\ell\mf{O}_\ell}$ for all $\ell\mid m$. Since $S$ is infinite but $\mf{O}/m'\mf{O}$ is finite, by the pigeonhole principle there exists an infinite subset $S'\subseteq S$ such that the cosets $t_s+m'\mf{O}$ are equal for all $s\in S'$; pick $t\in\mf{O}$ with $t\equiv t_s\bmod{m'\mf{O}}$ for all $s\in S'$.

    Now consider any $s\in S'$. If $\ell\mid m$ then
    \begin{equation}\label{eq:coset}
        t+m'\mf{O}_\ell=t_s+m'\mf{O}_\ell=t_s+\ell\mf{O}_\ell=x_{\ell,s}+\ell\mf{O}_\ell,
    \end{equation}
    which contains an element $x_{\ell,s}$ of norm $s$. Since $x_{\ell,s}\notin\ell\mf{O}_\ell$ we have $t\notin \ell\mf{O}$. If $\ell\nmid m$ then $t+m'\mf{O}_\ell=\mf{O}_\ell$, which contains an element of norm $s$ because $\nrd:\mf{O}_\ell\to\ZZ_\ell$ is surjective for maximal orders. So we may apply \cref{thm:local global inhomogeneous} to conclude that for sufficiently large $s\in S'$, there exists $x\in t+m'\mf{O}$ with $\nrd(x)=s$; for such $x$ we have $x\equiv x_{\ell,s}\bmod{\ell\mf{O}_\ell}$ by \cref{eq:coset}.
\end{proof}

Throughout the argument below, if $A$ is a $\ZZ$-algebra and $x\in A$, and a prime $\ell\in\ZZ$ is given, we write $\widetilde{x}$ for the image of $x$ in $A/\ell A$.

\begin{lemma}\label{lem:nrd n in Zl2}
    Let $\ell$ be prime and $n\in\ZZ_\ell^\times$. There exists $a\in\ZZ_{\ell^2}$ with $N_{\QQ_{\ell^2}/\QQ_\ell}(a)=n$ and $\widetilde{a}\in \FF_{\ell^2}\setminus\FF_\ell$.
\end{lemma}
\begin{proof}
    Since the norm map $\FF_{\ell^2}^\times\to\FF_\ell^\times$ is surjective with fibres of size $\ell+1$, with at most two elements of each fibre in $\FF_\ell^\times$, there exists $\widetilde{a}\in \FF_{\ell^2}^\times\setminus\FF_\ell^\times$ with $N_{\FF_{\ell^2}/\FF_\ell}(\widetilde{a})=n$ in $\FF_\ell$. The minimal polynomial of $\widetilde{a}$ then has two distinct roots over $\FF_{\ell^2}$, so we can Hensel lift to $a\in \ZZ_{\ell^2}^\times$ with $N_{\QQ_{\ell^2}/\QQ_\ell}(a)=n$.
\end{proof}

\begin{lemma}\label{lem:nrd n disc unit}
    Let $\ell$ be a prime, $\mc{A}_\ell$ a quaternion algebra over $\QQ_\ell$, and $\mf{O}_\ell\subseteq \mc{A}_\ell$ a maximal order. 
    Let $n\in\ZZ_\ell$, and if $\mc{A}_\ell$ is a division algebra suppose further that $n\in\ZZ_\ell^\times$. There exists $x\in\mf{O}_\ell$ with $\nrd(x)=n$ and $\disc(x)\in\ZZ_\ell^\times$.
\end{lemma}
\begin{proof}
    If $n\in\ZZ_\ell^\times$, take $a$ as in \cref{lem:nrd n in Zl2} and $x$ to be the image of $a$ under an embedding $\ZZ_{\ell^2}\to\mf{O}_\ell$. Letting $\sigma$ be the nontrivial automorphism of $\QQ_{\ell^2}/\QQ_\ell$, we have $\disc(x)=(a-\sigma(a))^2$, which is not divisible by $\ell$ because $\widetilde{a}\neq\widetilde{\sigma(a)}$ in $\FF_{\ell^2}$.
    
    If $n\in\ell\ZZ_\ell$ (in which case we assume $\mc{A}_\ell$ is split), 
    $x=\begin{psmallmatrix}
        0&-n \\ 1&1
    \end{psmallmatrix}$
    has $\nrd(x)=n$ and $\disc(x)=1-4n\in \ZZ_\ell^\times$.
\end{proof}

\begin{proposition}\label{lem:extend basis}
    Let $\ell$ be a prime, $\mc{A}_\ell$ a quaternion algebra over $\QQ_\ell$, and $\mf{O}_\ell\subseteq \mc{A}_\ell$ a maximal order. 
    Let $n\in\ZZ_\ell$, and if $\mc{A}_\ell$ is a division algebra suppose further that $n\in\ZZ_\ell^\times$. Let $x\in\mf{O}_\ell$ and suppose the reduction $\widetilde{x}\in \mf{O}_\ell/\ell\mf{O}_\ell$ is not in $\FF_\ell$. Then there exists $y\in \mf{O}_\ell\setminus \ell\mf{O}_\ell$ %
    with $\nrd(y)=n$ such that $1,x,y,xy$ form a $\ZZ_\ell$-basis for $\mf{O}_\ell$. 
\end{proposition}
\begin{proof}
    It suffices to show that there exists $y\in \mf{O}_\ell$ with $\nrd(y)=n$ and with $\FF_\ell[\widetilde{x},\widetilde{y}]=\mf{O}_\ell/\ell\mf{O}_\ell$. For in this case we necessarily have $\ZZ_\ell[x,y]=\mf{O}_\ell$ by \cref{lem:lattice quotient l}, and a computation shows that any element of $\ZZ_\ell[x,y]$ is contained in $\Span_{\ZZ_\ell}(1,x,y,xy)$. This condition also forces $\widetilde{y}\neq 0$ so $y\notin\ell\mf{O}_\ell$.

    First suppose $\mc{A}_\ell$ is split and $n\in\ZZ_\ell^\times$. Let $\widetilde{x}$ be the reduction of $x$ in $\mf{O}_\ell/\ell\mf{O}_\ell$. If the minimal polynomial of $\widetilde{x}$ over $\FF_\ell$ is irreducible, set $y:=\begin{psmallmatrix}
        n&1\\0&1
    \end{psmallmatrix}\in\mf{O}_\ell$. 
    If instead the minimal polynomial of $\widetilde{x}$ over $\FF_\ell$ is reducible, consider the element $a$ from \cref{lem:nrd n in Zl2} and let $y$ be the image of $a$ under an embedding $\ZZ_{\ell^2}\to \mf{O}_\ell$, so that $\widetilde{y}$ has irreducible minimal polynomial. In either case, $\FF_\ell[\widetilde{x},\widetilde{y}]$ contains both a copy of $\FF_{\ell^2}$ and an element with reducible minimal polynomial that is not in $\FF_\ell$, and hence not in $\FF_{\ell^2}$. 
    This implies that $\FF_\ell[\widetilde{x},\widetilde{y}]= \mf{O}_\ell/\ell\mf{O}_\ell$ because $\FF_\ell[\widetilde{x},\widetilde{y}]$ and $\mf{O}_\ell/\ell\mf{O}_\ell$ are both vector spaces over $\FF_{\ell^2}$, the former of dimension greater than $1$ and the latter of dimension $2$.

    Now let $\ell=p$ with $\mc{A}_p$ a division algebra, so $n\in\ZZ_p^\times$. In this case $\mf{O}_p$ has the form
    \[
    \mf{O}_{p}
    =
    \left\{
    \begin{pmatrix}
    u & p v \\
    \sigma(v) & \sigma(u)
    \end{pmatrix}
    : u,v\in \ZZ_{p^2}
    \right\},
    \]
    where $\ZZ_{p^2}$ is the ring of integers of the  unramified 
    quadratic extension $\QQ_{p^2}/\QQ_p$ and $\sigma$ is the non-trivial element of $\operatorname{Gal}(\QQ_{p^2}/\QQ_{p})$. If the reduction $\widetilde{x}$ is not in the diagonal subring of $\mf{O}_p/p\mf{O}_p$, %
    let $y=\begin{psmallmatrix}
        a & 0 \\ 0 & \sigma(a)
    \end{psmallmatrix}$ for $a$ as in \cref{lem:nrd n in Zl2}. If $\widetilde{x}$ is in the diagonal subring (and hence $\FF_\ell[\widetilde{x}]\simeq\FF_{\ell^2}$ because $\widetilde{x}\notin\FF_\ell$), we can apply \cref{lem:nrd n in Zl2} to get an element $b\in\ZZ_{p^2}$ with $N_{\QQ_{p^2}/\QQ_p}(b)=n+p$, and take $y=\begin{psmallmatrix}
        b & p \\ 1 & \sigma(b)
    \end{psmallmatrix}$. Again $\FF_\ell[\widetilde{x},\widetilde{y}]$ contains both an embedding of $\FF_{\ell^2}$ and an element that is not in the image of this embedding, so it must equal $\mf{O}_\ell/\ell\mf{O}_\ell$.

    Finally let $\mc{A}_\ell$ be split and $n\in\ell\ZZ_\ell$. The element $y_0=\begin{psmallmatrix}
        -1 & n+1 \\ -1 & 1
    \end{psmallmatrix}\in M_2(\ZZ_\ell)$ has $\nrd(y_0)=n$ and characteristic polynomial $x^2+n$; its reduction $\widetilde{y}_0\in M_2(\FF_\ell)$ has a unique $1$-dimensional eigenspace, the span of the column vector $(1,1)^t$. 
    Since $\widetilde{x}\in M_2(\FF_\ell)$ is not a scalar matrix, there exists a nonzero $v\in\FF_\ell^2$ which is not an eigenvector for~$\widetilde{x}$. Let $\widetilde{w}\in \GL_2(\FF_\ell)$ satisfy $\widetilde{w}(1,1)^t=v$, let $w\in\GL_2(\ZZ_\ell)$ be a lift of $\widetilde{w}$, and let $y=wy_0w^{-1}\in M_2(\ZZ_\ell)$. Then every eigenvector of $\widetilde{y}$ is in the span of $v$, so $\widetilde{x}$ and $\widetilde{y}$ have no eigenvectors in common. Note also that $\widetilde{y}$ is a nonzero nilpotent element. By \cite[Theorem 1]{Racine}, every proper subalgebra of $M_2(\FF_\ell)$ is either contained in the image of an embedding of $\FF_{\ell^2}$ or has a common eigenvector. As neither of these are true of $\FF_\ell[\widetilde{x},\widetilde{y}]$, we can conclude that $\FF_\ell[\widetilde{x},\widetilde{y}]=\mf{O}_\ell/\ell\mf{O}_\ell$.
\end{proof}

\begin{proof}[Proof of \emph{\cref{prop:number of generators}}]
    We begin by taking any $\alpha\in\mf{O}\setminus\ZZ$ with $\nrd(\alpha)\in S$. Here is one way to construct such an $\alpha$. Fix any prime $\ell$. Then for any $s\in S$ there exists $x_{\ell,s}\in \mf{O}_\ell$ with $\nrd(x_{\ell,s})=s$ and $\ell\nmid \disc(x_{\ell,s})$ by \cref{lem:nrd n disc unit}. The discriminant condition forces $x_{\ell,s}\notin\ell\mf{O}_\ell$, so we may apply \cref{cor:local global} to obtain $\alpha\in\mf{O}$ and $s_1\in S$ with $\nrd(\alpha)=s_1$ and $\alpha\equiv x_{\ell,s_1}\bmod{\ell\mf{O}_\ell}$. Since $\disc(\alpha)\equiv\disc(x_{\ell,s_1})\not\equiv 0\bmod{\ell\ZZ_\ell}$ we have $\disc(\alpha)\neq 0$ and hence $\alpha\notin\ZZ$.

    Now for all $\ell\mid \disc(\alpha)$ and $s\in S$, use \cref{lem:nrd n disc unit} to pick $y_{\ell,s}\in\mf{O}_{\ell}$ with $\nrd(y_{\ell,s})=s$ and $\disc(y_{\ell,s})\in\ZZ_\ell^\times$. Applying \cref{cor:local global} again, we obtain $\beta\in \mf{O}\setminus\ZZ$ and $s_2\in S$ with $\nrd(\beta)=s_2$ and $\beta\equiv y_{\ell,s_2}\bmod{\ell\mf{O}_\ell}$ for all $\ell\mid\disc(\alpha)$. Then $\disc(\beta)\equiv\disc(y_{\ell,s_2})\not\equiv 0\bmod{\ell\ZZ_\ell}$ for all $\ell\mid \disc(\alpha)$, so $\disc(\beta)$ is coprime to $\disc(\alpha)$. 
    
    Finally, take $\Lambda:=\Span_\ZZ( 1,\alpha,\beta,\alpha\beta)$. %
Since $\ZZ[\alpha]$ and $\ZZ[\beta]$ are quadratic orders with coprime discriminants, $\Lambda$ is a finite-index sublattice of $\mf{O}$, and we set $m:=[\mf{O}:\Lambda]$. If $m=1$ then we are done, so assume $m\geq 2$. Let $\ell$ be a prime factor of $m$ and $s\in S$. Suppose first that $\ell \nmid \disc(\alpha)$. Then $\widetilde{\alpha}\notin \FF_\ell$, so by \cref{lem:extend basis}, there exists $z_{\ell,s}\in \mf{O}_\ell\setminus \ell\mf{O}_\ell$ with $\nrd(z_{\ell,s})=s$ and such that $\Span_{\ZZ_\ell}(1,\alpha,z_{\ell,s},\alpha z_{\ell,s})=\mf{O}_\ell$. If instead $\ell \mid \disc(\alpha)$, then $\ell\nmid\disc(\beta)$ by construction of $\beta$, so by \cref{lem:extend basis}, there exists $z_{\ell,s}\in \mf{O}_\ell\setminus \ell\mf{O}_\ell$ with $\nrd(z_{\ell,s})=s$ and such that $\Span_{\ZZ_\ell}( 1,\beta,z_{\ell,s},\beta z_{\ell,s})=\mf{O}_\ell$.
    Once again we apply \cref{cor:local global} to obtain $\gamma\in\mf{O}$ and $s_3\in S$ with $\nrd(\gamma)=s_3$, $\gamma\equiv z_{\ell,s_3}\bmod{\ell\mf{O}_\ell}$ for all $\ell\mid m$. Applying \cref{lem:lattice quotient l}, we have the following generating sets for $\mf{O}_\ell$.
    \begin{itemize}
        \item If $\ell\nmid m$ then $\mf{O}_\ell=\Span_{\ZZ_\ell}(1,\alpha,\beta,\alpha\beta)$.
        \item If $\ell\mid m$ and $\ell\nmid\disc(\alpha)$ then $\mf{O}_\ell=\Span_{\ZZ_\ell}(1,\alpha,\gamma,\alpha\gamma)$.
        \item If $\ell\mid m$ and $\ell\mid\disc(\alpha)$ then $\mf{O}_\ell=\Span_{\ZZ_\ell}(1,\beta,\gamma,\beta\gamma)$.
    \end{itemize}
    Thus the index of $\Span_\ZZ(1,\alpha,\beta,\gamma,\alpha\beta,\alpha\gamma,\beta\gamma)$ in $\mf{O}$ cannot be divisible by any prime, and hence it equals $1$.
\end{proof}

\begin{remark}
    In \cref{thm: explicit p and k} we bound the norms of elements in a generating set, but with only very weak bounds on the number of elements in the generating set (\cref{prop:weak number bound}). By contrast, in \cref{prop:number of generators}, we bound the number of elements in a generating set, but have no bound on the norms of the elements. 
    One might ask for a result that balances the two: is it possible to find a generating set consisting of a very small number of elements, each of small norm? 
    
    The two results are similar in that they both rely on decomposing a theta function into cuspidal and non-cuspidal contributions, and showing that for suitable $n$, the coefficient of $q^n$ is dominated by the non-cuspidal contribution. To combine the two approaches, we would need a result analogous to \cref{thm:eigenform_coeffs}, giving an explicit bound on the coefficients of the cusp form components for \emph{cosets} of integral ideals within a maximal order, rather than for suborders. It is as yet unclear whether there is a generalized notion of Brandt module that can produce these inhomogeneous theta functions. 
\end{remark}

\section{Experimental results}\label{sec:experiment}

In this section we make computational observations demonstrating our main theorems and estimating the tightness of some results. All code and computations were done in \texttt{Magma V29.9} \cite{Magma}. A GitHub repository containing the code used for generating figures and running examples is available at \url{https://GitHub.com/wmahaney/generating_sets_for_maximal_quaternion_orders.git}.

\subsection{Estimating minimal exponent for prime power generation}

\cref{thm: explicit p and k} implies that the endomorphism ring of a supersingular elliptic curve $E/\overline{\FF}_p$ can be generated by elements of degree $\ell^k$ for relatively small values of $k$. In this section we test this explicitly and see if, in practice, the lower bound on allowable $k$ can be reduced.

First,
if $E_0$ is a uniformly randomly selected supersingular elliptic curve over $\Fpbar$ with~$p$ sufficiently large relative to $\ell$, and $k \gtrsim 2\log_{\ell}(p)$, we find that with high probability there is a finite index suborder of $\End(E_0)$ generated as a lattice by some elements of degree $\ell^{k_0}$ with $0 \le k_0 \le k$. A heuristic argument for this may be given using the bound $d \le 1 + 2 \log_{\ell} \left(  \frac{p-1}{\sqrt{6}} \right)$ on the diameter of $G(p, \ell)$ \cite[Corollary 32]{boscardin_spindler_diameter}, a strengthening of a classical bound due to Pizer \cite{PizerBAMS}, together with the expansion properties of $G(p,\ell)$.

Experimentally, we find something stronger for every example we tested:
\begin{experiment}\label{ex:experimental-bound}
    Let $k = 1+2 \lceil \log_\ell \left(   \frac{p-1}{\sqrt{6}}  \right)  \rceil $. For any maximal order~$\mf{O}$ in $\mathcal{A}_{p}$, the elements of norm $\ell^k$ generate $\mf{O}$ as a $\ZZ$-algebra.
\end{experiment}
\noindent The interested reader can check the GitHub repository to see certificates from some random tests and other output data\footnote{\tiny\url{https://github.com/wmahaney/generating_sets_for_maximal_quaternion_orders/blob/main/data}} and run tests themselves with the functions.\footnote{\tiny\url{https://github.com/wmahaney/generating_sets_for_maximal_quaternion_orders/blob/main/does_generate.magma}}

Given fixed $p$ and $\ell$, we can use \cref{thm: explicit p and k} to compute $k=k(p, \ell)$ such that a maximal order~$\mf{O} \subset \A_{p}$ is guaranteed to be generated as a $\ZZ$-algebra by the set of all elements in~$\mf{O}$ of norm $\ell^k$; the experimental results are consistently lower than this theoretical bound.
For instance, if $p \approx \ell^{30}$, then \cref{thm: explicit p and k} gives a lower bound of $k(p, \ell) \approx 75$ while 
\cref{ex:experimental-bound} gives a lower bound of $1+2 \log_{\ell}\left( \frac{p-1}{\sqrt{6}}  \right) \approx 61$.

\subsection{Estimating \texorpdfstring{$|a_{\ell^k}(\O) - g_{\ell^k}(\O)|$}{|a_l^k(O)-g_l^k(O)|}} 

Following the proof of \cref{prop: max submax upper lower} one can show that all maximal and submaximal orders $\O \subset \mathcal{A}_p$ with $p \ge 5$ satisfy
\[
|a_{\ell^k}(\O)-g_{\ell^k}(\O)| \le 2k\ell^{\tfrac{k+2}{2}}|\O^\times|,
\]
where $a_{\ell^k}(\O)$ is the $\ell^k$-coefficient of the theta series $\theta_{\O}$, and $g_{\ell^k}(\O)$ is the $\ell^k$-coefficient of the modular form $g$ defined in \cref{cor:max submax eigenform_coeffs}$(a)$. Recall that $g$ is the ``genus contribution,'' which in some sense measures the expected size of $a_{\ell^k}(\O)$ across all $\O$ in its genus. For $p\geq 5$ we have $|\O^\times|\leq 6$, so this is a stronger bound than \cref{prop: max submax upper lower}. 

In this section we compute data to estimate how tight the above bound is by calculating the ratio
\[
\mathscr{R}_{\ell^k}(\O) = \frac{|a_{\ell^k}(\O) - g_{\ell^k}(\O)|}{2k\ell^{\tfrac{k+2}{2}} |\O^\times|}.
\]
We will only present plots and discussion in the case $\ell = 2$, as the plots for $\ell = 3$ look similar; the interested reader can refer to the GitHub to see additional figures or generate their own. 

For each prime $p \in [7, 50000]$ and reduced discriminant $N \in \{ p, p^2, p^3, pr, pr^2 \}$ (we use $r=5$ here) corresponding to a maximal/submaximal order classification we do the following for each $k \in [1, 15]$:
\begin{enumerate}
    \item Randomly sample five orders $\mathcal{O}_i$ of the given classification and compute the ratios $\mathscr{R}_{2^k}(\O_i)$. 
    \item Compute the maximum $C_{2^k}(N)$ of 
    $\mathscr{R}_{2^k}(\O_1),\ldots, \mathscr{R}_{2^k}(\O_5)$.
    \item Plot $(k, p, C_{2^k}(N))$ (\cref{max-bound data plots}).
\end{enumerate}

 To randomize the order chosen in each trial we generated a quaternion order in Magma and then took a simple random walk of length $ \lceil \log_2(p) \rceil$ from the generated order. We ran five trials for each $p$ and discriminant classification,  using $k_{{\rm min}}=1, k_{{\rm max}} = 15$, and a value of $r=5$ when necessary. 
The full plots are presented below. %

Observationally, the ratio $\mathscr{R}_{2^k}(\O)$ never exceeded 0.35 except for very small values of $p$, and appears to tend to zero if (relatively large) $p$ is fixed and $k$ tends to infinity (and $r=5$ when $N$ is not a power of $p$), indicating that there may be a stronger bound than \cref{prop: max submax upper lower}.

\bibliographystyle{alphaurl}
\bibliography{bib.bib}

\newlength{\plotH}\setlength{\plotH}{0.26\textheight}

\begin{figure}[p]

\captionsetup[subfigure]{skip=3pt}

\hypertarget{max-bound_generation_plots}{\;}

\centering

\begin{subfigure}[t]{0.49\textwidth}\centering
  \includegraphics[width=\linewidth,height=\plotH,keepaspectratio]{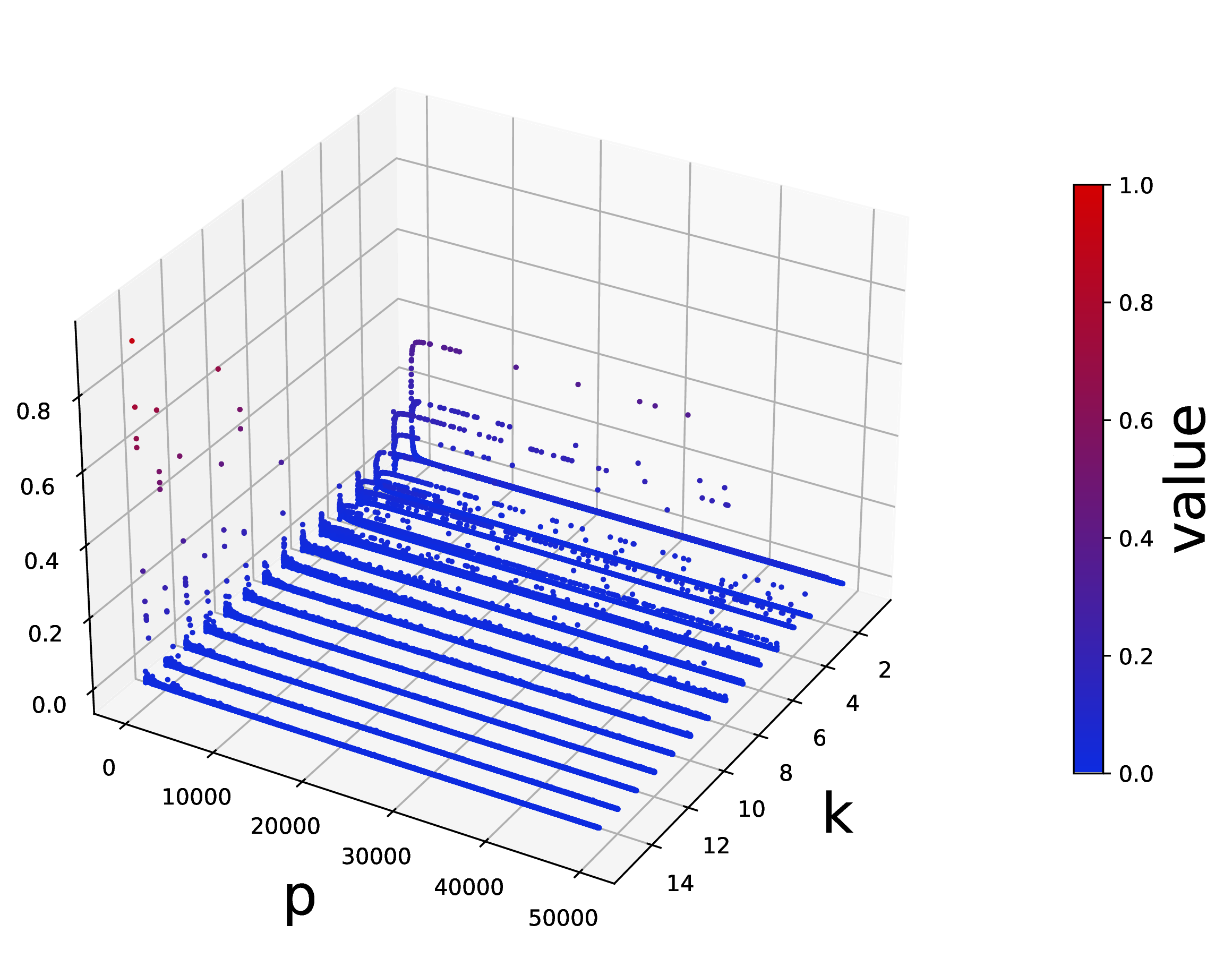}
\end{subfigure}

\vspace{\floatsep}
\begin{subfigure}[t]{0.49\textwidth}\centering
  \includegraphics[width=\linewidth,height=\plotH,keepaspectratio]{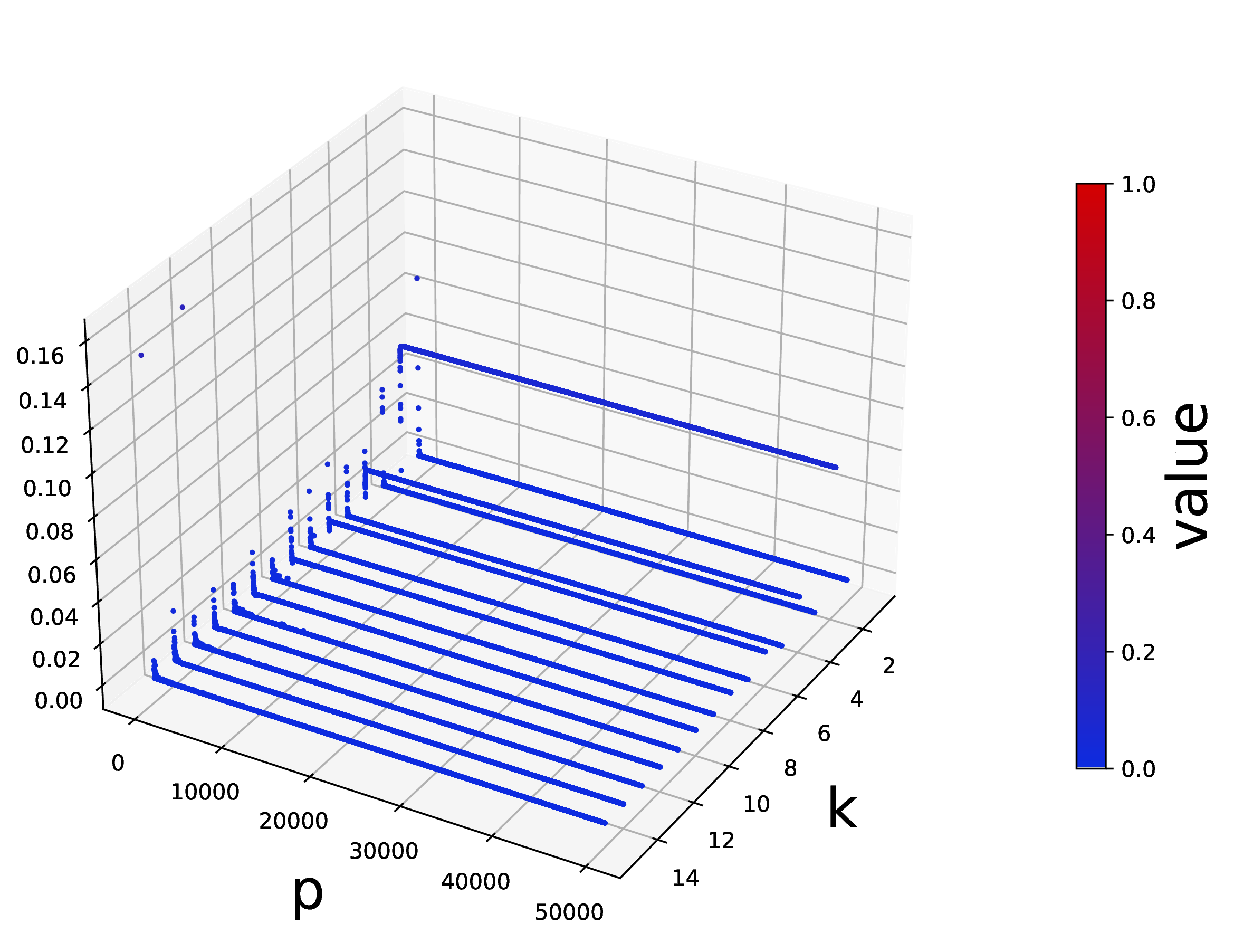}
\end{subfigure}\hfill
\begin{subfigure}[t]{0.49\textwidth}\centering
  \includegraphics[width=\linewidth,height=\plotH,keepaspectratio]{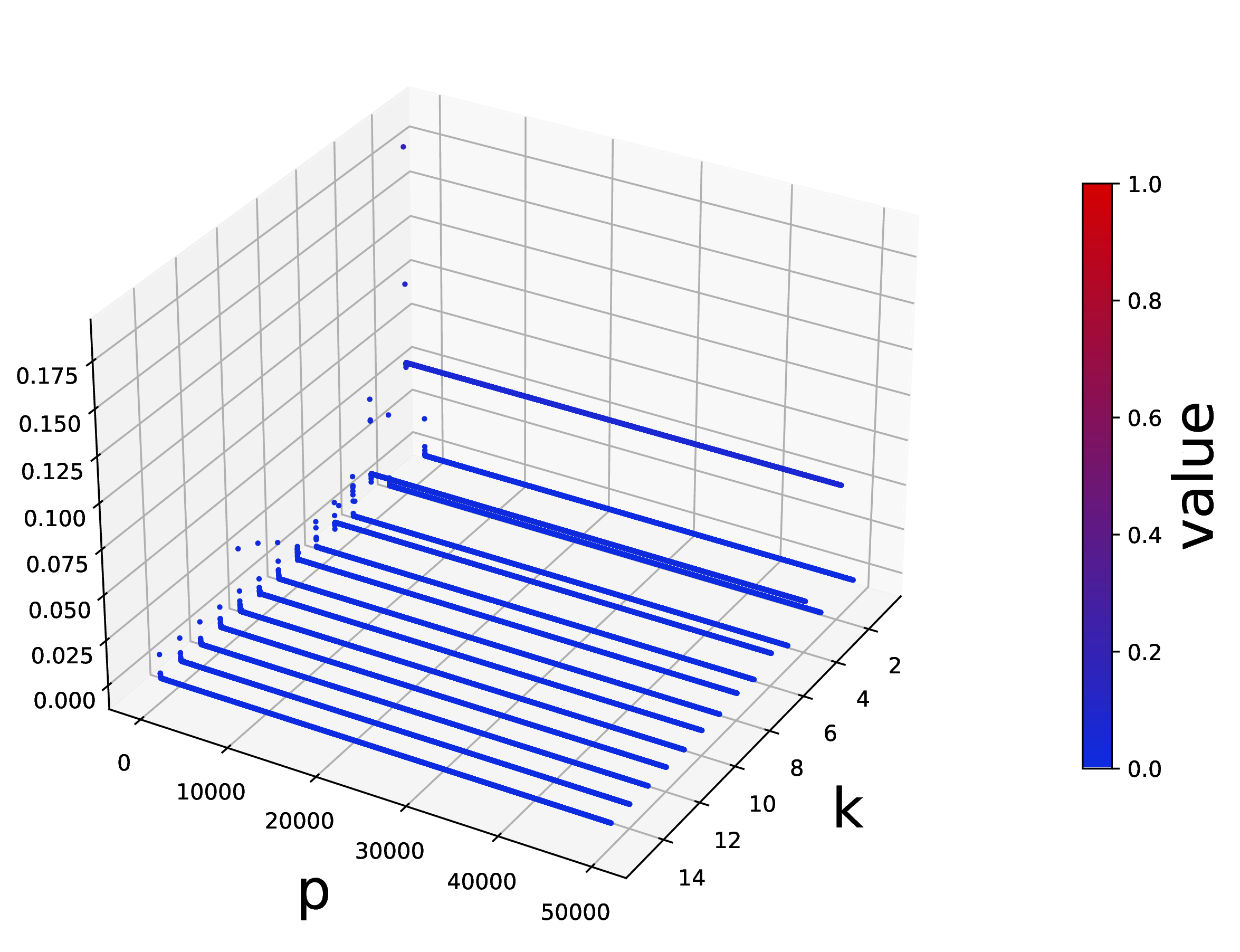}
\end{subfigure}
\vspace{\floatsep}
\begin{subfigure}[t]{0.49\textwidth}\centering
  \includegraphics[width=\linewidth,height=\plotH,keepaspectratio]{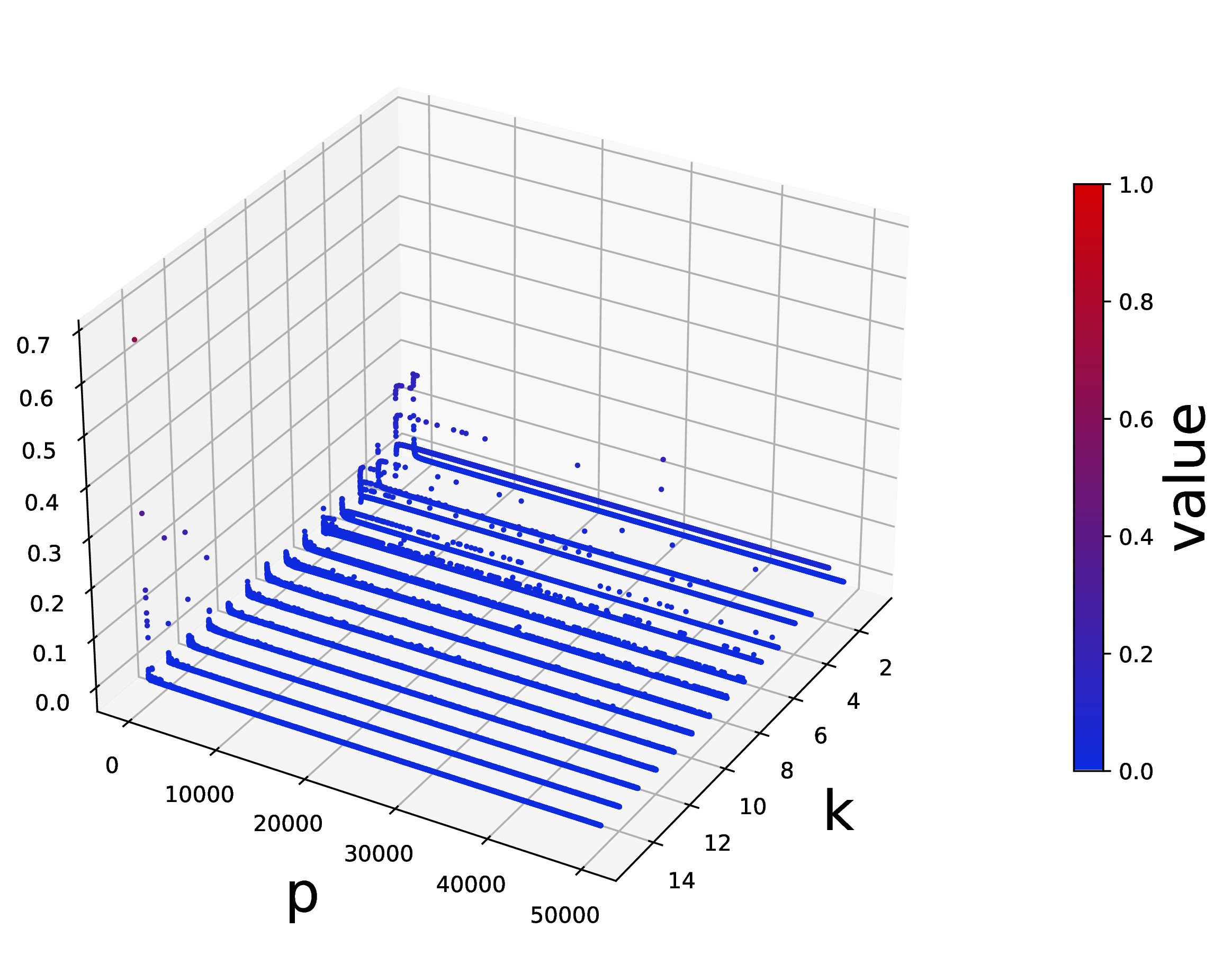}
\end{subfigure}\hfill
\begin{subfigure}[t]{0.49\textwidth}\centering
  \includegraphics[width=\linewidth,height=\plotH,keepaspectratio]{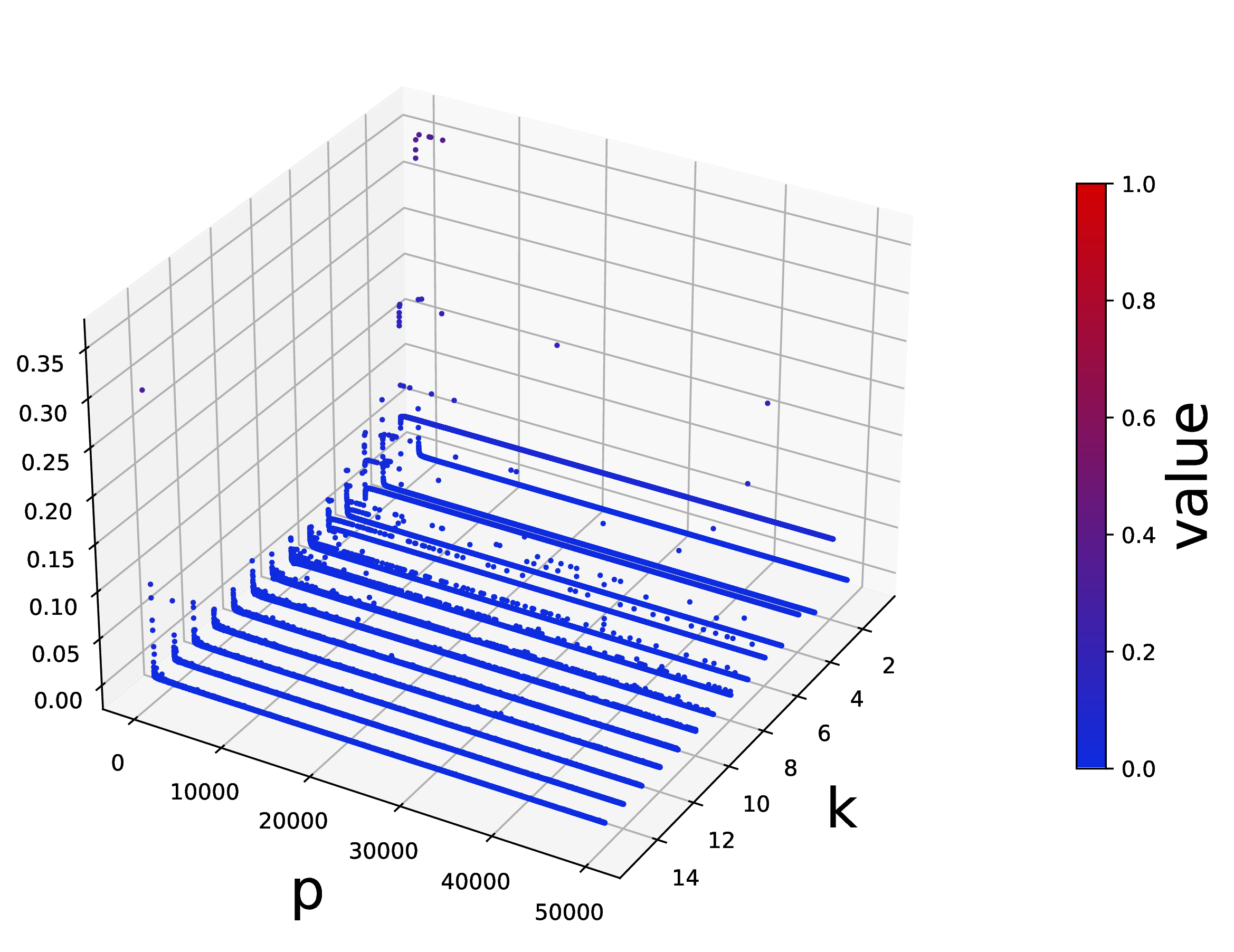}
\end{subfigure}
\caption{Maximal ratio of $|a_{\ell^k}(\O) - g_{\ell^k}(\O)|$ to $2k \ell^{\frac{k+2}{2}}| \O^\times|$ over $5$ trials, for $\ell = 2$ and $k=1, \ldots, 15$. Primes $p$ range over $[7, 50000]$. From top-to-bottom and left-to-right the order classification is maximal, submaximal of index $p$, submaximal of index $p^2$, submaximal of index $r$, and submaximal of index $r^2$.}
\label{max-bound data plots}

\end{figure}

\end{document}